\pdfoutput=1
\documentclass[a4paper,10pt]{article}
 \usepackage{fullpage}
 \usepackage{ae,lmodern}
\usepackage[english,greek,frenchb,french]{babel}
\usepackage[utf8]{inputenc}  
\usepackage[T1]{fontenc}
\usepackage{url,csquotes}
\usepackage[hidelinks,hyperfootnotes=false]{hyperref}
\usepackage{float}
\usepackage{amsfonts,amssymb,enumerate}
\usepackage[shortlabels]{enumitem}
\usepackage{amsmath}
\usepackage{authblk}
\allowdisplaybreaks[1]
\usepackage{amsthm}
\usepackage{graphicx}
\usepackage{bbm}
\usepackage{listings}
\usepackage{titling}
\usepackage{dsfont}
\usepackage{color}
\usepackage{bmpsize}
\newcommand{\E}{\mathbb{E}}
    \newcommand{\Prb}{\mathbb{P}}
	\newcommand{\cP}{\mathcal{P}}
		
				\newcommand{\cS}{\mathcal{S}}
				\newcommand{\cI}{\mathcal{I}}
				\newcommand{\cJ}{\mathcal{J}}

		\newcommand{\cE}{\mathcal{E}}
		
		\newcommand{\cF}{\mathcal{F}}
		\newcommand{\fE}{\mathfrak{E}}
				\newcommand{\fF}{\mathfrak{F}}
				\newcommand{\fT}{\mathfrak{T}}
				
		\newcommand{\cM}{\mathcal{M}}
		\newcommand{\cG}{\mathcal{G}}
		\newcommand{\cL}{\mathcal{L}}
				\newcommand{\cO}{\mathcal{O}}
		\newcommand{\cU}{\mathcal{U}}
	\newcommand{\cB}{\mathcal{B}}
		\newcommand{\B}{B(n)}
			\newcommand{\bR}{\mathrm{R}}
			\newcommand{\br}{\mathrm{r}}
		\newcommand{\cV}{\mathcal{V}}
				\newcommand{\cA}{\mathcal{A}}
					\newcommand{\cR}{\mathcal{R}}
							\newcommand{\cN}{\mathcal{N}}
							
		\newcommand{\cH}{\mathcal{H}}
	\newcommand{\sR}{\mathbb{R}}
	\newcommand{\sN}{\mathbb{N}}
	\newcommand{\vv}{v}

		\newcommand{\sS}{\mathbb{S}}
	\DeclareMathOperator{\diam}{diam}

				\DeclareMathOperator{\card}{card}

				\DeclareMathOperator{\cyl}{cyl}

\DeclareMathOperator{\disc}{disc}								\DeclareMathOperator{\diver}{div}
								\DeclareMathOperator{\capa}{Cap}
    \newcommand{\sZ}{\mathbb{Z}}
    \newcommand{\sC}{\mathcal{C}}

    \newcommand{\ep}{\varepsilon}
    \newcommand{\dis}{\mathfrak{d}}

    \newcommand{\ind}{\mathds{1}}
 \theoremstyle{plain}   
 \newtheorem{thm}{Theorem}[section]
\newtheorem{lem}[thm]{Lemma}
\newtheorem{cor}[thm]{Corollary}

\newtheorem{prop}[thm]{Proposition}
\newtheorem{rk}[thm]{Remark}
\newtheorem{hypo}{Hypothesis}
\numberwithin{equation}{section}

\title{Large deviation principle for the cutsets and lower large deviation principle for the maximal flow in first passage percolation
\thanks{Research was partially supported by the ANR project PPPP (ANR-16-CE40-0016) and the Labex MME-DII (ANR 11-LBX-0023-01).}}
\date{}
\author{Barbara Dembin\thanks{ETH Zürich. barbara.dembin@math.ethz.ch} , Marie Théret\thanks{Modal'X, UPL, Univ Paris Nanterre, F92000 Nanterre France and FP2M, CNRS FR 2036.
marie.theret@parisnanterre.fr} }

\begin{document}
\selectlanguage{english}
\maketitle
{\bf Abstract:} We consider the standard first passage percolation model in the rescaled lattice $\sZ^d/n$ for $d\geq 2$ and a bounded domain $\Omega$ in $\sR^d$. We denote by $\Gamma^1$ and $\Gamma^2$ two disjoint subsets of $\partial \Omega$ representing respectively the sources and the sinks, \textit{i.e.}, where the water can enter in $\Omega$ and escape from $\Omega$. A cutset is a set of edges that separates $\Gamma ^1$ from $\Gamma^2$ in $\Omega$, it has a capacity given by the sum of the capacities of its edges. Under some assumptions on $\Omega$ and the distribution of the capacities of the edges, we already know a law of large numbers for the sequence of minimal cutsets $(\cE_n^{min})_{n\geq 1}$: the sequence $(\cE_n^{min})_{n\geq 1}$ converges almost surely to the set of solutions of a continuous deterministic problem of minimal cutset in an anisotropic network. We aim here to derive a large deviation principle for cutsets and deduce by contraction principle a lower large deviation principle for the maximal flow in $\Omega$. 
\section{Introduction}
\subsection{First definitions and main results}
\subsubsection{The environment, minimal cutsets}\label{section:env}
We use here the same notations as in \cite{CT1}.
 Let $n\geq 1$ be an integer. We consider the graph $(\sZ^d_n,\E^d_n)$ having for vertices $\sZ^d _n=\sZ^d/n$ and for edges $\E_n^d$, the set of pairs of points of $\sZ_n^d$ at Euclidean distance $1/n$ from each other.  With each edge $e\in\E_n^d $ we associate a capacity $t(e)$, which is a random variable with value in $\sR^+$. The family $(t(e))_{e\in\E_n^d}$ is independent and identically distributed with a common law $G$. We interpret this capacity as a rate of flow, \textit{i.e.}, it corresponds to the maximal amount of water that can cross the edge per second. Throughout the paper, we work with a distribution $G$ on $\sR^+$ satisfying the following hypothesis.
\begin{hypo} \label{hypo:G} There exists $M>0$ such that $G([M,+\infty[)=0$ and $G(\{0\})<1-p_c(d)$.
\end{hypo}
\noindent Here $p_c(d)$ denotes the critical parameter of Bernoulli bond percolation on $\sZ^d$.

Let $\Omega$ be a bounded domain in $\sR^d$. Let $\Gamma^1$ and $\Gamma ^2$ be two disjoint subsets of the boundary $\partial \Omega$ of $\Omega$ that represent respectively the sources and the sinks. We aim to study the minimal cutsets that separate  $\Gamma^1$ from $\Gamma^2$ in $\Omega$ for the capacities $(t(e))_{e\in\E_n^d}$. We shall define discretized versions for those sets. For $x=(x_1,\dots,x_d)\in\sR^d$, we define $$\|x\|_2=\sqrt{\sum_{i=1}^dx_i^2},\qquad\|x\|_1=\sum_{i=1}^d|x_i|\qquad\text{and}\qquad \|x\|_\infty =\max\big\{\,|x_i|,\,i=1,\dots,d\,\big\}.$$
We use the subscript $n$ to emphasize the dependence on the lattice $(\sZ^d_n,\E^d _n)$. Let $\Omega_n$, $\Gamma_n$, $\Gamma_n^1$ and $\Gamma_n^2$ be the respective discretized version of $\Omega$, $\Gamma$, $\Gamma^1$ and $\Gamma^2$:
\begin{align*}
&\Omega_n=\left\{\,x\in \sZ^d_n:\, d_\infty(x,\Omega)<\frac{1}{n}\right\},\\
&\Gamma_n=\big\{\,x\in \Omega_n:\, \exists y\notin \Omega_n,\,\langle x,y\rangle\in\E_n^d\,\big\}\,,\\
&\Gamma_n ^i=\left\{\,x\in \Gamma_n:\, d_\infty(x,\Gamma^i)<\frac{1}{n}, \,  d_\infty(x,\Gamma^{3-i})\geq\frac{1}{n}\right\},\qquad \text{for $i=1,2$},
\end{align*}
where $d_\infty$ is the $L^\infty$ distance associated with the norm $\|\cdot\|_\infty$ and $\langle x,y\rangle$ represents the edge whose endpoints are $x$ and $y$. 
  We denote by $\Pi_n$ the set of edges that have both endpoints in $\Omega_n$, \textit{i.e.},
  $$\Pi_n=\left\{\,e=\langle x,y\rangle\in\E_n^d: \,x,y\in \Omega_n\,\right\}\,.$$
 Throughout the paper, $\Omega,\Gamma^1,\Gamma^2$ satisfy the following hypothesis:
\begin{hypo}\label{hypo:omega}
The set $\Omega$ is an open bounded connected subset of $\sR^d$, it is a Lipschitz domain and its boundary $\Gamma=\partial \Omega$ is included in a finite number of oriented hypersurfaces of class $\sC^1$ that intersect each other transversally. The sets $\Gamma^1$ and $\Gamma^2$ are two disjoint subsets of $\Gamma$ that are open in $\Gamma$ such that $\inf\{\|x-y\|_2,\, x\in\Gamma^1,\,y\in\Gamma^2\}>0$, and their relative boundaries $\partial_\Gamma \Gamma^1$ and $\partial_\Gamma \Gamma ^2$ have null $\cH^{d-1}$ measure, where $\cH^{d-1}$ denotes the $(d-1)$-dimensional Hausdorff measure.
\end{hypo}
 \noindent {\bf $(\Gamma_n ^1,\Gamma_n ^2)$-cutset in  $\Omega_n$.} A set of edges $\cE_n\subset\Pi_n$ is a $(\Gamma_n ^1,\Gamma_n ^2)$-cutset in  $\Omega_n$ if for any path $\gamma$ from $\Gamma_n^1$ to $\Gamma_n^2$ in $\Omega_n$, $\gamma\cap \cE_n\neq\emptyset$. We denote by $\sC_n(\Gamma^1,\Gamma^2,\Omega)$ the set of $(\Gamma_n ^1,\Gamma_n ^2)$-cutset in $\Omega_n$. For $\cE_n\in\sC_n(\Gamma^1,\Gamma^2,\Omega)$, we define its capacity $V(\cE_n)$ and its associated measure $\mu_n(\cE_n)$ by
$$V(\cE_n)=\sum_{e\in\cE_n}t(e)\,$$
and 
 \begin{align}
 \mu_n(\cE_n)=\frac{1}{n^{d-1}}\sum_{e\in\cE_n}t(e)\delta_{c(e)}
 \end{align}
where $c(e)$ denotes the center of the edge $e$ and $\delta_{c(e)}$ the dirac mass at $c(e)$.
The set $\cE_n$ is a discrete set, but in the limit it is more convenient to work with a continuous set. 
We first define $\br(\cE_n)\subset \sZ_n^d$ by
$$\br(\cE_n)=\left\{\,x\in \sZ_n^d: \text{ there exists a path from $x$ to $\Gamma_n^1$ in $(\sZ_n^d,\Pi_n\setminus \cE_n)$}\right\}\,.$$ 
We define upon $\br(\cE_n)$ a continuous version $\bR(\cE_n)$ by setting  
$$\bR(\cE_n)=\br(\cE_n)+\frac{1}{2n} [-1,1]^d\,.$$
Hence, we have $\bR(\cE_n)\cap\sZ_n^d=\br(\cE_n)$.

 \noindent {\bf Minimal cutsets.}  A set $\cE_n\in \sC_n(\Gamma^1,\Gamma^2,\Omega)$ is a minimal cutset in $\Omega_n$ if we have $$V(\cE_n)=\inf\left\{V(\cF_n):\cF_n\in \sC_n(\Gamma^1,\Gamma^2,\Omega)\right\}\,.$$ We denote by $\sC_n(0)$ the set of minimal cutsets in $\Omega_n$. We will often use the notation $\cE_n^{min}$ to denote an element of $\sC_n(0)$ chosen according to a deterministic rule. Minimal cutsets are the analogous in dimension $d-1$ of geodesics in the classical interpretation of first passage percolation. Indeed, geodesics minimize the sum of times along paths that are one-dimensional objects, whereas minimal cutsets minimize the sum of capacities along surfaces that are $(d-1)$-dimensional objects.

\noindent 
{\bf Almost minimal cutsets.} Let $\ep>0$. A set $\cE_n\in \sC_n(\Gamma^1,\Gamma^2,\Omega)$ is a $(\Gamma_n ^1,\Gamma_n ^2)$ $\ep$-cutset in $\Omega_n$ if for any $\cF_n\in \sC_n(\Gamma^1,\Gamma^2,\Omega)$, we have $$V(\cE_n)\leq V(\cF_n)+\ep n^{d-1}\,.$$ We denote by $\sC_n(\ep)$ the set of $(\Gamma_n ^1,\Gamma_n ^2)$ $\ep$-cutset in $\Omega_n$.
Note that the typical size of an almost minimal cutset is of surface order, that is of order $n^{d-1}$ (see for instance lemma \ref{lem:Zhang}).

\noindent {\bf Maximal flow.} We define $\phi_n$ the maximal flow between $\Gamma_n^1$ and $\Gamma_n^2$ in $\Omega_n$ as follows
$$\phi_n(\Gamma^1,\Gamma^2,\Omega)=\inf\left\{V(\cE_n): \,\cE_n\in\sC_n(\Gamma^1,\Gamma^2,\Omega)\right\}\,.$$
The reason why this quantity is called a maximal flow is due to the max-flow min-cut theorem that states that the study of the minimal capacity of a cutset is the dual problem of the study of the maximal flow. Just as the study of minimal capacity is linked with the study of cutsets, the study of maximal flow is linked with the study of streams (a stream is a function that describes a stationary circulation of water in the lattice). We won't define rigorously what a stream is and its link with maximal flow. We refer for instance to the companion paper \cite{dembin2020large} where we study large deviation principle for admissible streams to obtain an upper large deviation principle for the maximal flow.


\subsubsection{Presentation of the limiting objects and main results}
We want to define the possible limiting objects for $\bR(\cE_n)$ and $\mu_n(\cE_n)$ where $\cE_n\in\sC_n(\ep)$.

\noindent {\bf Continuous cutsets.} 
We denote by $\sC_{<\infty}$ the set of subsets of $\Omega$ having finite perimeter in $\Omega$, \textit{i.e.}, 
\[\sC_{<\infty}=\left\{E\text{ Borelian subset of $\sR^d$}: \, E\subset\Omega, \,\cP(E,\Omega)<\infty\right\}\,.\]
 When $E$ is regular enough, its perimeter in the open set $\Omega$ corresponds to $\cH^{d-1}(\partial E\cap \Omega)$. We will give a more rigorous definition of the perimeter later (see \eqref{eq:defper}).
Let $E\in \sC_{<\infty}$.
We want to build from $\partial E$ a continuous surface that would be a continuous cutset between $\Gamma^1$ and $\Gamma^2$ (we don't give a formal definition of what a continuous cutset is). However, a continuous path from $\Gamma^1$ to $\Gamma^2$ does not have to intersect $\partial E$ in general. For regular sets $E$ and $\Omega$, such a path should intersect 
\[\widehat \fE= (\partial E \cap \Omega )\cup ( \Gamma^1\setminus \partial E)\cup(\partial E\cap\Gamma^2)\]
We define $\fE$ as a more regular version of $\widehat \fE$ (see figure \ref{figintro}),
\begin{align}\label{def:fE}
\fE= (\partial^* E \cap \Omega )\cup (\partial ^* \Omega \cap(( \Gamma^1\setminus \overline{\partial^*E})\cup(\partial^ * E\cap\Gamma^2))\,,
\end{align}
where $\overline X$ denotes the closure of the set $X$ and $\partial^*$ is the reduced boundary, we will give a rigorous definition later (see section \ref{sect:tools}).
\begin{figure}[H]
\begin{center}
\def\svgwidth{0.7\textwidth}
   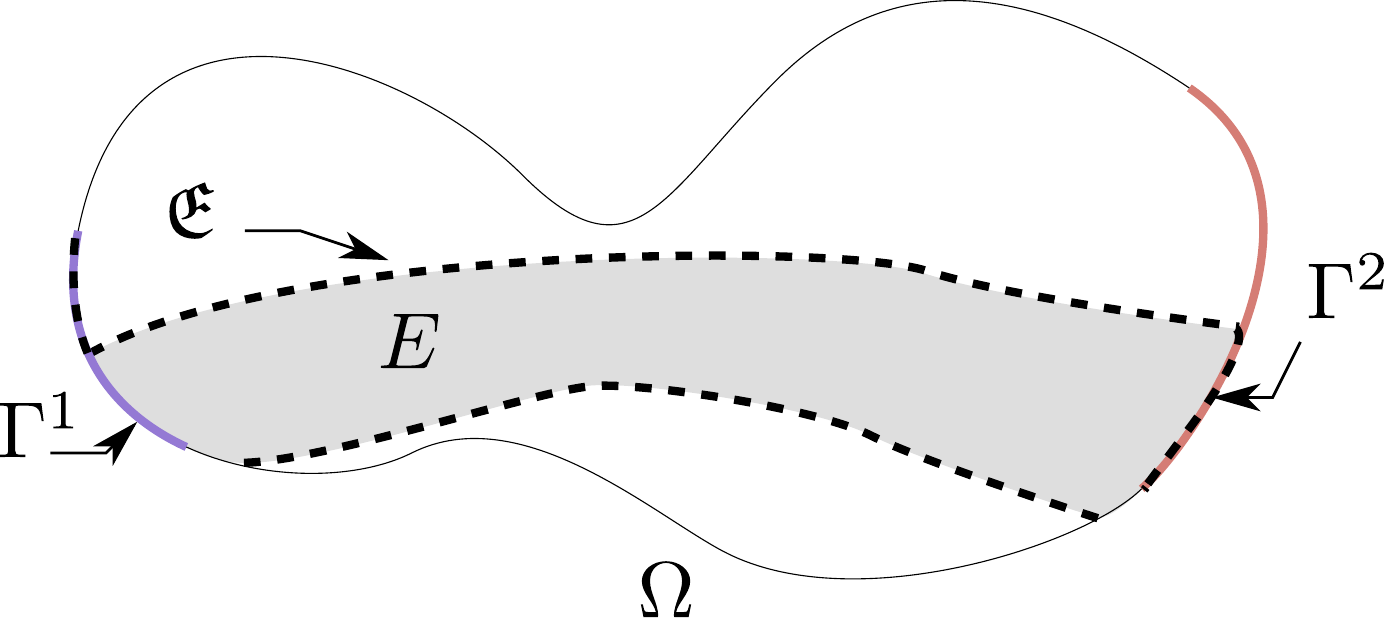
   \caption{\label{figintro}Representation of the set $\fE$ (the dotted lines) for some $E\in\sC_{<\infty}$.}
   \end{center}
\end{figure}

 Note that in order to obtain a lower large deviation principle for the minimal cutsets, it is not enough to keep track of the localization of the cutset. Indeed, if we only know the localization of the discrete minimal cutset it does not give information on its capacity. We will not only need the total capacity of a cutset but also its local distribution. For a given localization of the minimal cutset and a given capacity, there are different macroscopic configurations where there exists a discrete cutset at the given localization with the given capacity. However this different macroscopic configurations do not have necessarily the same cost, \textit{i.e.}, the same probability to be observed. This is why we need to know locally the capacity. In the continuous setting, this boils down to introducing a function $f:\fE\rightarrow \sR_+$. The local capacity at a point $x$ in $\fE$ will be given by $f(x)$. Naturally, the total capacity will be obtained by summing these contributions.
We define 
$\capa(E,f)$ as the capacity of the cutset $\fE$ equiped with a local capacity $f(x)$ at $x$:
$$\capa(E,f)=\int_{\fE}f(x)d\cH^{d-1}(x)\,.$$
The smaller $f(x)$ is the bigger the cost of having this local capacity is. We will see that there is a local capacity that is in some sense costless. For $x\in\fE$ and $n(x)$ the associated normal exterior unit vector (of $E$ or $\Omega$ depending whether $x\in \partial^* E$ or $x\in\partial ^* \Omega$), the typical local capacity is $\nu_G(n(x))$ where $\nu_G$ is a called the flow constant. In other words, the probability that the local capacity is close to $\nu_G(n(x))$ at $x$ is almost $1$.

We say that $(E,f)$ is minimal if for any $F\subset \Omega$ of finite perimeter we have
$$\capa(E,f)\leq \int_{\overline{\fF}\cap \fE}f(y)d\cH^{d-1}(y)+\int_{(\fF\setminus \overline{\fE})\cap \partial \Omega}\nu_G(n_{\Omega}(y))d\cH^{d-1}(y)+ \int_{(\fF\setminus \overline{\fE})\setminus \partial \Omega}\nu_G(n_{F}(y))d\cH^{d-1}(y)\,$$
where $\fF$ is defined for $F$ in \eqref{def:fE}.
This condition will be useful in what follows. This condition is very natural, if we start with a minimal cutset that is close to $\fE$ and with local capacity $f$, we expect that the limiting object inherits this property. On $\fE$ the local capacity is $f$ but everywhere else the local capacity is the typical one given by the flow constant $\nu_G$. The right hand side may be interpreted as the capacity of $\fF$ in the environment where the local capacity on $\fE$ is $f$.
Let $\fT$ be the following set:
$$\fT= \left\{\, (E,f)\in \sC_{<\infty}\times L^\infty(\fE\rightarrow \sR,\cH^{d-1}): \begin{array}{c}(E,f) \text{ is minimal, } \capa(E,f)\leq 10d^2M \cH^{d-1}(\Gamma^1)\,,\\f(x)\leq \nu_G(n_E(x)) \text{ $\cH^{d-1}$-a.e. on $\partial ^* E\cap \Omega$},\\f(x)\leq \nu_G(n_\Omega(x)) \text{ $\cH^{d-1}$-a.e. on $\partial ^* \Omega$}
\end{array}\right\}\,.$$
We also define $\fT_{\cM}$ as follow
$$\fT_{\cM}=\left\{\,(E,f\cH^{d-1}|_{\fE}):\,(E,f)\in\fT\,\right\}\,.$$
For two Borelian sets $E$ and $F$, we define $\dis(E,F)=\cL^d(E\Delta F)$ where $\Delta$ denotes the symmetric difference and $\cL^d$ is the $d$-dimensional Lebesgue measure.
\begin{rk}\label{rk:1}Let $\cE_n\in\sC_n(\Gamma^1,\Gamma^2,\Omega)$. If $\lim_{n\rightarrow\infty}\dis(\bR(\cE_n),E)=0$ for some $E\in\cB(\sR^d)$ and $\mu_n(\cE_n)$ weakly converges towards $\mu$, 
then we do not necessarily have that $\mu$ is absolutely continuous with respect to $\cH^{d-1}|_\fE$. More generally, for any sequence $(E_n)_{n\geq 1}$ of Borelian sets of $\sR^d$ such that $\lim_{n\rightarrow\infty}\dis(E_n,E)=0$, we do not have necessarily that $\cH^{d-1}|_{\partial E_n}$ weakly converges towards $\cH^{d-1}|_{\partial E}$. However, we can prove that if instead of studying any sequence of cutsets we study a sequence of minimal cutsets, then in the limit $\mu$ will be supported on $\fE$. But, since it is too difficult to ensure we build a configuration of the capacities of the edges with a minimal cutset at a given localization (it is difficult to ensure that the cutset we have built is indeed minimal), we will work instead with almost minimal cutsets, \textit{i.e.}, $\ep$-cutsets. They are more flexible than minimal cutsets and their continuous limit is in the set $\fT_\cM$.
\end{rk}
We denote by $\sS^{d-1}$ the unit sphere in $\sR^d$.
Let $(E,f)\in \fT$. For any $v\in\sS^{d-1}$, we denote by $\cJ_v$ the rate function associated with the lower large deviation principle for the maximal flow in a cylinder oriented in the direction $v$ (see theorem \ref{thm:lldtau}). We can interpret $\cJ_v(\lambda)$ as the cost of having a local capacity $\lambda$ which is abnormally small in the direction $v$ ($\lambda<\nu_G(v)$). It will be properly defined later in theorem \ref{thm:lldtau}.
We define the following rate function:
$$\cI(E,f)=\int_{\partial ^* E\cap \Omega}\cJ_{n_E(x)}(f(x))d\cH^{d-1}(x)+\int_{\partial ^*\Omega \cap ((\Gamma^1\setminus \overline{\partial^*E} )\cup (\Gamma^2\cap \partial^*E))}\cJ_{n_\Omega(x)}(f(x))d\cH^{d-1}(x)\,\,.$$
Roughly speaking $\cI(E,f)$ is the total cost of having a cutset $\fE$ with the local capacities given by $f$, the overall cost is equal to the sum of the contributions of the local costs over the continuous cutset $\fE$.

We denote by $\cM(\sR^d)$ the set of positive measures on $\sR^d$.
We endow $\cM(\sR^d)$ with the weak topology $\mathcal{O}$. We denote by $\cB$ the $\sigma$-field generated by $\mathcal{O}$. We endow the set $\cB(\sR^d)$ of Borelian sets of $\sR^d$ with the topology $\cO'$ associated with the distance $\dis$.
 We denote by $\cB'$ the $\sigma$-field associated with this distance. Let $n\geq 1$ and $\ep>0$, $\Prb_n^\ep$ denotes the following probability:
$$\forall A\in \cB'\otimes \cB\qquad\Prb_n^\ep(A)=\Prb(\exists \cE_n\in\sC_n(\ep):(\bR(\cE_n),\mu_n(\cE_n))\in A)\,.$$
We define the following rate function $\widetilde{I}$ on $\cB(\sR^d)\times \cM(\sR^d)$ as follows:
\[\forall (E,\nu)\in \cB(\sR^d)\times \cM(\sR^d)\qquad \widetilde{I}(E,\nu) = \left\{
    \begin{array}{ll}
        +\infty  & \mbox{if } (E,\nu)\notin\fT_{\cM}\\
        \cI(E,f) & \mbox{if } \nu=f\cH^{d-1}|_{\fE} \text{ with }\,(E,f)\in\fT\,.
    \end{array}
\right.
\]
The following theorem is the main result of this paper. 
\begin{thm}[Large deviation principle on cutsets]\label{thm:pgd} Let $G$ that satisfies hypothesis \ref{hypo:G}. Let $(\Omega,\Gamma^1,\Gamma^2)$ that satisfies hypothesis \ref{hypo:omega}. The sequence $(\Prb_n^\ep)_{n\geq 1}$ satisfies a large deviation principle with speed $n^{d-1}$ governed by the good rate function $\widetilde{I}$ and with respect to the topology $\mathcal{O}'\otimes\cO$ in the following sense: for all $A\in\cB'\otimes\cB$
\begin{align*} -\inf\big\{\widetilde{I}(\nu ):\,\nu\in \mathring{A}\big\}\leq\lim_{\ep\rightarrow 0}\liminf_{n\rightarrow\infty}\frac{1}{n^{d-1}}\log\Prb_n^\ep(A)\leq \lim_{\ep\rightarrow 0}\limsup_{n\rightarrow\infty}\frac{1}{n^{d-1}}\log\Prb_n^\ep(A)\leq -\inf\big\{\widetilde{I}(\nu ):\,\nu\in \overline{A}\big\}\,.\end{align*}
\end{thm}
\begin{rk}Because of the limit in $\ep$, it is not a proper large deviation principle. Unfortunately, we were not able to obtain a large deviation principle on $\cE_n^{min}$ because of the lower bound (see remark \ref{rk:1}). However, this result is enough to prove a lower large deviation principle on the maximal flow.
\end{rk}
We can deduce from theorem \ref{thm:pgd}, by a contraction principle, the existence of a rate function governing the lower large deviations of $\phi_n(\Gamma^1,\Gamma^2,\Omega)$.
Let $J$ be the following function defined on $\sR^+$:
$$\forall \lambda\geq 0\qquad J(\lambda)=\inf\left\{\widetilde{I}(E,\nu) : \,(E,\nu)\in\cB(\sR^d)\times\cM(\sR^d),\, \nu(\sR^d)=\lambda\right\}\,$$
and define $\lambda_{min}$ as
$$\lambda_{min}=\inf\left\{ \lambda\geq 0: J(\lambda)<\infty\right\}\,.$$
The real number $\phi_\Omega>0$ will be defined in \eqref{eq:def:phiomega}, it is the almost sure limit of $\phi_n(\Gamma^1,\Gamma^ 2,\Omega)/n^{d-1}$.
We have the following lower large deviation principle for the maximal flow.
\begin{thm}[Lower large deviation principle for the maximal flow]\label{thm:lldmf} 
Let $G$ that satisfies hypothesis \ref{hypo:G}. Let $(\Omega,\Gamma^1,\Gamma^2)$ that satisfies hypothesis \ref{hypo:omega}. The sequence $(\phi_n(\Gamma^1,\Gamma^ 2,\Omega)/n^{d-1},n\in\sN)$ satisfies a large deviation principle of speed $n^{d-1}$ governed by the good rate function $J$. Moreover, the map $J$ is finite on $]\lambda_{min},\phi_\Omega]$, infinite on $[0,\lambda_{min}[\cup]\phi_\Omega,+\infty[$,
and we have
\begin{align}\label{eq:imppgd}
\forall \lambda <\lambda_{min}\quad \exists N\geq 1\quad \forall n\geq N \qquad \Prb(\phi_n(\Gamma^1,\Gamma^2,\Omega)\leq\lambda n ^{d-1})=0\,.
\end{align}
\end{thm}
\begin{rk}The property \eqref{eq:imppgd} is used to prove upper large deviation principle for the maximal flow in the companion paper \cite{dembin2020large}. It also gives the precise role of $\lambda_{min}$ in our study. 
\end{rk}
\subsubsection{Upper large deviations for the maximal flows}
In the companion paper \cite{dembin2020large}, we proved the upper large deviation principle for the maximal flow.

\begin{thm}[Upper large deviation principle for the maximal flow]\label{thm:uldmf} 
Let $G$ that satisfies hypothesis \ref{hypo:G}. Let $(\Omega,\Gamma^1,\Gamma^2)$ that satisfies hypothesis \ref{hypo:omega}. The sequence $(\phi_n(\Gamma^1,\Gamma^ 2,\Omega)/n^{d-1},n\in\sN)$ satisfies a large deviation principle of speed $n^d$ with the good rate function $\widetilde J_u.$ Moreover, there exists $\lambda_{max}>0$ such that the map $\widetilde J_u$ is convex on $\sR_+$, infinite on $[0,\lambda_{min}[\cup]\lambda_{max},+\infty[$, $\widetilde J_u $ is null on $[\lambda_{min},\phi_\Omega]$ and strictly positive on $]\phi_\Omega,+\infty[$. 
\end{thm}
We refer to \cite{dembin2020large} for a precise definition of $\widetilde J_u $.
Theorems \ref{thm:lldmf} and \ref{thm:uldmf} give the full picture of large deviations of $\phi_n(\Gamma^1,\Gamma^ 2,\Omega)$. The lower large deviations are of surface order since it is enough to decrease the capacities of the edges along a surface to obtain a lower large deviations event. The upper large deviations are of volume order, to create an upper large deviations event, we need to increase the capacities of constant fraction of the edges. This is the reason why to study lower large deviations, it is natural to study cutsets that are $(d-1)$-dimensional objects, whereas to study the upper large deviations, we study streams (functions on the edges that describe how the water flows in the lattice) that are $d$-dimensional objects. Actually, theorem \ref{thm:lldmf} is used in \cite{dembin2020large} to prove theorem \ref{thm:uldmf}. Theorem \ref{thm:lldmf} justifies the fact that the lower large deviations are not of the same order as the order of the upper large deviations.

\subsection{Background} We now present the mathematical background needed in what follows. We present a maximal flow in cylinders with good subadditive properties and give a rigorous definition of the limiting objects.

\subsubsection{Probabilistic background}\label{sect:flowconstant}
Let $A$ be a non-degenerate hyperrectangle, \textit{i.e.}, a rectangle of dimension $d-1$ in $\sR^d$. Let $\vv\in\sS^{d-1}$ such that $\vv$ is not contained in an hyperplane parallel to $A$. 
Let $h>0$.
If $\vv$ is one of the two unit vectors normal to $A$, we denote by $\cyl(A,h)$ the following cylinder of height $2h$: 
$$\cyl(A,h)=\big\{\,x+t\vv:\,x\in A,\,t\in[-h,h]\,\big\}\,.$$
We have to define discretized versions of the bottom half $B'(A,h)$ and the top half $T'(A,h)$ of the boundary of the cylinder $\cyl(A,h)$,
%
that is if we denote by $z$ the center of $A$:
\begin{align}\label{eq:defT(A,h)}
T'(A,h)=\big\{\,x\in\sZ^d_n\cap \cyl(A,h)\,:\,(z-x)\cdot\vv> 0 \text{ and }\exists y\notin \cyl(A,h),\,\langle x,y\rangle\in\E^d_n\,\big\}\,,
\end{align}
\begin{align}\label{eq:defB(A,h)}
B'(A,h)=\big\{\,x\in\sZ^d_n\cap \cyl(A,h)\,:\,(z-x)\cdot\vv<0 \text{ and }\exists y\notin \cyl(A,h),\,\langle x,y\rangle\in\E^d_n\,\big\}\,,
\end{align}
where $\cdot$ denotes the standard scalar product in $\sR^d$.
We denote by $\tau_n(A,h)$ the maximal flow from the upper half part to the lower half part of the boundary of the cylinder, \textit{i.e.}, 
$$\tau_n(A,h)=\inf\left\{V(E): \, E\text{ cuts $T'(A,h)$ from $B'(A,h)$ in $\cyl(A,h)$}\right\}\,.$$
The random variable $\tau_n(A,h)$ has good subadditivity properties since minimal cutsets in adjacent cylinders can be glued together along the common side of these cylinders by adding a negligible amount of edges. Therefore, by applying ergodic subadditive theorems in the multi-parameter case, we can obtain the convergence of $\tau_n(A,h)/n^{d-1}$.
\begin{thm}[Rossignol-Théret \cite{RossignolTheretcont}]\label{thm:defnulgn}
 Let $G$ be a measure on $\sR^+$ such that $G(\{0\})<1-p_c(d)$. For any $v\in\sS^{d-1}$, there exists a constant $\nu_G(v)>0$ such that for any non-degenerate hyperrectangle $A$ normal to $\vv$, for any $h>0$, we have
 $$\lim_{n\rightarrow \infty}\frac{\tau_n(A,h)}{\cH^{d-1}(A)n^{d-1}}=\nu_G(v)\text{ a.s.}.$$
\end{thm}
The function $\nu_G$ is called the flow constant. This is the analogue of the time constant defined upon the geodesics in the standard first passage percolation model where the random variables $(t(e))_e$ represent passage times.
\begin{rk} If $G(\{0\})\geq 1-p_c(d)$, the flow constant is null (see Zhang \cite{Zhang}). It follows that $\phi_\Omega=0$ and there is no interest of studying lower large deviations.
\end{rk}

We present here some result on upper large deviations for the random variable $\tau_n(A,h)$ that will be useful in what follows. It states that the upper large deviation for the random variable $\tau_n(A,h)$ are of volume order.
\begin{thm}[Théret \cite{TheretUppertau14}] \label{thm:theretuppertau} Let $\vv$ be a unit vector, $A$ be an hyperrectangle orthogonal to $\vv$ and $h>0$.
Let us assume that $G$ satisfies hypothesis \ref{hypo:G}. For every $\lambda>\nu_G(v)$, we have
\[\liminf_{n\rightarrow\infty} -\frac{1}{\cH^{d-1}(A)n^d}\log\Prb\left(\frac{\tau_n(A,h)}{\cH^{d-1}(A)n^{d-1}}\geq \lambda\right) >0\,.\]
\end{thm}
\subsubsection{Some mathematical tools and definitions}\label{sect:tools}
Let us first recall some mathematical definitions.
For a subset $X$ of $\sR^d$, we denote by $\overline{X}$ the closure of $X$, by $\mathring{X}$ the interior of $X$. Let $a\in\sR^d$, the set $a+X$ corresponds to the following subset of $\sR^d$
$$a+X=\{a+x:\,x\in X\}\,.$$
For $r>0$, the $r$-neighborhood $\cV_i(X,r)$ of $X$ for the distance $d_i$, that can be Euclidean distance if $i=2$ or the $L^\infty$-distance if $i=\infty$, is defined by
$$\cV_i(X,r)=\left\{\,y\in\sR^d:\,d_i(y,X)<r\,\right\}\,.$$
We denote by $B(x,r)$ the closed ball centered at $x\in\sR^d$ of radius $r>0$. Let $v\in\sS^{d-1}$. We define the upper half ball $B^+(x,r,v)$ and the lower half ball $B^-(x,r,v)$ as follows:
$$B^+(x,r,v)=\left\{y\in B(x,r): (y-x)\cdot v\geq 0\right\}\qquad\text{and}\qquad B^-(x,r,v)=\left\{y\in B(x,r): (y-x)\cdot v< 0\right\}\,.$$
We also define the disc $\disc(x,r,v)$ centered at $x$ of radius $r$ and of normal vector $v$ as follows:
\[\disc(x,r,v)=\left\{y\in B(x,r): (y-x)\cdot v =0\right\}\,.\]
For $n\geq 1$, we define the discrete interior upper boundary $\partial ^+_n B(x,r,v)$ and the discrete interior lower boundary $\partial ^-_n B(x,r,v)$ as follows:
\[\partial ^+_n B(x,r,v)=\left\{y\in B(x,r)\cap\sZ_n^d:\exists z\in\sZ_n^d\setminus B(x,r), (z-x)\cdot v\geq 0\text{ and } \|z-y\|_1=\frac{1}{n}\right\}\,\]
and
\[\partial ^-_n B(x,r,v)=\left\{y\in B(x,r)\cap\sZ_n^d:\exists z\in\sZ_n^d\setminus B(x,r), (z-x)\cdot v<0\text{ and } \|z-y\|_1=\frac{1}{n}\right\}\,.\]
For $U\subset \sZ_n^d$, we define its edge boundary $\partial ^e U$ by
\begin{align}\label{eq:defpartiale}
\partial ^e U=\left\{\,\langle x, y \rangle\in \E_n^d:\, x\in U,\,y\notin U\,\right\}\,.
\end{align}
For $F\in\cB(\sR^d)$ and $\delta>0$, we denote by $\cB_{\dis}(F,\delta)$ the closed ball centered at $F$ or radius $\delta$ for the topology associated with the distance $\dis$, \textit{i.e.},
\[B_{\dis}(F,\delta)=\left\{E\in\cB(\sR^d): \,\cL^d(F\Delta E)\leq \delta\right\}\,.\]
Let $\sC_b(\sR^d,\sR)$ be the set of continuous bounded functions from $\sR^d$ to $\sR$.
We denote by $\sC^k_c(A,B)$ for $A\subset \sR^p$ and $B\subset\sR^q$, the set of functions of class $\sC^k$ defined on $\sR^p$, that takes values in $B$ and whose domain is included in a compact subset of $A$. The set of functions of bounded variations in $\Omega$, denoted by $BV(\Omega)$, is the set of all functions $u\in L^1(\Omega\rightarrow\sR,\cL^d)$ such that
\[\sup\left\{\int_\Omega\diver h \,d\cL^d:\,  h\in\sC_c^\infty(\Omega,\sR^d),\,\forall x\in\Omega \quad h(x)\in B(0,1)\right\}<\infty\,.\]

\noindent{\bf Rectifiability and Minkowski content.}
We will need the following proposition that enables to relate the measure of a neighborhood of a set $E$ with its Hausdorff measure. 
Let $p\geq 1$. Let $M$ be a set such that $\cH^p (M)<\infty$. We say that a set $M$ is $p$-rectifiable if there exists countably many Lipschitz maps $f_i:\sR^p \rightarrow \sR^d$ such that
\[\cH^p \left(M\setminus \bigcup _{i\in\sN} f_i(\sR^p)\right)=0\,.\] 
\begin{prop}\label{prop:minkowski}
Let $p\geq 1$. Let $M$ be a subset of $\sR^d$ that is $p$-rectifiable. Then we have
$$\lim_{r\rightarrow 0 }\frac{ \cL^d(\cV_2(M,r))}{\alpha_{d-p} r^{d-p}}= \cH^{p}(M)\,$$
where $\alpha_{d-p}$ denote the volume of the unit ball in $\sR^{d-p}$.
In particular, for $p=d-1$, we have
$$\lim_{r\rightarrow 0 }\frac{ \cL^d(\cV_2( M,r))}{2r}= \cH^{d-1}(M)\,.$$
\end{prop}
\noindent  This proposition is a consequence of the existence of the $(d-1)$-dimensional Minkowski content. We refer to Definition 3.2.37 and Theorem 3.2.39 in \cite{FED}.\newline

\noindent {\bf Sets of finite perimeter and surface energy.}
The perimeter of a Borel set $E$ of $\sR^d$ in an open set $\Omega$ is defined as 
\begin{align}\label{eq:defper}
\cP(E,\Omega)=\sup \left\{\int _E \diver f(x)\,d\cL^d(x):\, f\in C^\infty_c(\Omega,B(0,1))\right\}\,, 
\end{align}
where $C^\infty_c(\Omega,B(0,1))$ is the set of the functions of class $\sC^\infty$ from $\sR^d$ to $B(0,1)$ having a compact support included in $\Omega$, and $\diver$ is the usual divergence operator. The perimeter $\cP(E)$ of $E$ is defined as $\cP(E,\sR^d)$. The topological boundary of $E$ is denoted by $\partial E$.  \newline

\noindent {\bf The reduced boundary.}
For $E$ a set of finite perimeter, we denote by $\chi_E$ its characteristic function. The distributional derivative $\nabla \chi_E$ of $\chi_E$  is a vector Radon measure and $\cP(E,\Omega)=\|\nabla \chi_E\|(\Omega)$ where 
$\|\nabla \chi_E\|$ is the total variation measure of $\nabla\chi_E$.
The reduced boundary $\partial ^* E$ of $E$ is a subset of $\partial E$ such that, at each point $x$ of $ \partial ^* E$, it is possible to define a normal vector $n_E(x)$ to $E$ in a measure-theoretic sense, that is points such that
$$\forall r>0\qquad \|\nabla \chi_E\|(B(x,r))>0$$
and
$$\lim_{r\rightarrow 0}-\frac{\nabla \chi_E(B(x,r))}{\|\nabla \chi_E\|(B(x,r))}=n_E(x)\,.$$
For a point $x\in\partial ^* E$, we have
\begin{align}\label{eq:propnormvec}
\lim_{r\rightarrow 0}\frac{1}{r^d}\cL^d((E\cap B(x,r))\Delta B^-(x,r,n_E(x)))=0\,,
\end{align}
and for $\cH^{d-1}$ almost every $x$ in $\partial^* E$,
\begin{align}\label{eq:propsurfboule}
\lim_{r\rightarrow 0}\frac{1}{\alpha_{d-1}r^{d-1}}\cH^{d-1}(\partial ^* E\cap B(x,r))=1\,
\end{align}
where $\alpha_{d-1}$ is the volume of the unit ball in $\sR^{d-1}$.
Moreover, we have
$$\forall A\in\cB(\sR^d)\qquad \|\nabla \chi_E\|(A)=\cH^{d-1}(\partial ^* E\cap A)\,.$$
By De Giorgi's structure theorem (see for instance theorem 15.9 in \cite{maggi_2012}), the reduced boundary is $(d-1)$-rectifiable for any set of finite perimeter. This result will enable to apply proposition \ref{prop:minkowski} to the reduced boundary. In what follows, we won't recall this result and we will use proposition \ref{prop:minkowski} without justification.
 
Let $\beta>0$. We define the set $\sC_\beta$ as the sets of Borelian subsets of $\Omega$ of perimeter less than $\beta$
\begin{align}\label{eq:defcbeta}
\sC_\beta =\Big\{\,F\in\cB(\sR^d) : F\subset \Omega,\,\cP(F, \Omega)\leq \beta\,\Big\}
\end{align}
endowed with the topology associated to the distance $\dis$. For this topology, the set $\sC_\beta$ is compact.

The following little lemma will appear several times in what follows. We refer to \cite{Cerf:StFlour} for a proof of this lemma.
\begin{lem}[Lemma 6.7 in \cite{Cerf:StFlour}] \label{lem:estimeanalyse}
Let $f_1,\dots,f_r$ be $r$ non-negative functions defined on $]0,1[$. Then,
$$\limsup_{\ep\rightarrow 0}\ep\log \left(\sum_{i=1}^rf_i(\ep)\right)=\max_{1\leq i\leq r}\limsup_{\ep\rightarrow0}\ep\log f_i(\ep)\,.$$
\end{lem}

\subsection{State of the art}
\subsubsection{Law of large numbers for minimal cutset in a domain}
We work here with the same environment as in section \ref{section:env}. 
For any Borelian set $F\subset \Omega$ such that $\cP(F,\Omega)<\infty$, we define its capacity $\cI_\Omega(F)$ as follows
\[\cI_\Omega(F)=\int_{\Omega\cap\partial ^* F}\nu_G(n_F(x))d\cH^{d-1}(x)+\int_{\Gamma^2\cap\partial ^* F}\nu_G(n_F(x))d\cH^{d-1}(x)+ \int_{\Gamma^1\cap \partial ^*(\Omega\setminus F)}\nu_G(n_\Omega(x))d\cH^{d-1}(x)\,.\]
Note that by theorem \ref{thm:defnulgn}, the minimal capacity properly renormalized in a cylinder centered at a point $x\in\partial ^* F$ in the direction $n_F(x)$ in the lattice $(\sZ_n^d,\E_n^d)$ is $\nu_G(n_F(x))$. The capacity $\cI_\Omega(F)$ corresponds to $\capa(F,f)$ where $f(x)=\nu_G(n_F(x))$ for $x\in \partial ^* F$ and $f(x)=\nu_G(n_\Omega(x))$ for $x \in \Gamma^1\cap \partial ^*(\Omega\setminus F)$. 
We can interpret the capacity $\cI_\Omega(F)$ as the capacity of the continuous cutset $\fF$. We can prove that, almost surely, there exist cutsets in $\sC_n(\Gamma^1,\Gamma^2,\Omega)$ localized in some sense close to $\fF$ and of capacity of order $\cI_\Omega(F) n ^{d-1}$.
Denote by $\phi_\Omega$ the minimal continuous capacity, \textit{i.e.},
\begin{align}\label{eq:def:phiomega}
\phi_\Omega=\inf\left\{\cI_\Omega(F): \,F\subset\Omega,\,\cP(F,\Omega)<\infty\right\}\,.
\end{align}
Let us denote by $\Sigma^a$ the set of continuous cutsets that achieves $\phi_\Omega$
\[\Sigma ^ a=\left\{F\subset \Omega: \,\cP(F,\Omega)<\infty,\, \cI_\Omega(F)=\phi_\Omega\right\}\,.\]

In \cite{CT1}, Cerf and Théret proved a law of large numbers for the maximal flow and the minimal cutset in the domain $\Omega$. The minimal cutset converges in some sense towards $\Sigma^a$. 

\begin{thm}\label{thm:CerfTheret}[Cerf-Théret \cite{CT1}] Let $G$ that satisfies hypothesis \ref{hypo:G} and such that $G(\{0\})<1-p_c(d)$ (to ensure that $\nu_G$ is a norm). Let $(\Omega,\Gamma^1,\Gamma^2)$ that satisfies hypothesis \ref{hypo:omega}. Then the sequence $(\cE_n^{min})_{n\geq 1}$ (we recall that $\cE_n^{min}\in\sC_n(0)$) converges  almost surely for the distance $\dis$ towards the set $\Sigma^a$, that is,
\[a.s.,\,\quad \lim_{n\rightarrow\infty}\inf_{F\in\Sigma^a}\dis(\bR(\cE_n^{min}),F)=0\,.\]
Moreover, we have
\[\lim_{n\rightarrow\infty}\frac{\phi_n(\Gamma^1,\Gamma^2, \Omega)}{n^{d-1}}=\phi_\Omega>0\,.\]
\end{thm}
They first prove that from each subsequence of $(\bR(\cE_n^{min}))_{n\geq 1}$ they can extract a subsequence that converges for the distance $\dis$ towards a set $F\subset \Omega$ such that $\cP(F,\Omega)<\infty$.
Using locally the law of large numbers for the maximal flow in a cylinder (theorem \ref{thm:defnulgn}), they prove that
\[\liminf_{n\rightarrow \infty}\frac{V(\cE_n^{min})}{n^{d-1}}\geq \cI_\Omega(F)\geq \phi_\Omega\,.\] 
The remaining part of the proof required working with maximal streams (which is the dual object associated with minimal cutset). We refer to the companion paper \cite{dembin2020large} for a more precise study of maximal streams.

\subsubsection{Lower large deviations for the maximal flow}
\noindent{\bf Lower large deviation principle in a cylinder.}
In \cite{Rossignol2010}, Rossignol and Théret proved a lower large deviation principle for the variable $\tau$. The rate function they obtain is going to be our basic brick to build the rate function for lower large deviation principle in a general domain.
\begin{thm}[Large deviation principle for $\tau$]\label{thm:lldtau} Suppose that $G(\{0\})<1-p_c(d)$ and that $G$ has an exponential moment. Let $h>0$. Then for every vector $\vv\in\sS^{d-1}$, for every non degenerate hyperrectangle $A$ normal to $\vv$, the sequence 
$$\left(\frac{\tau_n(A,h)}{\cH^{d-1}(A)n^{d-1}},n\in\sN\right)$$ satisfies a large deviation principle of speed $\cH^{d-1}(A)n^{d-1}$ governed by the good rate function $\cJ_{\vv}$. 
Moreover, we know that that $\cJ_{\vv}$ is convex on $\sR_+$, infinite on $[0,\delta_G\|v\|_1[\cup]\nu_G(v),+\infty[$ where $\delta_G=\inf\{t, \Prb(t(e)\leq t)>0\}$, equal to $0$ at $\nu_G(v)$, and if $\delta_G\|\vv\|_1<\nu_G(v)$, we also know that $\cJ_{\vv}$ is finite on $]\delta_G\|\vv\|_1,\nu_G(v)]$, continuous and strictly decreasing on $[\delta_G\|\vv\|_1,\nu_G(v)]$ and strictly positive on $]\delta_G\|\vv\|_1,\nu_G(v)[$.
\end{thm}

\noindent{\bf Lower large deviations for the maximal flow in a domain.} We here work with the same environment as in section \ref{section:env}. Cerf and Théret proved in \cite{CT3} that the lower large deviations for the maximal flow $\phi_n$ through $\Omega$ are of surface order.
\begin{thm}If $(\Omega,\Gamma^1,\Gamma^2)$ satisfy hypothesis \ref{hypo:omega} and the law $G$ of the capacity admits an exponential moment, \textit{i.e.}, there exists $\theta>0$ such that
$$\int_{\sR^+ }\exp(\theta x)dG(x)<\infty$$ and if $G(\{0\})<1-p_c(d)$, then there exists a finite constant $\phi_\Omega>0$ such that 
\[\forall \lambda<\phi_\Omega \qquad\limsup_{n\rightarrow \infty}\frac{1}{n^{d-1}}\log \Prb(\phi_n(\Gamma^1,\Gamma^2,\Omega)\leq \lambda n ^{d-1})<0\,.\]
\end{thm}
\begin{rk} Note that this constant $\phi_\Omega$ is the same than in \eqref{eq:def:phiomega}.
\end{rk}
To prove this result, on the lower large deviation event, they consider the continuous subset $E_n=\bR(\cE_n^{min})$ where $\cE_n ^{min}$ is a minimal cutset $\cE_n ^{min}$. Since we can control the number of edges in a minimal cutset thanks to the work of Zhang \cite{Zhang2017}, we can control the perimeter of $E_n$ in $\Omega$ and work with high probability with a continuous subset $E_n$ with perimeter less than some constant $\beta>0$. The set $E_n$ belongs to the compact set $\sC_\beta$, this enables to localize the set $E_n$ close to some set $F\in\sC_\beta$. To conclude, the idea is to say that since the capacity of $\cE_n^{min}$ is strictly smaller that $\cI_\Omega(F)n ^{d-1}$, then there exists locally a region on the boundary of $F$ where the flow is abnormally low. We can relate this event with lower large deviations for the maximal flow in a small cylinder, whose probability decay speed is of surface order. This large deviation result was used to prove the convergence of the rescaled maximal flow $\phi_n/n^{d-1}$ towards $\phi_\Omega$. This strategy was already using the study of a cutset but was too rough to derive a large deviation principle.

\subsection{Sketch of the proof}
\noindent{\bf Step 1. Admissible limiting object.} One of the main difficulty of this work was to identify the right object to work with. The aim is that the limiting objects must be measures supported on surfaces. In particular, we want to prove that the limiting object for $(\bR(\cE_n),\mu_n(\cE_n))$ is contained in the set $\fT_\cM$. Thanks to the compactness of the set $\sC_\beta$, if we have a control on the perimeter of $\bR(\cE_n)$ we can easily prove that up to extraction, it converges towards a continuous Borelian subset $E$ of $\Omega$. If we work with general cutsets without any restriction, the limiting object for $\mu_n(\cE_n)$ may not be a measure supported on the surface $\fE$. This is due to the potential presence of long thin filaments for the set $\cE_n$ that are of negligible volume for the set $\bR(E_n)$ but in the limit these filaments can create measure outside of $\fE$. Working with a minimal cutset $\cE_n^{min}$ prevents the existence of these long filaments that are not optimal to minimize the capacity when $G(\{0\})<1-p_c(d)$. This will ensure that the weak limit of $\mu_n(\cE_n^{min})$ is supported on $\fE$. However, working with minimal cutset leads to a major difficulty for the lower bound (see the next step). One solution is to work with almost minimal cutsets, \textit{i.e.}, cutsets in $\sC_n(\ep)$. We prove that in some sense the limiting objects for $((R(\cE_n),\mu_n(\cE_n)), \,\cE_n\in\sC_n(\ep))_{n\geq 1}$ are contained in the set $\fT_\cM$. \newline

\noindent{\bf Step 2. Lower bound.} For any $(E,\nu)\in\fT_\cM$ such that $\widetilde{I}(E,\nu)<\infty$, we prove that for any neighborhood $U$ of $(E,\nu)$ the probability that there exists a cutset $\cE_n\in\sC_n(\ep)$ such that that $(R(\cE_n),\mu_n(\cE_n))\in U$ is at least $\exp(-n^{d-1}\widetilde{I}(E,\nu))$. Write $\nu=f\cH^{d-1}|_\fE$.
To prove this result, we build a configuration where the expected event occurs using elementary events as building blocks. We first cover almost all the boundary of $E$ by a finite family $(B(x_i,r_i,v_i))_{i\in I}$ of disjoint closed balls such that on each ball $\partial E$ is almost flat and $f$ is almost constant. Using result for lower large deviations for the maximal flow in cylinders, we can prove that the probability that there exists in the ball $B(x,r,v)$ a cutset separating $\partial ^-_n B(x,r,v)$ from  $\partial ^+_n B(x,r,v)$ of capacity smaller than $f(x)\alpha_{d-1} r^{d-1}$ is of probability at least $\exp(-\alpha_{d-1} r^{d-1}\cJ_v(f(x)))$. Let us denote by $\cE^{(i)}$ this event associated with the ball $B(x_i,r_i,v_i)$ and $\cE_n(i)$ the cutset given in the definition of the event (if there are several possible choices, we pick one according to a deterministic rule). We can build a set of edges $\cF_0$ of negligible cardinal such that $\cE_n:=\cF_0\cup\cup_{i\in I}\cE_n(i)\in\sC_n(\Gamma^1,\Gamma^2,\Omega)$.
 If the radii of the balls are small enough depending on the neighborhood $U$, we can prove that on  the event $\cap_{i\in I}\cE^{(i)}$, we have $(\bR(\cE_n),\mu_n(\cE_n))\in U$.
Since the balls are disjoint, we have
\begin{align*}
\Prb(\exists \cE_n\in\sC_n(\Gamma^1,\Gamma^2,\Omega): (\bR(\cE_n),\mu_n(\cE_n))\in U)&\geq
\Prb\left(\bigcap_{i\in I}\cE^{(i)}\right)=\prod_{i\in I}\Prb(\cE^{(i)})\\&\geq \exp\left(-\sum_{i\in I}\cJ_{v_i}(f(x_i))\alpha_{d-1} r_i^{d-1} n ^{d-1}\right)\\&\approx \exp\left(-\widetilde{I}(E,\nu)n^{d-1}\right)\,.
\end{align*}
It remains to prove that $\cE_n$ is almost minimal, this is the main difficulty of this step. To do so, we have to ensure that everywhere outside the balls, the flow is not abnormally low. This step is very technical. Using this strategy, we did not manage to prove that the cutset we have built is minimal but only almost minimal.\newline

\noindent {\bf Step 3. Upper bound.} In the lower bound section, we build a configuration upon elementary events on balls. Here, we do the reverse. We deconstruct the configuration into a collection of elementary events on disjoint balls. Fix $(E,\nu)\in\fT_\cM$, write $\nu=f\cH^{d-1}|_\fE$. We first cover almost all the boundary of $E$ by a finite family $(B_i=B(x_i,r_i,v_i))_{i\in I}$ of disjoint closed balls such that on each ball $\partial E$ is almost flat and $f$ is almost constant. We pick a neighborhood $\cU$ of $(E,\nu)$ adapted to this covering such that 
for $\cE_n\in\sC_n(\Gamma^1,\Gamma^2,\Omega)$ such that $(\bR(\cE_n),\mu_n(\cE_n))\in \cU$, we have that for each $i\in I$ the set $\cE_n\cap B_i$ is almost a cutset between $\partial_n^-B_i$ and $\partial_n^+ B_i$ in $B_i$ (up to adding a negligible number of edges) and $V(\cE_n\cap B_i)\leq f(x_i)\alpha_{d-1}r_i ^{d-1}$. Hence, on the event  
$\{\exists\cE_n\in\sC_n(\Gamma^1,\Gamma^2,\Omega):(\bR(\cE_n),\mu_n(\cE_n))\in \cU\}$, the event $\cap_{i\in I }\cE^{(i)}$ occurs where $\cE^{(i)}$ was defined in the previous step. We can prove using estimates on lower large deviations on cylinders that $\Prb(\cE^{(i)})\leq \exp(-\alpha_{d-1}r_i^{d-1}\cJ_{v_i}(f(x_i))n^{d-1})$.
Since the balls are disjoint it follows that
\begin{align*}
\Prb\left(\exists\cE_n\in\sC_n(\Gamma^1,\Gamma^2,\Omega):(\bR(\cE_n),\mu_n(\cE_n))\in \cU\right)\leq\Prb\left(\bigcap_{i\in I }\cE^{(i)}\right)&=\prod_{i\in I}\Prb(\cE^{(i)})\\ &\leq \exp\left(-\sum _{i\in I}\alpha_{d-1}r_i^{d-1}\cJ_{v_i}(f(x_i))n^{d-1}\right)\\
&\approx \exp(-\widetilde{I}(E,\nu)n ^{d-1})\,.
\end{align*}

The remaining of the proof uses standard tools of large deviations theory. The proof of this large deviation principle for almost minimal cutsets enables to deduce by a contraction principle a lower large deviation principle for the maximal flow.\newline

\noindent {\bf Organization of the paper.} In section 
\ref{sec:prelwork}, we prove some useful results such as a covering theorem and lower large deviations for the maximal flow in a ball. In section \ref{sec:admissible} (corresponding to step 1), we study the properties of the limiting objects to prove that they belong with high probability to a neighborhood of the set $\fT_\cM$. 
In section \ref{sec:lb} (corresponding to step 2) and \ref{sec:ub} (corresponding to step 3), we prove local estimates on the probability $\Prb_n^\ep(U)$ for some neighborhoods $U$ to be able to deduce a large deviation principles for the almost minimal cutsets. Finally, in section \ref{sect:goodtaux}, we conclude the proof of the two main theorems \ref{thm:pgd} and \ref{thm:lldmf}.

\section{Preliminary work}\label{sec:prelwork}
\subsection{Vitali covering}
We will use the Vitali covering theorem for $\cL^d$. A collection of sets $\cU$ is called a Vitali class for a Borelian set $\Omega$ of $\sR^d$, if for each $x\in\Omega$ and $\delta>0$, there exists a set $U\in\cU$ such that $x\in U$ and $0<\diam U<\delta$ where $\diam U$ is the diameter of the set $U$ for the Euclidean distance. We now recall the Vitali covering theorem for $\cH^{d-1}$ (Theorem 1.10 in \cite{FAL})
\begin{thm}[Vitali covering theorem]\label{thm:vitali} Let $F\subset \sR^d$ such that $\cH^{d-1}(F)<\infty$ and $\cU$ be a Vitali class of closed sets for $F$. Then we may select a countable disjoint sequence $(U_i)_{i\in I}$ from $\cU$ such that 
$$\text{either}\quad \sum_{i\in I}(\diam U_i)^{d-1}=+\infty \qquad \text{or} \quad \cH^{d-1}\left(F\setminus \bigcup_{i\in I}U_i\right)=0\,.$$
\end{thm}
\noindent We next recall the Besicovitch differentiation theorem in $\sR ^d$
(see for example theorem 13.4 in \cite{Cerf:StFlour}):
\begin{thm}[Besicovicth differentiation theorem]\label{thm:Besicovitch} Let $\mu$ be a finite positive Radon measure on $\sR^d$. For any Borel function $f \in L^1(\mu)$, the quotient
$$\frac{1}{\mu(B(x,r))}\int_{B(x,r)}f(y)d\mu(y)$$
converges $\mu$-almost surely towards $f(x)$ as $r$ goes to $0$. 
 \end{thm}
 We recall that for any set $E\in\sC_{<\infty}$, the set $\fE$ denotes the continuous cutset associated with $E$ that was defined in \eqref{def:fE}. We will use at several moments in the proof the Vitali covering theorem. To avoid repeating several times the same arguments, we here present a general result for covering a surface by disjoint closed balls that satisfy a list of properties. These properties are typical for balls centered at points in the surface provided that their radius is small enough.
%
%
%
 \begin{prop}[Covering $\fE$ by balls]\label{prop:utilisationvitali}Let $E\in\sC_{<\infty}$. Let $\ep\in]0,1/2]$.
 Let $(P_1)_{x,r},\dots,(P_m)_{x,r}$ be a family of logical proposition depending on $x\in \sR^d$, $\ep$ and $r>0$.
 We assume that there exists $\cR$ such that $\cH^{d-1}(\fE\setminus \cR)=0$ and
 \[\forall x\in \cR\qquad \exists r_x>0\qquad\text{such that}\qquad\forall i\in\{1,\dots,m\}\quad \forall\,0< r\leq r_x\qquad (P_i)_{x,r}\text{ holds}\,.\]
 Then, there exists a finite family of disjoint closed balls $(B(x_i,r_i,v_i))_{i\in I}$ with $v_i=n_\Omega(x_i)$ (respectively $n_E(x_i)$) for $x_i\in \partial ^* \Omega\setminus \overline{\partial ^* E}$ (resp. $x_i\in \partial ^* E$) such that 
\begin{align*}
\cH^{d-1}(\fE\setminus \cup_{i\in I}B(x_i,r_i)))\leq \ep\,,
\end{align*}
\begin{align*}
\forall i \in I\quad  \forall\, 0<r\leq r_i\qquad \left|\frac{1}{\alpha_{d-1} r^{d-1}}\cH^{d-1}(\fE\cap B(x_i,r))-1\right|\leq \ep\,,
\end{align*}
\begin{align*}
\forall i \in I\quad  \forall\, 0<r\leq r_i\qquad \forall j\in\{1,\dots,m\}\qquad (P_j)_{x_i,r}\quad \text{holds}\,.
\end{align*}
 \end{prop}
 At this point, the logical propositions $(P_i)_{x,r}$ may seem a bit abstract. To better understand the kind of applications, we will do, let us give an example of such a proposition:
$$ (P_1)_{x,r}:= \qquad  \ll x\in \partial^* \Omega \implies \cL^d((\Omega\cap B(x,r))\Delta B^-(x,r,n_\Omega(x)))\leq \ep \alpha_d r^d \gg\,.$$
 
 \begin{proof}[Proof of proposition \ref{prop:utilisationvitali}]
 We follow the proof of lemma 14.6 in \cite{Cerf:StFlour}.
 Let $\ep$ be a positive constant, with $\ep<1/2$. 
 
\noindent {\bf First case: $x\in\partial ^* E$. }By inequality \eqref{eq:propsurfboule}, we have for $\cH^{d-1}$-almost every $x\in \partial ^* E$
$$\lim_{r\rightarrow 0}\frac{1}{\alpha_{d-1}r^{d-1}}\cH^{d-1}(\partial ^* E\cap B(x,r))=1\,.$$
We denote by $\cR_1$ the set of points in $\partial^* E$ such that the equality holds. Hence, we have 
\begin{align}\label{eq:neg1}
\cH^{d-1}(\partial^* E\setminus \cR_1)=0\,.
\end{align}
It yields that for $x\in\cR^1$, there exists a positive constant $r_1(x,\ep)>0$ such that for any $r \leq r_1(x,\ep)$
\[\left|\frac{1}{\alpha_{d-1} r^{d-1}}\cH^{d-1}(\partial^*  E\cap B(x,r))-1\right|\leq \ep\,.\]
\noindent {\bf Second case: $x\in  \partial ^* \Omega \cap (\Gamma ^1\setminus \overline{\partial ^*E})$.}
By inequality \eqref{eq:propsurfboule}, we have for $\cH^{d-1}$-almost every $x\in \partial ^* \Omega$
$$\lim_{r\rightarrow 0}\frac{1}{\alpha_{d-1}r^{d-1}}\cH^{d-1}(\partial  \Omega\cap B(x,r))=1\,.$$
We denote by $\cR_2$ the set of points in $ \partial ^* \Omega \cap (\Gamma ^1\setminus \overline{\partial ^*E})$ such that the previous equality holds. Hence, we have 
\begin{align}\label{eq:neg2}
\cH^{d-1}\left((\partial ^* \Omega \cap (\Gamma^1\setminus \overline{\partial^*E}))\setminus \cR_2\right)=0\,.
\end{align}
It yields that for $x\in\cR_2$, there exists a positive constant $r_1(x,\ep)\in]0,\ep]$ such that for any $r\leq r_1(x,\ep)$
\[\left|\frac{1}{\alpha_{d-1} r^{d-1}}\cH^{d-1}(\partial^*  \Omega\cap B(x,r))-1\right|\leq \ep\,.\]


\noindent {\bf Extract a countable covering by balls.} The family of balls
\[\big(B(x,r),\, x\in(\cR_1\cup\cR_2)\cap\cR,\,r<\min(1,r_1(x,\ep),r_x)\big)\]
is a Vitali class for $(\cR_1\cup\cR_2)\cap \cR$. By Vitali Covering theorem (theorem \ref{thm:vitali}), we may select from this family a countable (or finite) disjoint sequence of balls $(B(x_i,r_i),i\in I)$  such that 
either $$\sum_{i\in I }r_i^{d-1}=+\infty \qquad\text{or}\qquad \cH^{d-1}\left(((\cR_1\cup\cR_2)\cap\cR)\setminus \bigcup_{i\in I }B(x_i,r_i)\right)=0\,.$$
We know that $E$ and $\Omega$ both have finite perimeter.
Since the balls in the family $(B(x_i,r_i),i\in I)$ are disjoint, we have
$$\sum_{i\in I }\cH^{d-1}(B(x_i,r_i)\cap\fE)\leq \cH^{d-1}(\fE)<\infty\,.$$
We recall that for any $i\in I$, we have
$$\left|\alpha_{d-1}r_i^{d-1}- \cH^{d-1}(B(x_i,r_i)\cap\fE)\right|\leq \alpha_{d-1}r_i^{d-1}\ep\,.$$
Hence, we have
$$\sum_{i\in I }r_i^{d-1}\leq \frac{1}{\alpha_{d-1}(1-\ep)}\sum_{i\in I }\cH^{d-1}(B(x_i,r_i)\cap\fE)<\infty\,.$$
As a result, we obtain that 
\begin{align}\label{eq:neg3}
\cH^{d-1}\left((\cR_1\cup\cR_2)\cap\cR \setminus \bigcup_{i\in I }B(x_i,r_i)\right)=0\,.
\end{align}
Consequently, we can extract from $I$ a finite set $I_0$ such that 
\begin{align*}
\cH^{d-1}\left((\cR_1\cup\cR_2)\cap\cR \setminus \bigcup_{i\in I _0}B(x_i,r_i)\right)\leq \ep\,.
\end{align*}
We set
$v_i=n_\Omega(x_i)$ (respectively $n_E(x_i)$) for $x_i\in \partial ^* \Omega\setminus \overline{\partial ^* E}$ (resp. $x_i\in \partial ^* E$).
Since $\cH^{d-1}(\fE\setminus ((\cR_1\cup\cR_2)\cap\cR))=0$, this concludes the proof.
 \end{proof}
 \subsection{Localization of $\Omega$ inside balls centered at the boundary of $\Omega$}
The following result will be important in what follows. Since the boundary of $\Omega$ is smooth, the intersection of $\Omega$ with a small ball $B(x,r)$ centered at $x\in \Gamma$ is close to $B^-(x,r,n_\Omega(x))$. In the following lemma, we localize precisely the symmetric difference between $\Omega\cap B(x,r)$ and $B^-(x,r,n_\Omega(x))$.
 \begin{lem}\label{lem:Omegabord}Let $\Omega$ that satisfies hypothesis \ref{hypo:G}. Let $x\in \partial^ * \Omega$. For any $\delta>0$, there exists $r_0$ depending on $\delta$, $x$ and $\Omega$ such that
 \[\forall\, 0<r\leq r_0\qquad(\Omega\cap B(x,r))\Delta B^- (x,r,n_\Omega(x))\subset\left\{\,y\in B(x,r):|(y-x)\cdot n_\Omega(x)|\leq \delta \|y-x\|_2\,\right\}\,.\]
 \end{lem}
 \begin{proof}Let $x\in\partial ^*\Omega$. We recall that $\Gamma=\partial \Omega$ is contained in the union of a finite number of oriented hypersurfaces of class $\sC^1$ that intersect each other transversally. We claim that if $x\in \partial ^*\Omega$, it cannot be contained in the transversal intersection. Indeed, by definition for points in a transversal intersection, we cannot properly define an exterior unit normal vector at points in the intersection. It follows that $\Gamma$ is locally $\sC^1$ around $x$. Let $\delta>0$. There exists $r_0>0$ such that 
 \[\forall\, 0<r\leq r_0\quad \forall y \in B(x,r)\cap\Gamma\qquad \|n_\Omega(y)-n_\Omega(x)\|_2\leq\delta\,.\]
 Since $\Omega$ is a Lipschitz domain, up to choosing a smaller $r$, there exists an hyperplane $H$ containing $x$ of normal vector $n_\Omega(x)$ and $\phi: H\rightarrow \sR$ a Lipschitz function such that
\[B(x,r)\cap \Gamma = \left\{y+\phi(y) n_\Omega(x): y\in H\cap B(x,r)\right\}\,\]
and
\[B(x,r)\cap \Omega = \left\{y+t n_\Omega(x): y\in H\cap B(x,r),\, t\leq \phi(y)\right\}\cap B(x,r)\,\]
 Let $y\in B(x,r)\cap\partial\Omega $. 
 \begin{figure}[H]
\begin{center}
\def\svgwidth{0.6\textwidth}
   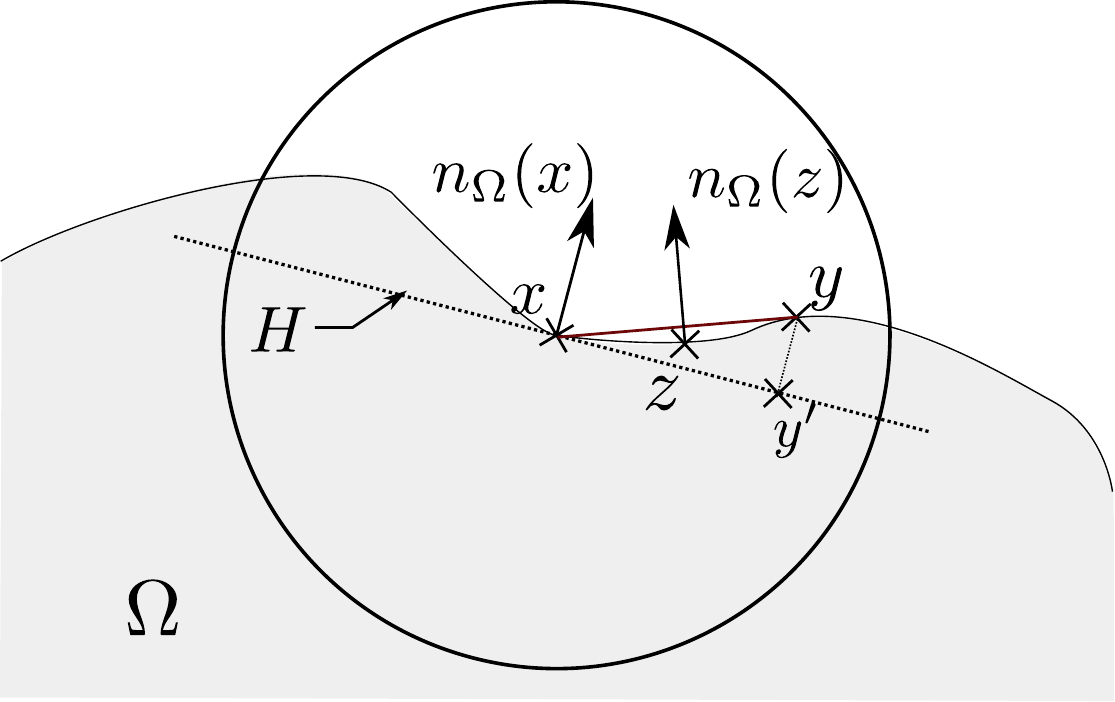
   \caption{\label{fig:boundaryomega}Reprentation of $B(x,r)$, $H$, $y$ and $z$.}
   \end{center}
\end{figure}
Let us denote by $x'$ and $y'$ the points in $H$ such that $x=x'+\phi(x') n_\Omega(x)$ and $y=y'+\phi(y') n_\Omega(x)$ (actually, $x'=x$ and $\phi(x')=0$, see figure \ref{fig:boundaryomega}). Let us denote by $\phi_0$ the following mapping
$$\forall s\in[0,1]\qquad \phi_0(s)=\phi((1-s)x'+sy')\,.$$
Note that since $\partial \Omega \cap B(x,r)$ is contained in a $\sC^1$ hypersurface, the mapping $\phi_0\in\sC^1(\sR, \sR)$.
By the mean value theorem, there exists $s\in]0,1[$ such that $\phi_0'(s)=\phi_0(1)-\phi_0(0)=\phi(y')-\phi(x')$. In other words, the vector $y'-x'+(\phi(y')-\phi(x')) \, n_\Omega(x)=y-x$ belongs to the tangent space of the point $z=((1-s)x'+sy'+\phi_0(s) n_\Omega(x))\in \Gamma$. Consequently, we have $$n_{\Omega}(z)\cdot (y-x)=0\,.$$
 Hence, we get using Cauchy Schwartz inequality
 \[|(y-x)\cdot n_\Omega(x)|=|(y-x)\cdot(n_\Omega(z)-n_\Omega(x))| \leq \|y-x\|_2 \|n_\Omega(z)-n_\Omega(x)\|_2\leq\delta \|y-x\|_2\,\]
 and
  \[\partial \Omega \cap B(x,r)\subset \left\{\,y\in B(x,r):|(y-x)\cdot n_\Omega(x)|\leq \delta \|y-x\|_2\,\right\}\,.\]
  Let $w\in\Omega \cap B^ +(x,r,n_\Omega(x))$. There exists $w'\in H\cap B(x,r)$ and $0\leq t\leq\phi(w')$ such that $w=w'+t n_\Omega(x)$. If $\phi(w')=0$, then we have $w\in  \left\{\,y\in B(x,r):|(y-x)\cdot n_\Omega(x)|\leq \delta \|y-x\|_2\,\right\}$. Let us assume $\phi(w')>0$. 
 Since $w'+\phi(w')n_\Omega(x)\in \Gamma$, we have
  \[\phi(w')=|(w'+\phi(w')n_\Omega(x)-x)\cdot n_\Omega(x)|\leq \delta \|w'+\phi(w')n_\Omega(x)-x\|_2\,.\]
It follows that
  \begin{align*}
  |(w-x)\cdot n_\Omega(x)|= |(w'-x)\cdot n_\Omega(x)+ t|=t&= \frac{t}{\phi(w')} \phi(w')\leq \frac{t}{\phi(w')}  \delta \|w'+\phi(w') n_\Omega(x)-x\|_2\\&\leq \delta\frac{t}{\phi(w')} \|w'-x\|_2+\delta t\leq  \delta(\|w'-x\|_2+ t) \\&=\delta \|w'+t n_\Omega(x)-x\|_2
  \end{align*}
  and $w\in  \left\{\,y\in B(x,r):|(y-x)\cdot n_\Omega(x)|\leq \delta \|y-x\|_2\,\right\}$. As a result, we get
  $$\Omega \cap B^ +(x,r,n_\Omega(x))\subset  \left\{\,y\in B(x,r):|(y-x)\cdot n_\Omega(x)|\leq \delta \|y-x\|_2\,\right\}
  \,.$$
   By the same arguments, we can prove that 
     $$\Omega ^ c\cap B^ -(x,r,n_\Omega(x))\subset  \left\{\,y\in B(x,r):|(y-x)\cdot n_\Omega(x)|\leq \delta \|y-x\|_2\,\right\}
  \,.$$
  The result follows.
 \end{proof}
 \subsection{Lower large deviations for the maximal flow in a ball}
 \subsubsection{Upper bound on the probability of lower large deviations for the maximal flow in a ball}\label{section:upperboundball}
 Note that the starting point to understand large deviations for the maximal flow in general domain is to first understand large deviations in cylinders. Several times in this paper, we will cover surfaces by disjoint balls. Hence, we will need to understand how the maximal flow behaves in a ball. From our knowledge on large deviations in cylinders, we can deduce results for large deviations in balls that is the basic brick we need in this paper.
We will need the following lemma that is an adaptation of what is done in section 6 in \cite{CT3}. Let $n\geq 1$. Let $x\in\sR^d$, $v\in\sS^{d-1}$ and $r$, $\delta$, $\zeta$ be positive constants.
We first define 
$G_n(x,r,v,\delta,\zeta)$ to be the event that there exists a set $U\subset B(x,r)\cap \sZ_n^d$ such that:
 $$\card (U\Delta (B^-(x,r,v)\cap\sZ ^d_n))\leq  4\delta \alpha_d r^d n^d$$
 and
 $$V((\partial ^eU)\cap B(x,r))\leq \zeta \alpha_{d-1}r^{d-1}n ^{d-1}\,$$
 where we recall that $\partial ^e U$ denotes the edge boundary of $U$ and was defined in \eqref{eq:defpartiale}. The set $(\partial ^eU)\cap B(x,r)$ correspond to the edges in $\partial ^e U$ that have both endpoints in $B(x,r)$.
 \begin{lem}\label{lem: Gxrv} There exists a constant $\kappa_0$ depending only on $d$ and $M$ such that for any $x\in\sR^d$, $v\in\sS^{d-1}$ and $r$, $\delta$, $\zeta$ positive constants, we have
\[\limsup_{n\rightarrow \infty} \frac{1}{n^{d-1}}\Prb \left(G_n(x,r,v,\delta,\zeta)\right)\leq -g(\delta)\alpha_{d-1}r^{d-1} \cJ_v\left(\frac{\zeta+\kappa_0\sqrt{\delta}}{g(\delta)}\right)\,\]
where $g(\delta)=(1-\delta)^{(d-1)/2}$.
\end{lem}
\begin{proof}[Proof of lemma \ref{lem: Gxrv}]
We aim to prove that on the event $G_n(x,r,v,\delta,\zeta)$, we can build a cutset that separates the upper half part of $\partial B(x,r,v)$ (upper half part according to the direction $v$) from the lower half part that has a capacity close to $\zeta \alpha_{d-1}r^{d-1}n ^{d-1}$. To do so, we build from the set $U$ an almost flat cutset in the ball. The fact that $\card (U\Delta B^-(x,r,v))$ is small implies that $\partial^e U$ is almost flat and is close to $\disc(x,r,v)$. However, this does not prevent the existence of long thin strands that might escape the ball and prevent $U$ from being a cutset in the ball. The idea is to cut these strands by adding edges at a fixed height. We have to choose the appropriate height to ensure that the extra edges we needed to add to cut these strands are not too many, so that we can control their capacity. The new set of edges we create by adding to $U$ these edges will be in a sense a cutset. The last thing to do is then to cover the $\disc(x,r,v)$ by hyperrectangles in order to use the rate function that controls the decay of the probability of having an abnormally low flow in a cylinder.
Let $\rho>0$ be a small constant that we will choose later depending on $\delta$. We define 
$$\gamma_{max}=\rho r\,.$$
This constant $\gamma_{max}$ will represent the height of a cylinder of basis $\disc(x,r',v)$. We have to choose $r'$ in such a way that $\cyl(\disc(x,r',v),\gamma_{max})\subset B(x,r)$.
We set 
$$r'=r \sqrt{1-\rho^2}\,.$$
On the event $G_n(x,r,v,\delta,\zeta)$, we consider a fixed set $U$ satisfying the properties described in the definition of the event.
For each $\gamma$ in $\{1/n,\dots , (\lfloor n\gamma_{max}\rfloor -1)/n\}$, we define
\[D(\gamma)=\cyl(\disc(x,r',v),\gamma),\]
\[\partial ^+_n D(\gamma)=\left\{y\in D(\gamma)\cap\sZ_n^d:\exists z\in\sZ_n^d, (z-x)\cdot v>\gamma\text{ and } \|z-y\|_1=\frac{1}{n}\right\}\,\]
and
\[\partial ^-_n D(\gamma)=\left\{y\in D(\gamma)\cap\sZ_n^d:\exists z\in\sZ_n^d, (z-x)\cdot v<-\gamma\text{ and } \|z-y\|_1=\frac{1}{n}\right\}\,.\]
The sets $\partial^+_n D(\gamma)\cup\partial ^-_nD(\gamma)$ are pairwise disjoint for different $\gamma$.
Moreover, we have
$$\sum_{\gamma=1/n,\dots,(\lfloor n\gamma_{max}\rfloor -1)/n}\card(U\cap \partial^+_n D(\gamma))+\card(U^c\cap \partial^-_n D(\gamma))\leq \card(U\Delta (B^-(x,r,v)\cap\sZ_n^d))\leq  4\delta \alpha_d r^d n^d\,.$$
By a pigeon-hole principle, there exists $\gamma_0$ in $\{1/n,\dots,(\lfloor n\gamma_{max}\rfloor -1)/n\}$ such that 
$$\card(U\cap \partial^+_n D(\gamma_0))+\card(U^c\cap \partial^-_n D(\gamma_0))\leq \frac{4\delta \alpha_d r^d n^d}{\lfloor n\gamma_{max}\rfloor -1}\leq \frac{5\delta \alpha_d r^dn^{d-1}}{\gamma_{max}}=\frac{5\delta \alpha_d r^{d-1}n^{d-1}}{\rho}$$
for $n$ large enough. If there are several choices for $\gamma_0$, we pick the smallest one.
We denote by $X=U\cap D(\gamma_0)$. We define by $X^+$ and $X^-$ the following set of edges:
$$X^+=\{\langle y, z\rangle\in \E_n^d: y\in \partial ^+_nD(\gamma_0)\cap X, z\notin D(\gamma_0)\}\,,$$
$$X^-=\{\langle y, z\rangle\in \E_n^d: y\in \partial ^-_nD(\gamma_0)\cap X^c, z\notin D(\gamma_0)\}\,.$$
Let us control the number of edges in $X^+\cup X^-$:
$$\card(X^+ \cup X^-)\leq 2d\left(\card(U\cap \partial^+_n D(\gamma_0))+\card(U^c\cap \partial^-_n D(\gamma_0))\right)\leq 2d \frac{5\delta \alpha_d r^{d-1}n^{d-1}}{\rho}=C_d\delta \rho ^{-1}r^{d-1}n^{d-1}\,$$
where $C_d=10d\alpha_d$.
We want to relate the event $G_n(x,r,v,\delta,\zeta)$ to flows in cylinders. There exists a constant $c_d$ depending only on the dimension such that for any positive $\kappa$, there exists a finite collection of closed disjoint hyperrectangles $(A_i)_{i\in J}$ included in $\disc(x,r',v)$ such that 
\begin{align}\label{eq:condrect}
\sum_{i\in J}\cH^{d-1}(A_i)\geq \alpha_{d-1}r'^{d-1}-\kappa\qquad\text{and}\qquad \sum_{i\in J}\cH^{d-2}(\partial A_i)\leq c_d r'^{d-2}\,.
\end{align}
Since all the hyperrectangles are closed and disjoint, we have
$$\xi=\min\{d_2(A_i,A_j):i\neq j\in J\}>0\,.$$ 
For any $i\in J$, we denote by $\cP_i(n)$ the edges with at least one endpoint in $\cV_2(\cyl(\partial A_i,\gamma_0),d/n)$. For $n$ sufficiently large, we have $\xi>4d/n$ and all the sets $\cP_i(n)$ are pairwise disjoint.
Moreover, using proposition \ref{prop:minkowski}, we have
\begin{align*}\sum_{i\in J}\card(\cP_i(n))\leq 2d\sum_{i\in J}\cL^d(\cV_2(\cyl(\partial A_i,\gamma_0),2d/n))n^{d}& \leq \sum_{i\in J}16d^2\cH^{d-2}(\partial A_i)2\gamma_0n^{d-1}\\&\leq 32d^2\rho rc_dr^{d-2} n^{d-1}\leq 32d^2 c_d\rho r^{d-1}n ^{d-1}\,.
\end{align*}
We define $$\cE_i=\cP_i(n)\cup \big((X^+ \cup X^-\cup(\partial ^e U\cap D(\gamma_0)))\cap \cyl(A_i,\gamma_{max})\big) \,.$$
We can check that $\cE_i$ is a cutset for $\tau_n(A_i,\gamma_0)$ (we don't prove here that the set $\cE_i$ is a cutset, we refer to the proof of lemma \ref{lem:inclusion} for a proof that a set is a cutset). The sets $\cE_i$ are pairwise disjoint and 
\begin{align*}
\sum_{i\in J} V(\cE_i)&\leq V\left(\partial ^e U\cap B(x,r)\right)+ M\left( \sum_{i\in J}\card(\cP_i(n))+ \card(X^+\cup X^-)\right)\\&\leq \left(\zeta \alpha_{d-1}+ (C_d\delta \rho^{-1}+32d^2c_d\rho)M\right) r^{d-1}n^{d-1}\,.
\end{align*}
 Set $\rho = \sqrt{\delta}$.
We obtain that
\begin{align}\label{eq:conclem1}
\Prb \left(G_n(x,r,v,\delta,\zeta)\right)\leq \Prb\left(\exists(\cE_i)_{i\in J}:\begin{array}{c} \forall i \in J\quad\cE_i\subset \E_n^d \text{ is a cutset in $\cyl(A_i,\gamma_{max})$ and }\\\sum_{i\in J} V(\cE_i)\leq (\zeta + \kappa_0\sqrt \delta)\alpha_{d-1}r^{d-1}n^{d-1}\end{array}\right)\,
\end{align}
where $\kappa_0= M(C_d+32d^2c_d)/\alpha_{d-1}$.
Let $\ep_0>0$.
We want to sum on all possible values of $$\left(\left\lceil\frac{ V(\cE_i)}{\ep_0 r^{d-1}n^{d-1}}\right\rceil\ep_0 r^{d-1}n^{d-1},i\in J\right)\,.$$ 
It is easy to check that 
$$\sum_{i\in J}\left\lceil \frac{V(\cE_i)}{\ep_0 r^{d-1}n^{d-1}}\right \rceil\ep_0 r^{d-1}n^{d-1}\leq \sum_{i\in J}V(\cE_i)+ |J|\ep_0 r^{d-1}n^{d-1}\leq ((\zeta +\kappa_0\sqrt \delta )\alpha_{d-1}+|J|\ep_0)r^{d-1}n^{d-1}\,,$$
where $|J|$ denotes the cardinality of $J$.
There are at most 
$((\zeta +\kappa_0\sqrt \delta )\alpha_{d-1}/\ep_0+|J|)^{|J|}$ possible values for the family
$$\left(\left\lceil \frac{V(\cE_i)}{\ep_0 r^{d-1}n^{d-1}}\right \rceil,i\in J\right)\,.$$ Hence, the number of possible values is finite and does not depend on $n$. Let $\mathcal{S}$ be the following set $$\cS=\left\{(\beta_i)_{i\in J}\in\sN^ {|J|}:\,\sum_{i\in J}\beta_i\ep_0 r^{d-1}n^{d-1}\leq ((\zeta +\kappa_0\sqrt \delta )\alpha_{d-1}+|J|\ep_0)r^{d-1}n^{d-1}\right\}\,.$$
By inequality \eqref{eq:conclem1}, we obtain
\begin{align}\label{eq:conclem2}
\Prb \left(G_n(x,r,v,\delta,\zeta)\right)&\leq \Prb\left(\exists(\cE_i)_{i\in J}:\begin{array}{c} \forall i \in J\quad\cE_i\subset \E_n^d \text{ is a cutset in $\cyl(A_i,\gamma_{max})$ and }\\\sum_{i\in J} V(\cE_i)\leq (\zeta + \kappa_0\sqrt \delta)\alpha_{d-1}r^{d-1}n^{d-1}\end{array}\right)\,\nonumber\\
&\leq \sum_{(\beta_i)_{i\in J }\in\cS}\Prb\left(\exists(\cE_i)_{i\in J}: \begin{array}{c}\,\forall i \in J\quad \cE_i \text{ is a cutset for $\tau_n(A_i,\gamma_{max})$ and } \\\quad \left\lceil V(\cE_i)/\ep_0r^{d-1}n^{d-1}\right\rceil=\beta_i\end{array}\right)\,.
\end{align}
Since the cylinders are all disjoint, using the independence, we have
\begin{align*}
 \forall (\beta_i)_{i\in J} \in\cS\qquad\Prb&\left(\exists(\cE_i)_{i\in J}: \,\forall i \in J\quad \cE_i \text{ is a cutset for $\tau_n(A_i,\gamma_{max})$ and } \quad \left\lceil\frac{ V(\cE_i)}{\ep_0r^{d-1}n^{d-1}}\right\rceil=\beta_i\right)
\\
&\qquad\leq \prod_{i\in J}\Prb\left(\tau_n(A_i,\gamma_{\max})\leq \beta_i\ep_0r^{d-1}{n^{d-1}}\right)\,.
\end{align*}
Thanks to theorem \ref{thm:lldtau}, it follows that 
\begin{align*}
\limsup_{n\rightarrow \infty}\frac{1}{n^{d-1}}&\log\Prb\left(\exists(\cE_i)_{i\in J}: \,\forall i \in J\quad \cE_i \text{ is a cutset for $\tau_n(A_i,\gamma_{max})$ and } \quad \left\lceil\frac{ V(\cE_i)}{\ep_0r^{d-1}n^{d-1}}\right\rceil=\beta_i\right)
\\
&\leq \sum_{i\in J}\limsup_{n\rightarrow \infty}\frac{1}{n^{d-1}}\log\Prb\left(\tau_n(A_i,\gamma_{\max})\leq \beta_i\ep_0r^{d-1}{n^{d-1}}\right)\\
&\leq-\sum_{i\in J}\cH^{d-1}(A_i)\cJ_v\left(\frac{\beta_i\ep_0 r^{d-1}}{\cH^{d-1}(A_i)}\right)\,.
\end{align*}
Using that $\cJ_v$ it is a decreasing function, inequality \eqref{eq:condrect} and the convexity of $\cJ_v$, we have for $(\beta_i)_{i\in J }\in\cS$,
\begin{align*}
\cJ_v\left(\frac{((\zeta +\kappa_0\sqrt \delta )\alpha_{d-1}+|J|\ep_0)r^{d-1}}{\alpha_{d-1}r'^{d-1}-\kappa}\right)&\leq \cJ_v\left(\frac{((\zeta +\kappa_0\sqrt \delta )\alpha_{d-1}+|J|\ep_0)r^{d-1}}{\cH^{d-1}(\cup_{i\in J}A_i)}\right) \\
&\leq\cJ_v\left(\frac{1}{\cH^{d-1}(\cup_{i\in J}A_i)}\sum_{i\in J}\cH^{d-1}(A_i)\left(\frac{\beta_i\ep_0 r^{d-1}}{\cH^{d-1}(A_i)}\right)\right)\\
&\leq \sum_{i\in J}\frac{\cH^{d-1}(A_i)}{\cH^{d-1}(\cup_{i\in J}A_i)}\cJ_v\left(\frac{\beta_i\ep_0 r^{d-1}}{\cH^{d-1}(A_i)}\right)\,.
\end{align*}
Combining the two previous inequalities and \eqref{eq:condrect}, we obtain
\begin{align*}
\limsup_{n\rightarrow \infty}\frac{1}{n^{d-1}}&\log\Prb\left(\exists(\cE_i)_{i\in J}: \,\forall i \in J\quad \cE_i \text{ is a cutset for $\tau_n(A_i,\gamma_{max})$ and } \quad \left\lceil\frac{ V(\cE_i)}{\ep_0r^{d-1}n^{d-1}}\right\rceil=\beta_i\right)\\&\leq -\cH^{d-1}(\cup_{i\in J}A_i)\cJ_v\left(\frac{((\zeta +\kappa_0\sqrt \delta )\alpha_{d-1}+|J|\ep_0)r^{d-1}}{\alpha_{d-1}r'^{d-1}-\kappa}\right)\\
&\leq -(\alpha_{d-1}r'^{d-1}-\kappa)\cJ_v\left(\frac{((\zeta +\kappa_0\sqrt \delta )\alpha_{d-1}+|J|\ep_0)r^{d-1}}{\alpha_{d-1}r'^{d-1}-\kappa}\right)\,.
\end{align*}
Combining this inequality with inequality \eqref{eq:conclem2} and lemma \ref{lem:estimeanalyse} gives
\begin{align*}
\limsup_{n\rightarrow \infty}\frac{1}{n^{d-1}}\log\Prb \left(G_n(x,r,v,\delta,\zeta)\right)\leq -(\alpha_{d-1}r'^{d-1}-\kappa)\cJ_v\left(\frac{((\zeta +\kappa_0\sqrt \delta )\alpha_{d-1}+|J|\ep_0)r^{d-1}}{\alpha_{d-1}r'^{d-1}-\kappa}\right)\,.
\end{align*}
Since $\cJ_v$ is a good rate function, it is lower semi-continuous, hence,
$$\liminf_{\kappa\rightarrow 0}\liminf_{\ep_0\rightarrow 0}\cJ_v\left(\frac{((\zeta +\kappa_0\sqrt \delta )\alpha_{d-1}+|J|\ep_0)r^{d-1}}{\alpha_{d-1}r'^{d-1}-\kappa}\right)\geq \cJ_v\left(\frac{(\zeta + \kappa_0\sqrt \delta)\alpha_{d-1}r^{d-1}}{\alpha_{d-1}r'^{d-1}}\right)\,.$$
As a result, we obtain
\begin{align*}
\limsup_{n\rightarrow\infty}& \frac{1}{n^{d-1}} \log\Prb (G_n(x,r,v,\delta,\zeta))\leq -\alpha_{d-1}r^{d-1}(1-\delta)^{(d-1)/2}\cJ_v\left(\frac{\zeta + \kappa_0\sqrt \delta}{(1-\delta)^{(d-1)/2}}\right)
\end{align*}
where we use that $r'=r\sqrt{1-\delta}$. By setting $g(\delta)=(1-\delta)^{(d-1)/2}$, the result follows.
\end{proof}

\subsubsection{Lower bound on the probability of lower large deviations for the maximal flow in a ball}
Let $n\geq 1$. Let $x\in\sR^d$, $v\in\sS^{d-1}$ and $r$, $\delta$, $\zeta$ be positive constants.
We first define 
$\overline G_n(x,r,v,\delta,\zeta)$ as the event that there exists a cutset $\cE_n$ in $B(x,r)\cap\Omega_n$ that cuts $\partial^+_n B(x,r,v)\cup ((\Gamma_n^1\cup \Gamma_n^ 2)\setminus \partial ^-_n B(x,r,v))$ from $\partial^-_n B(x,r,v)$ such that
 $$\cE_n\subset \cyl(\disc(x,r,v),2\delta r)$$
 and
 $$V(\cE_n)\leq \zeta \alpha_{d-1}r^{d-1}n ^{d-1}\,.$$

\begin{lem}\label{lem:phicassimple}Let $\eta>0$. Let $\delta>0$ such that
\begin{align}\label{cond:deltalemphi}
M 10d^2(1-(\sqrt{1-4\delta^2}(1-2d\delta))^{d-1}+ 4\delta \alpha_{d-2})\leq \frac{\eta}{4} \alpha_{d-1}\,.
\end{align}
\begin{itemize}[$\bullet$]
\item  Let $x\in\Omega$ and $r>0$ such that $B(x,r)\subset \Omega$. For any $v\in\sS^{d-1}$,
we have
$$\limsup_{n\rightarrow \infty}\frac{1}{n^{d-1}}\log\Prb\left(\overline G_n(x,r,v,1,\nu_G(v)+\eta)^ c\right)=-\infty\,.$$
For any $\zeta>0$,
$$
\liminf_{n\rightarrow \infty}\frac{1}{n^{d-1}}\log \Prb\left(\overline G_n(x,r,v,\delta,\zeta+\eta)\right)\geq- \alpha_{d-1}r ^{d-1}\cJ_v(\zeta)\,.
$$
\item
Let $x\in\partial ^* \Omega \cap (\Gamma^1\cup \Gamma^2)$. Let $r>0$ such that 
\[(\Omega\cap B(x,r,n_\Omega(x)))\Delta B^- (x,r,n_\Omega(x))\subset\left\{\,y\in B(x,r):|(y-x)\cdot n_\Omega(x)|\leq \delta \|y-x\|_2\,\right\}\,,\]
we have
$$\limsup_{n\rightarrow \infty}\frac{1}{n^{d-1}}\log\Prb\left(\overline G_n\left(x,r,n_\Omega(x),1,\nu_G(n_\Omega(x))+\eta\right)^ c\right)=-\infty\,.$$
For any $\zeta>0$,
$$
\liminf_{n\rightarrow \infty}\frac{1}{n^{d-1}}\log \Prb\left(\overline G_n(x,r,n_\Omega(x),\delta,\zeta+\eta)\right)\geq- \alpha_{d-1}r ^{d-1}\cJ_{n_\Omega(x)}(\zeta)\,.
$$
\end{itemize}
\end{lem}

\begin{proof}[Proof of lemma \ref{lem:phicassimple}] Let $\eta>0$ and $\delta$ we will choose later depending on $\eta$.

\noindent {\bf First case: $x\in\Omega$.}
Let $v\in\sS^{d-1}$ and $r>0$.
 We set $h_{max}=2\delta r$ and
$$r'=r\sqrt{1-4\delta^2}\,.$$
With such a choice of $r'$, we have $\cyl(\disc(x,r',v),h_{max})\subset B(x,r)$.
Let $(e_1,\dots,e_{d-1},v)$ be an orthonormal basis. 
Let $S_{\delta}$ be the hyper-square of side-length $\delta r$ of normal vector $v$ having for expression $[0,\delta r[^{d-1}\times\{0\}$ in the basis  $(e_1,\dots,e_{d-1},v)$. We can pave $\disc(x,r',v)$ with a family $(S_i)_{i\in I }$ of translates of $S_\delta$ such that the $S_i$ are pairwise disjoint, there are all included in $\disc(x,r',\vv)$ and $\disc(x,r'(1-2d\delta) ,v)\subset\disc(x,r'-d\delta r ,v)\subset \cup_{i \in I }S_i$. We denote by $\cE_n(i)$ the cutset that achieves the infimum in $\tau_n(S_i,h_{max})$. If there are several possible choices, we use a deterministic rule to break ties.
Let $\cF_0$ be the set of edges with at least one endpoint included in 
$$\cV_2\big(\disc(x,r,v)\setminus \disc(x,r'(1-2d\delta) ,v), d/n\big)\cup \bigcup_{i\in I }\cV_2(\partial S_i, d/n)\,.$$
Note that $\cH^{d-2}(\partial S_\delta)=2(d-1)(\delta r) ^{d-2}$ where $\partial S_\delta$ denotes the relative boundary, \textit{i.e.}, the boundary of $S_\delta$ in the hyperplane $\{x\in\sR^d:x\cdot v=0\}$.
Using proposition \ref{prop:minkowski}, we have for $n$ large enough 
\begin{align*}
\card(\cF_0)&\leq 4d n^d \cL^d\left( \cV_2(\disc(x,r,v)\setminus \disc(x,r'(1-2d\delta) ,v), d/n)\cup \bigcup_{i\in I }\cV_2(\partial S_i, d/n)\right)\\
&\leq 10d^2\alpha_{d-1}r^{d-1}\left(1-\left(\sqrt{1-4\delta^2}(1-2d\delta)\right)^{d-1}\right) n^{d-1} +16d^4 |I|\delta^{d-2}r^{d-2}\alpha_2 n ^{d-2}\\
&\leq 10d^2r^{d-1}\left(1-\left(\sqrt{1-4\delta^2}(1-2d\delta)\right)^{d-1}\right) n^{d-1} +16d^4 \frac{r^{d-2}}{\delta} \alpha_2\alpha_{d-1}n ^{d-2}\,.
\end{align*}
We choose $\delta$ small enough such that 
$$ M 10d^2\left(1-\left(\sqrt{1-4\delta^2}(1-2d\delta)\right)^{d-1}\right)\leq \frac{\eta}{4} \alpha_{d-1}\,.$$
Then, we choose $n$ large enough such that 
$$16d^4 M\frac{r^{d-2}}{\delta} \alpha_2\alpha_{d-1}n ^{d-2}\leq \frac{\eta}{4} \alpha_{d-1}r^{d-1} n^{d-1}\,.$$
With this choice, we have
$$M\card(\cF_0)\leq\frac{ \eta}{2}\alpha_{d-1}r^{d-1}n^{d-1}\,.$$

We notice that the set $\cF_0\cup \left(\cup_{i\in I}\cE_n(i)\right)$ is a cutset that cuts $\partial ^+_n B(x,r,v)$ from $\partial ^-_n B(x,r,v)$ in 
$B(x,r)\cap \Omega_n$.
It follows that
\begin{align*}
\Prb\left(\overline G_n(x,r,v,1,\nu_G(v)+\eta)^c\right) &\leq\Prb\left(V\left(\cF_0\cup \left(\cup_{i\in I}\cE_n(i)\right)\right)\geq (\nu_G(v)+\eta)\alpha_{d-1}r ^{d-1}n^{d-1}\right)
\\&\leq \Prb\left(\sum_{i\in I }\tau_n(S_i,h_{max})\geq |I|\left(\nu_G(v)+\frac{\eta}{2}\right)(\delta r) ^{d-1}n^{d-1}\right)\\
&\leq \sum_{i\in I}\Prb\left(\tau_n(S_i,h_{max})\geq \left(\nu_G(v)+\frac{\eta}{2}\right)(\delta r) ^{d-1}n^{d-1}\right)\,.
\end{align*} 
By theorem \ref{thm:theretuppertau}, it follows that 
$$\limsup_{n\rightarrow \infty}\frac{1}{n^{d-1}}\log\Prb\left(\overline G_n(x,r,v,1,\nu_G(v)+\eta)^ c\right)=-\infty\,.$$

Let $\zeta>0$.
The set $\cF_n=\cF_0\cup \left(\cup_{i\in I}\cE_n(i)\right)$ is a cutset that cuts $\partial ^+_n B(x,r,v)$ from $\partial ^-_n B(x,r,v)$ in 
$B(x,r)$ and is contained in $\cyl(\disc(x,r,v),2\delta r)$.
On the event $\cap_{i\in I}\{\tau_n(S_i,h_{max})\leq \zeta (\delta r )^{d-1}n^{d-1}\}$, we have
\begin{align*}
V(\cF_n)&\leq M\card(\cF_0)+ \sum_{i\in I}\tau_n(S_i,h_{max})\\
&\leq \eta \alpha_{d-1} r^{d-1}n^{d-1}+ |I|\zeta (\delta r )^{d-1}n^{d-1}\\
&\leq (\zeta +\eta)\alpha_{d-1}r^{d-1}n^{d-1}\,
\end{align*}
and the event $\overline G _n (x,r,v,\delta,\zeta +\eta)$ occurs.
Besides, we have using the independence,
\begin{align*}
\Prb\left(\cap_{i\in I}\{\tau_n(S_i,h_{max})\leq \zeta (\delta r )^{d-1}n^{d-1}\}\right)=\prod_{i\in I}\Prb(\tau_n(S_i,h_{max})\leq \zeta (\delta r )^{d-1}n^{d-1})\,.
\end{align*}
It follows by theorem \ref{thm:lldtau} that 
\begin{align*}
\liminf_{n\rightarrow \infty}&\frac{1}{n^{d-1}}\log \Prb(\overline G_n(x,r,v,\delta,\zeta+\eta))\\
&\geq \sum_{i\in I}\liminf_{n\rightarrow \infty}\frac{1}{n^{d-1}}\log \Prb(\tau_n(S_i,h_{max})\leq \zeta (\delta r )^{d-1}n^{d-1})\\
&\geq -\frac{\alpha_{d-1}r ^{d-1}}{(\delta r) ^{d-1}}(\delta r) ^{d-1}\cJ_v(\zeta)=- \alpha_{d-1}r ^{d-1}\cJ_v(\zeta)\,.
\end{align*}
The result follows. 

\noindent {\bf Second case: $x\in\partial^ *\Omega$.} We will only prove the case where $x\in \Gamma^1$ since the proof for $x\in\Gamma^2$ is similar.
The proof is similar to the first case with extra technical difficulties. In particular, we cannot pave $\disc(x,r,n_{\Omega}(x))$ directly because the cutset may exit $\Omega$ since $B(x,r)\not\subset \Omega$. To fix this issue, we are going to  move this cylinder slightly in the direction $-n_{\Omega(x)}$.
It is easy to check that 
\begin{align*}
\left\{\,y\in B(x,r):|(y-x)\cdot n_\Omega(x)|\leq \delta\|y-x\|_2\,\right\}&\subset \left\{\,y\in B(x,r):|(y-x)\cdot n_\Omega(x)|\leq \delta r\,\right\}\\&\subset \cyl(\disc(x,r,n_\Omega(x)),\delta r)\,.
\end{align*}
Set $$x'= x-\frac{3}{2}\delta r n_\Omega(x)\,.$$
Using that $\{y\in B(x,r): (y-x)\cdot n_\Omega(x)\leq -\delta r\}\subset \Omega$ (see figure \ref{fig:omeg}), we have
$$\cyl(\disc(x',r',n_\Omega(x)),\delta r/2)\subset \cyl(\disc(x,r',n_\Omega(x)),2\delta r)\cap \Omega\subset B(x,r)\cap \Omega\,.$$
\begin{figure}[H]
\begin{center}
\def\svgwidth{0.8\textwidth}
   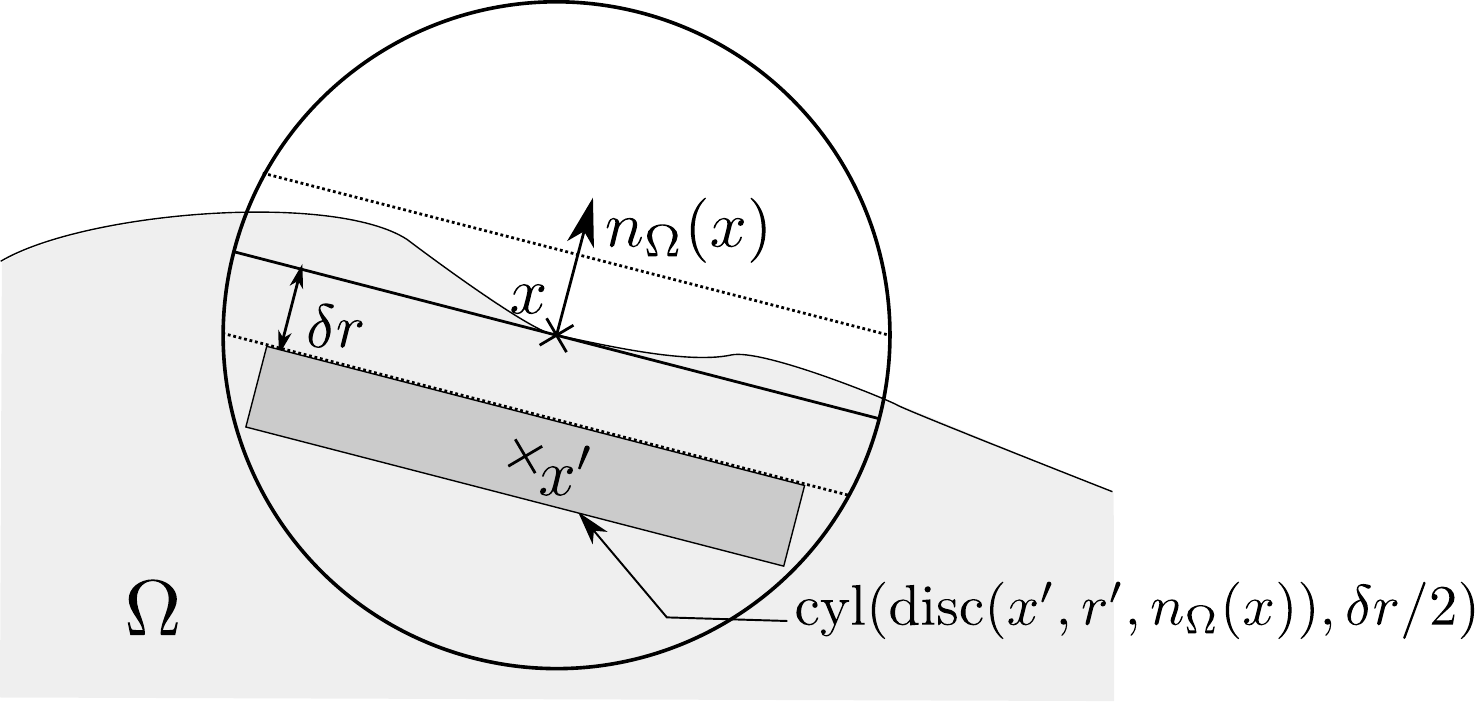
   \caption{\label{fig:omeg}The case $x\in\partial ^* \Omega$.}
   \end{center}
\end{figure}

Let $(e_1,\dots,e_{d-1},n_\Omega(x))$ be an orthonormal basis. 
Let $S_{\delta}$ be the hyper-square of side-length $\delta r$ of normal vector $n_\Omega(x)$ having for expression $[0,\delta r[^{d-1}\times\{0\}$ in the basis  $(e_1,\dots,e_{d-1},n_\Omega(x))$. We can pave $\disc(x',r',n_\Omega(x))$ with a family $(S_i)_{i\in I }$ of translates of $S_\delta$ such that the $S_i$ are pairwise disjoint, there are all included in $\disc(x',r',n_\Omega(x))$ and $\disc(x',r'(1-2d\delta) ,n_\Omega(x))\subset \cup_{i \in I }S_i$. We denote by $\cE_n(i)$ the cutset that achieves the infimum in $\tau_n(S_i,h_{max}/4)$. If there are several possible choices, we use a deterministic rule to break ties.
Let $\cF_0$ be the set of edges with at least one endpoint included in 
\begin{align*}
\cV_2\big(\disc(x',r,n_\Omega(x))\setminus \disc(x',r'(1-2d\delta) ,n_\Omega(x)), d/n\big)\cup \left(\bigcup_{i\in I }\cV_2(\partial S_i, d/n)\right)\\\hfill\cup \left(\cV_2(\partial B(x,r) \cap \cyl(\disc(x,r,n_\Omega(x)),2\delta r),d/n)\right)\,.
\end{align*}
Note that
\begin{align*}
\cH^{d-1}\left(\partial B(x,r) \cap \cyl(\disc(x,r,n_\Omega(x)),2\delta r)\right)&= 2\int _0^{2\delta r}\cH^{d-2}\left(\partial B(x,r)\cap \disc(x+u\,n_{\Omega}(x),r,n_{\Omega(x)})\right)du\\
&\leq 4\delta \alpha_{d-2}r^{d-1}\,.
\end{align*}
By the same computations as in the previous case, we have
\[\card(\cF_0)\leq 10d^2\left(1-\left(\sqrt{1-4\delta^2}(1-2d\delta)\right)^{d-1}+4\delta \alpha_{d-2}\right)n^{d-1}+ 16d^4 \frac{r^{d-2}}{\delta} \alpha_2\alpha_{d-1}n ^{d-2}\,.\]
We choose $\delta$ small enough such that 
$$ M 10d^2\left(1-\left(\sqrt{1-4\delta^2}(1-2d\delta)\right)^{d-1}+4\delta \alpha_{d-2}\right)\leq \frac{\eta}{4} \alpha_{d-1}\,.$$
Then, we choose $n$ large enough such that 
$$16d^4 M\frac{r^{d-2}}{\delta} \alpha_2\alpha_{d-1}n ^{d-2}\leq \frac{\eta}{4} \alpha_{d-1}r^{d-1} n^{d-1}\,.$$
With this choice
$$M\card(\cF_0)\leq \frac{\eta}{2}\alpha_{d-1}r^{d-1}n^{d-1}\,.$$
Set $\cF_n=\cF_0 \cup\left(\cup_{i \in I}\cE_n(i)\right)$.
The set $\cF_n$ may not be included in $\Pi_n$, but we claim that $\cF_n\cap \Pi_n$ is a cutset that cuts  $\partial^+_n B(x,r,n_\Omega(x))\cup (\Gamma_n^1\setminus \partial ^-_n B(x,r,n_\Omega(x)))$ from $\partial ^-_n B(x,r,n_\Omega(x))$ in $B(x,r)\cap \Omega_n$. 
Let $\gamma$ be a path from $\partial^+_n B(x,r,n_\Omega(x))\cup (\Gamma_n^1\setminus \partial ^-_n B(x,r,n_\Omega(x)))$ to $\partial ^-_n B(x,r,n_\Omega(x))$ in $\Omega_n\cap B(x,r)$.
We write $\gamma=(x_0,e_1,x_1\dots,e_p,x_p)$. If $x_0\in\Gamma_n^1$, since $\Gamma_n^1\cap B(x,r)\subset \cyl(\disc(x,r,n_\Omega(x)),\delta)$, then we have $x_0\cdot n_\Omega(x)\geq -\delta r$.  If $x_0\in\partial^+_n B(x,r,n_\Omega(x))$, then we have $x_0\cdot n_\Omega(x)\geq 0$

$\triangleright$ Let us assume $x_p\cdot n_\Omega(x)\geq -\delta r$. Since $x_p\in \partial ^-_n B(x,r,n_\Omega(x))$, we have $x_p\cdot n_\Omega(x)\leq 0$. It follows that $x_p\in \cV_2(\partial B(x,r) \cap \cyl(\disc(x,r,n_\Omega(x)),\delta r),d/n)$ and $\langle x_{p-1},x_p\rangle\in\cF_0$.

$\triangleright$ Let us assume $x_p\cdot n_\Omega(x)< -\delta r$. If $\gamma\cap \cV_2\big(\disc(x',r,n_\Omega(x))\setminus \disc(x',r'(1-2d\delta) ,n_\Omega(x)), d/n\big)\cup \bigcup_{i\in I }\cV_2(\partial S_i, d/n)\neq \emptyset$ then $\gamma\cap\cF_0\neq\emptyset$. Let us assume that $\gamma\cap \cV_2\big(\disc(x',r,n_\Omega(x))\setminus \disc(x',r'(1-2d\delta) ,n_\Omega(x)), d/n\big)\cup \bigcup_{i\in I }\cV_2(\partial S_i, d/n)=\emptyset$. Then, there exists $i\in I$ such that there exists an excursion of $\gamma$ in $\cyl(S_i,\delta r/2)$ from the bottom half to the top half of the cylinder. It follows that $\gamma\cap\cE_n(i)\neq \emptyset$.
The set $\cF_n$ is indeed a cutset.
We conclude similarly as in the first case.
\end{proof}

Lemma \ref{lem:phicassimple} can only be applied for radius that does not depend on $n$. We will work with balls with radius $r'>0$ depending on $n$. The following lemma enables to use the lemma \ref{lem:phicassimple} when the radius $r'$ depends on $n$ but is close to some fixed $r>0$.
\begin{lem} \label{lem:inclusion}Let $0<\ep\leq 1/4$. There exists $\kappa_d\geq 1$ a constant depending only on $d$ such that
\begin{itemize}[$\bullet$]
\item  For any $\delta>0$, for any $x\in\Omega$ and $r>0$ such that $B(x,r)\subset \Omega$, for any $v\in\sS^{d-1}$, for $n$ large enough depending on $r$ and $\ep$, we have
$$\forall r'\in \left[(1-\sqrt \ep)r,r\right]\qquad\overline G_n(x,(1-\sqrt\ep)r,v,\delta,\zeta)\subset \overline G_n(x,r',v,\delta,\zeta+\kappa_d M\sqrt\ep) 
\,.$$
\item For any $0<\delta\leq\ep$, for any $x\in\partial ^* \Omega \cap (\Gamma^1\cup \Gamma^2)$, for any $r>0$ such that 
\begin{align}\label{eq:am1bis}
\forall& \,0<r_0\leq r\nonumber\\&(\Omega\cap B(x,r_0,n_\Omega(x)))\Delta B^- (x,r_0,n_\Omega(x))\subset\left\{\,y\in B(x,r_0):|(y-x)\cdot n_\Omega(x)|\leq \delta \|y-x\|_2\,\right\}
\end{align}
and
\begin{align}\label{eq:am2bis}
  \forall\, 0<r_0\leq r\qquad \left|\frac{1}{\alpha_{d-1} r_0^{d-1}}\cH^{d-1}(\partial^*\Omega\cap B(x,r_0))-1\right|\leq \ep\,,
\end{align}
for any $n$ large enough depending on $r$ and $\ep$, we have
$$\forall r'\in \left[(1-\sqrt \ep)r,r\right]\quad\forall \delta_0\in[\delta,1]\qquad\overline G_n(x,(1-\sqrt\ep)r,v,\delta_0,\zeta)\subset \overline G_n(x,r',v,\delta_0,\zeta+\kappa_d M\sqrt\ep) \,.
$$
\end{itemize}
\end{lem}
\begin{proof}
\noindent {\bf First case: $x\in\Omega$.}
Let us assume the event $\overline G_n(x,(1-\sqrt\ep)r,v,\delta,\zeta)$ occurs. Let $\cE_n$ be a cutset in $B(x,(1-\sqrt\ep)r)\cap\Omega_n$ that cuts $\partial^+_n B(x,(1-\sqrt\ep)r,v)$ from $\partial^-_n B(x,(1-\sqrt\ep)r,v)$ such that
 $$\cE_n\subset \cyl(\disc(x,(1-\sqrt\ep)r,v),2\delta( 1-\sqrt\ep)r)$$
 and
 $$V(\cE_n)\leq \zeta \alpha_{d-1}((1-\sqrt\ep)r)^{d-1}n ^{d-1}\,.$$
 We define $\cR_n$ to be the set of edges with at least one endpoint in $$\cV_2( \disc(x,r',v)\setminus \disc(x,r(1-\sqrt\ep),v),d/n)\,.$$
By proposition \ref{prop:minkowski}, we have for $n$ large enough
\begin{align*}
\card(\cR_n)&\leq 4d\cL^d(\cV_2( \disc(x,r',v)\setminus \disc(x,r(1-\sqrt\ep),v),d/n))n^{d}\\&\leq 10d^ 2\cH^ {d-1}\left( \disc(x,r',v)\setminus \disc(x,r(1-\sqrt\ep),v)\right)n^{d-1}\\
&\leq 10d^ 2\left(1-(1-\sqrt\ep)^{d-1}\right)\alpha_{d-1}r^{d-1}n^{d-1}\\
&\leq 10d^ 3\sqrt\ep\alpha_{d-1}r^{d-1}n^{d-1}\\
&\leq 10d^ 3\sqrt\ep\alpha_{d-1} \left(\frac{r'}{1-\sqrt\ep}\right)^{d-1}n^{d-1}\leq 10d^3 2^{d-1}\sqrt\ep\alpha_{d-1}r'^{d-1}n^{d-1}
\end{align*}
where we use that $(1-x)^{d-1}\geq 1-(d-1)x$ for any $x\in[0,1]$ and that $\ep\leq 1/4$.
 Let us prove that $\cE_n \cup \cR_n$ is a cutset in $B(x,r')\cap\Omega_n$ that cuts $\partial^+_n B(x,r',v)$ from $\partial^-_n B(x,r',v)$.
 Let $\gamma$ be a path from $\partial^+_n B(x,r',v)$ to $\partial^-_n B(x,r',v)$ in $B(x,r')$. We write $\gamma=(x_0,e_1,x_1\dots,e_p,x_p)$. The path $\gamma$ must cross $\disc(x,r',v)$ from top to bottom. If $\gamma\cap \cR_n=\emptyset$ then $\gamma $ must cross $\disc(x,(1-\sqrt\ep)r,v)$. We have $\gamma\cap\partial^-_n B(x,(1-\sqrt\ep)r,v)\neq \emptyset$ and  $\gamma\cap\partial^+_n B(x,(1-\sqrt\ep)r,v)\neq \emptyset$. We set 
$$m=\inf\{k\geq 0: x_k\in\partial^-_n B(x,(1-\sqrt\ep)r,v)\}$$
and
$$l=\sup\{k\leq m: x_k\in\partial^+_n B(x,(1-\sqrt\ep)r,v)\}\,.$$
The subpath of $\gamma$ between $l$ and $m$ remains in $B(x,(1-\sqrt\ep)r)$, it follows that $\gamma\cap \cE_n\neq \emptyset$. 
Hence, $\cE_n \cup \cR_n$ is a cutset in $B(x,r')\cap\Omega_n$ that cuts $\partial^+_n B(x,r',v)$ from $\partial^-_n B(x,r',v)$.
Moreover, we have 
 $$\cE_n\cup \cR_n\subset \cyl(\disc(x,r',v),2\delta r')$$
and $$V(\cE_n\cup\cR_n)\leq (\zeta +10d^32^{d-1}M\sqrt\ep)\alpha_{d-1}r'^{d-1}n ^{d-1}\,.$$
Hence, the event $\overline G_n(x,r',v,\delta,\zeta+\kappa_d M\sqrt\ep)$ occurs where $\kappa_d$ will be chosen in such a way that $\kappa_d \geq 10d^32^{d-1}$.

\noindent {\bf Second case: $x\in\partial^ *\Omega$.} We will only prove the case where $x\in \Gamma^1$ since the proof for $x\in\Gamma^2$ is similar.
Let us assume the event $\overline G_n(x,(1-\sqrt\ep)r,n_\Omega(x),\delta_0,\zeta)$ occurs. Let $\cE_n$ be a cutset in $B(x,(1-\sqrt\ep)r)\cap\Omega_n$ that cuts $\partial^+_n B(x,(1-\sqrt\ep)r,n_\Omega(x))\cup (\Gamma_n^ 1\setminus \partial^-_n B(x,(1-\sqrt\ep)r,n_\Omega(x)))$ from $\partial^-_n B(x,(1-\sqrt\ep)r,n_\Omega(x))$ such that
 $$\cE_n\subset \cyl(\disc(x,(1-\sqrt\ep)r,n_\Omega(x)),2\delta_0 (1-\sqrt\ep)r)$$
 and
 $$V(\cE_n)\leq \zeta \alpha_{d-1}((1-\sqrt\ep)r)^{d-1}n ^{d-1}\,.$$
We define $\cR_n $ to be the set of edges with at least one endpoint in 
\begin{align*}
&\cV_2( (B(x,r')\setminus B(x,r(1-\sqrt\ep)))\cap \partial^*\Omega ,d/n)\\&\quad\cup \cV_2\left( (\partial B(x,r')\cup \partial B(x,r(1-\sqrt\ep)))\cap \cyl(\disc(x,r',n_\Omega(x)),2\delta r') ,d/n\right) \,,
\end{align*}
(see figure \ref{fig:repEnRn}).
\begin{figure}[H]
\begin{center}
\def\svgwidth{0.6\textwidth}
   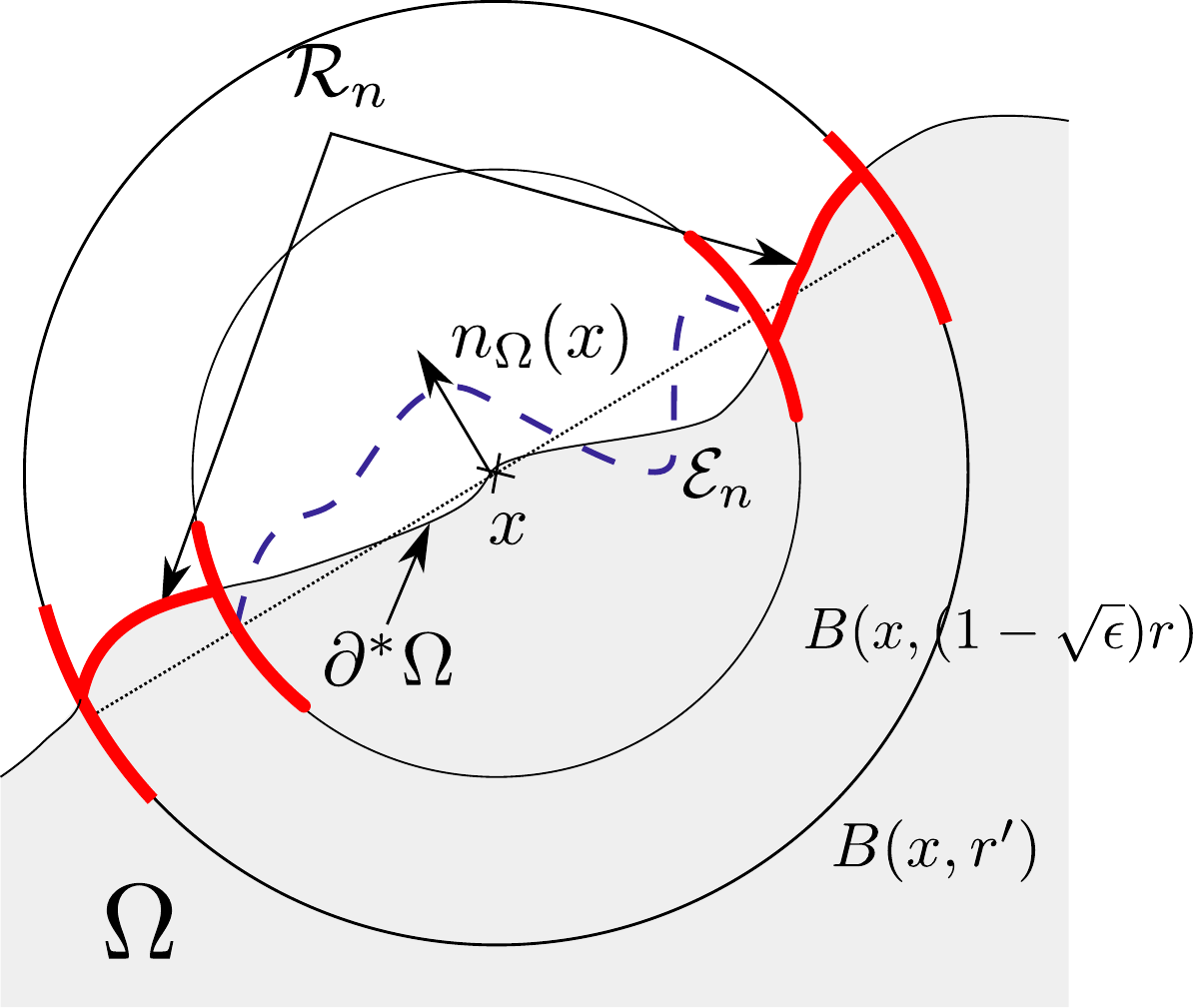
   \caption{\label{fig:repEnRn}Representation of the sets $\cE_n$ and $\cR_n$.}
   \end{center}
\end{figure}
\noindent By proposition \ref{prop:minkowski} and \eqref{eq:am2bis}, we have
\begin{align*}
\card(\cR_n)&\leq 10d^ 2\left(\cH^ {d-1}\left( B(x,r')\cap\partial^*\Omega\right)- \cH^ {d-1}\left( B(x,r(1-\sqrt\ep))\cap\partial^*\Omega\right)+8\delta \alpha_{d-2}r^ {d-1}\right)n^{d-1}\\
&\leq 10d^ 2\left((1+\ep) -(1-\ep)(1-\sqrt\ep)^{d-1}+8\delta \frac{\alpha_{d-2}}{\alpha_{d-1}} \right)\alpha_{d-1}r^{d-1}n^{d-1}\\
&\leq 10d^ 2\left(2\ep + (d-1)\sqrt\ep +8\ep \frac{\alpha_{d-2}}{\alpha_{d-1}} \right)\alpha_{d-1}r^{d-1}n^{d-1}\\
&\leq 10d^ 22^{d-1}\sqrt\ep\left(d+1 +8 \frac{\alpha_{d-2}}{\alpha_{d-1}} \right)\alpha_{d-1}r'^{d-1}n^{d-1}
\end{align*}
where we recall that $\delta\leq \ep$ and $\ep\leq 1/4$.
Let us prove that $\cE_n \cup \cR_n$ is a cutset in $B(x,r')\cap\Omega_n$ that cuts $\partial^+_n B(x,r',n_\Omega(x))\cup (\Gamma_n^ 1\setminus \partial^-_n B(x,r',n_\Omega(x)))$ from $\partial^-_n B(x,r',n_\Omega(x))$ (see figure \ref{fig:repEnRn}).
 Let $\gamma$ be a path from $\partial^+_n B(x,r',n_\Omega(x))\cup (\Gamma_n^ 1\setminus \partial^-_n B(x,r',n_\Omega(x)))$ to $\partial^-_n B(x,r',n_\Omega(x))$ in $B(x,r')\cap\Omega_n$. We write $\gamma=(x_0,e_1,x_1\dots,e_p,x_p)$. 
 
 $\triangleright$ Let us assume that $x_0\in \Gamma_n^1$. If $x_0\in \Gamma_n^1\setminus B(x,(1-\sqrt\ep)r)$ then $x_0\in \cV_2( (B(x,r')\setminus B(x,r(1-\sqrt\ep)))\cap \partial^*\Omega ,d/n)$ and $\gamma\cap\cR_n\neq\emptyset$. If $x_0\in \Gamma_n^1\cap  B(x,(1-\sqrt\ep)r)$, set 
 $$m=\inf\{k\geq 0: x_k\in\partial_n B(x,(1-\sqrt\ep)r)\}\,.$$
 If $x_m\in \partial^+_nB(x,(1-\sqrt\ep)r,n_\Omega(x))$, then $$x_m \in \Omega_n \cap B^+(x,r',n_{\Omega}(x))\subset \cV_2(\Omega \cap B^+(x,r',n_{\Omega}(x)),d/n)$$ and by \eqref{eq:am1bis} we have $$x_m\in \cV_2( \partial B(x,r(1-\sqrt\ep))\cap \cyl(\disc(x,r',n_\Omega(x)),\delta r') ,d/n)\,.$$ Hence, $\gamma\cap\cR_n\neq\emptyset$. If  $x_m\in \partial^-_nB(x,(1-\sqrt\ep)r,n_\Omega(x))$, then $\cE_n\cap \gamma\neq\emptyset$ by definition of the cutset $\cE_n$.
 
  $\triangleright$ Let us assume that $x_0\in \partial^+_n B(x,r',n_\Omega(x))\cap \Omega_n$, then $$x_0 \in \Omega_n \cap B^+(x,r',n_{\Omega}(x))\subset \cV_2(\Omega \cap B^+(x,r',n_{\Omega}(x)),d/n)$$ and by \eqref{eq:am1bis}, we have $x_0\in\cV_2( \partial B(x,r'(1-\sqrt\ep))\cap \cyl(\disc(x,r',n_\Omega(x)),\delta r') ,d/n)$. It follows that $\gamma\cap\cR_n\neq\emptyset$,

Hence, $\cE_n \cup \cR_n$ is a cutset in $B(x,r')\cap\Omega_n$ that cuts $\partial^+_n B(x,r',n_\Omega(x))\cup  (\Gamma_n^ 1\setminus \partial^-_n B(x,r',n_\Omega(x)))$ from $\partial^-_n B(x,r',n_\Omega(x))$.
Moreover, we have 
 $$\cE_n\cup \cR_n\subset \cyl(\disc(x,r',n_\Omega(x)),2\delta_0 r')$$
and $$V(\cE_n\cup\cR_n)\leq (\zeta +\kappa_d M\sqrt\ep)\alpha_{d-1}r'^{d-1}n ^{d-1}\,$$
where $$\kappa_d=10d^ 22^{d-1}\left(d+1 +8 \frac{\alpha_{d-2}}{\alpha_{d-1}} \right)\,.$$
The event $\overline G_n(x,r',n_\Omega(x),\delta_0,\zeta+\kappa_d M\sqrt\ep)$ occurs.

\end{proof}

\section{Properties of the limiting objects}\label{sec:admissible} 
This section corresponds to the step 1 of the sketch of the proof in the introduction. The aim is to prove that the limiting objects must be measures supported on surfaces, in particular, that they are contained in the set $\fT_\cM$. We prove that if $(E,\nu)$ is not in $\fT_\cM$, then it is very unlikely that there exists an almost minimal cutset in a small neighborhood of $(E,\nu)$. We prove this result by contradiction, if we have a lower bound on the probability of being in any neighborhood of $(E,\nu)$, then $(E,\nu)$ must be in $\fT_\cM$. To prove such a result, we prove that $(E,\nu)$ satisfies all the required properties to be in the set $\fT_\cM$.
\begin{prop}\label{prop:admissible}Let $(E,\nu)\in \cB(\sR^d)\times\cM(\sR^ d)\setminus \fT_{\cM}$.
Then, for any $K>0$, there exists a neighborhood $U$ of $(E,\nu)$ such that
$$\lim_{\ep\rightarrow 0}\limsup_{n\rightarrow \infty}\frac{1}{n^{d-1}}\log \Prb(\exists \cE_n\in\sC_n(\ep):\,(\bR(\cE_n),\mu_n(\cE_n))\in U)\leq -K\,.$$
\end{prop}

\begin{proof}
Let $(E,\nu)\in \cB(\sR^d)\times\cM(\sR^d)$. We assume that there exists $K>0$ such that for any neighborhood $U$ of $(E,\nu)$ that
\begin{align*}
\lim_{\ep\rightarrow 0}\limsup_{n\rightarrow \infty}\frac{1}{n^{d-1}}\log \Prb(\exists \cE_n\in\sC_n(\ep):\,(\bR(\cE_n),\mu_n(\cE_n))\in U)\geq -K\,.
\end{align*}
It is clear that the map $\ep\mapsto\Prb(\exists \cE_n\in\sC_n(\ep):\,(\bR(\cE_n),\mu_n(\cE_n))\in U)$ is non-decreasing. Hence, we get for any neighborhood $U$ of $(E,\nu)$
\begin{align}\label{eq:Enu}
\forall \ep>0\qquad\limsup_{n\rightarrow \infty}\frac{1}{n^{d-1}}\log \Prb(\exists \cE_n\in\sC_n(\ep):\,(\bR(\cE_n),\mu_n(\cE_n))\in U)\geq -K\,.
\end{align}
We aim to prove that $(E,\nu)\in \fT_{\cM}$.
We check one after the other that all the required properties are satisfied.
\noindent {\bf Step 1. We prove that $\cL^d(E\setminus \Omega)=0$,  $\nu((\overline{\Omega})^c)=0$ and $\nu(\overline{\Omega})\leq 10d^2M \cH^{d-1}(\Gamma^1)$.} Let $\delta>0$. Let $f$ be a continuous function compactly supported with its support included in $(\overline{\Omega})^c$. Let $g$ be a continuous function compactly supported that takes its values in $[0,1]$ such that $g=1$ on $\cV_2(\Omega,1)$. We set 
$$U=\left\{\,F\in\cB(\sR^d): \dis(F,E)\leq \delta \,\right\}\times \left\{\mu\in\cM(\sR^d):|\mu(f)-\nu(f)|\leq\delta, \,|\nu(g)-\mu(g)|\leq \delta\right\}\,.$$
Fix $\ep>0$. Thanks to inequality \eqref{eq:Enu}, there exists an increasing sequence $(a_n)_{n\geq 1}$ such that
$$\Prb(\exists \cE_{a_n}\in\sC_{a_n}(\ep):\,(\bR(\cE_{a_n}),\mu_{a_n}(\cE_{a_n}))\in U)>0\,.$$
For short, we will write $n$ instead of $a_n$.
On the event $\{\exists \cE_n\in\sC_n(\ep):\,(\bR(\cE_n),\mu_n(\cE_n))\in U\}$, we consider
 $\cE_n\in\sC_n(\ep)$ such that $(\bR(\cE_n),\mu_n(\cE_n))\in U$ (if there are several possible choices, we pick one according to a deterministic rule). It follows that for $n$ large enough, using proposition \ref{prop:minkowski}
\begin{align*}
\cL^d(E\setminus \Omega)\leq \cL^d (\bR(\cE_n)\setminus \Omega)+ \cL^d(E\setminus \bR(\cE_n))\leq \cV_2\left(\partial \Omega,\frac{d}{n}\right)+\cL^d(\bR(\cE_n)\Delta E)\leq 2\delta
\end{align*} 
and 
\begin{align}\label{eq:star1} 
 \nu(g)\leq \mu_n(\cE_n)(g)+\delta\,.
 \end{align}
Note that since the support of $f$ is included in $\overline \Omega ^c$, for $n$ large enough, for any $e\in\Pi_n$ we have $f(c(e))=0$. It follows that $\mu_n(\cE_n)(f)=0$ and
\begin{align*}
 \nu(f)\leq \mu_n(\cE_n)(f)+\delta\leq \delta\,.
 \end{align*}
Set $$\cF_n=\{\langle x,y \rangle \in\E_n^d: x\in\Gamma_n^1, y\in\Omega_n\}\,.$$
It is clear that $\cF_n\in\sC_n(\Gamma^1,\Gamma^2,\Omega)$.
Thanks to proposition \ref{prop:minkowski}, we have for $n$ large enough
\begin{align*}
V(\cF_n)\leq 2d M \card(\Gamma_n^1)\leq 2dM\cL^d(\cV_2(\Gamma^1, 2d/n))n^d\leq 10d^2M \cH^{d-1}(\Gamma^1)n^{d-1}\,.
\end{align*}
 Besides, by definition of $\sC_n(\ep)$ we have 
\begin{align}\label{eq:Vencont}
V(\cE_n)\leq V(\cF_n) +\ep n^{d-1}\leq 10d^2M \cH^{d-1}(\Gamma^1)n^{d-1}+\ep n^{d-1}\,
\end{align}
 and so
\begin{align}\label{eq:star2}
\mu_n(\cE_n)(g)\leq \frac{V(\cE_n)}{n^{d-1}}\leq 10d^2M \cH^{d-1}(\Gamma^1)+\ep\,.
\end{align}
Consequently, we get using \eqref{eq:star1} and \eqref{eq:star2}
$$\nu(\overline{\Omega})\leq\nu(g)\leq 10d^2M \cH^{d-1}(\Gamma^1)+\delta+\ep\,.$$
Finally, by letting $\delta$ and $\ep$ go to $0$, we obtain that $$\cL^d(E\setminus \Omega)=0\,,\qquad\nu(\overline{\Omega})\leq 10d^2M \cH^{d-1}(\Gamma^1) $$ and $\nu(f)=0$. Since $\nu(f)=0$ for any continuous function having its support included in $(\overline{\Omega})^c$, we conclude that $$\nu((\overline{\Omega})^c)=0\,.$$

\noindent {\bf Step 2. We prove that $\cP(E,\Omega)<\infty$.}  We will need the following lemma that is an adaptation of theorem 1 by Zhang in \cite{Zhang2017}. We postpone its proof after the proof of this proposition.
\begin{lem}\label{lem:Zhang} There exists $C_1>0$ depending on $G$ and $\Omega$ such that for any $K>0$, there exists $\beta$ depending on $K$, $G$ and $\Omega$ such that for all $n\geq 1$
$$\Prb(\exists \cE_n\in \sC_n(\Gamma^1,\Gamma^ 2,\Omega):\, V(\cE_n)\leq 11d^2M \cH^{d-1}(\Gamma^1) n^{d-1},\,\card(\cE_n)\geq \beta n ^{d-1})\leq C_1 \exp(-Kn^{d-1})\,.$$
\end{lem}
Let $\beta$ be such that for all $n\geq 1$
\begin{align}\label{eq:star3}
\Prb(\exists \cE_n\in \sC_n(\Gamma^1,\Gamma^ 2,\Omega):\, V(\cE_n)\leq 11d^2M \cH^{d-1}(\Gamma^1) n^{d-1},\,\card(\cE_n)\geq \beta n ^{d-1})\leq C_1\exp(-2Kn^{d-1})\,.
\end{align}
Let $U_0$ be a neighborhood of $E$ and $U_1$ be a neighborhood of $\nu$. We have for $\ep$ small enough using inequality \eqref{eq:Vencont}
\begin{align*}
\Prb&(\exists \cE_n\in\sC_n(\ep):\,(\bR(\cE_n),\mu_n(\cE_n))\in U_0\times U_1)\\&\qquad\leq \Prb(\exists \cE_n\in\sC_n(\ep):\,\bR(\cE_n)\in U_0, \,\card(\cE_n)\leq \beta n ^{d-1})\\
&\qquad\quad+\Prb(\exists \cE_n\in \sC_n(\Gamma^1,\Gamma^ 2,\Omega):\, V(\cE_n)\leq 11d^2M \cH^{d-1}(\Gamma^1) n^{d-1},\,\card(\cE_n)\geq \beta n ^{d-1})\,.
\end{align*}
Combining the previous inequality, inequalities \eqref{eq:Enu} and \eqref{eq:star3} and lemma \ref{lem:estimeanalyse}, we get
\begin{align*}
\limsup_{n\rightarrow \infty}\frac{1}{n^{d-1}}\log &\Prb(\exists \cE_n\in \sC_n(\Gamma^1,\Gamma^ 2,\Omega):\,\bR(\cE_n)\in U_0, \,\card(\cE_n)\leq \beta n ^{d-1})\\&\geq\limsup_{n\rightarrow \infty}\frac{1}{n^{d-1}}\log \Prb(\exists \cE_n\in\sC_n(\ep):\,(\bR(\cE_n),\mu_n(\cE_n))\in U_0\times U_1)\geq -K\,.
\end{align*}
For $n\geq 1$, we set $V_n=\{F\in\cB(\sR^d):\dis(F,E)\leq 1/n)$.
We can build an increasing sequence of integers $(a_n)_{n\geq 1}$ such that 
$$\Prb(\exists \cE_{a_n} \text{ $(\Gamma_{a_n}^1\cup\Gamma_{a_n}^2)$-cutset such that $\bR (\cE_{a_n})\in V_n$ and $\card(\cE_{a_n})\leq \beta {a_n} ^{d-1}$})>0\,.$$
We can choose a deterministic sequence of cutsets $(\cE_{a_n})_{n\geq1}$ such that for any $n\geq1$, $\dis(E,\bR(\cE_{a_n}))\leq 1/n$. Besides, it is easy to check that $$\cP(\bR(\cE_{a_n}),\Omega)\leq \frac{\card(\cE_{a_n})}{n^{d-1}}\leq\beta\,.$$ Since the map $F\mapsto \cP(F,\Omega)$ is lower semi-continuous for the topology induced by $\dis$, we get
$$\cP(E,\Omega)\leq \limsup_{n\rightarrow\infty} \cP(\bR(\cE_{a_n}),\Omega)\leq \beta<\infty\,.$$

\noindent {\bf Step 3. We prove that for any $x\in\fE$ $\limsup_{r\rightarrow 0} \frac{\nu(B(x,r))}{\alpha_{d-1}r^{d-1}}\leq \nu_G(n_\bullet(x))$.} Here $\bullet$ corresponds to $E$ or $\Omega$ depending on $x$.

{\noindent $\bullet$ \bf Case $x\in \partial^* E\cap \Omega$.} 
Let $\delta\in]0,1[$. Let $r_0>0$ be small enough depending on $\delta$ such that
\begin{align}\label{eq:choixr}
\forall 0<r\leq r_0\qquad\cL^d((E\cap B(x,r))\Delta B^-(x,r,n_E(x)))\leq \delta \alpha_d r^d\,.
\end{align}
Up to choosing a smaller $r_0$, we can assume that $B(x,r_0)\subset\Omega$ (we recall that $\Omega$ is open). Let $r\leq r_0/2$.
Let $\ep>0$. Since $\nu(B(x,r))<\infty$ there exists a continuous function $f$ taking its values in $[0,1]$ with support included in $B(x,r)$ such that 
\begin{align}\label{eq:nuf}
\nu(f)\geq \nu(B(x,r))-\ep \alpha_{d-1}r^{d-1}\,.
\end{align}
Set $$U=\left\{F\in\cB(\sR^d): \cL^d(F\Delta E)\leq \delta\alpha_d r ^d\right\}\times\left\{\mu\in\cM(\sR^d):|\mu(f)-\nu(f)|\leq \ep\alpha_{d-1}r ^{d-1}\right\}\,.
$$
Let $\ep_0>0$. Thanks to inequality \eqref{eq:Enu}, there exists an increasing sequence $(a_n)_{n\geq 1}$ such that
$$\Prb(\exists \cE_{a_n}\in\sC_{a_n}(\ep_0):\,(\bR(\cE_{a_n}),\mu_{a_n}(\cE_{a_n}))\in U)>0\,.$$
For short, we will write $n$ instead of $a_n$.
On the event $\{\exists \cE_n\in\sC_n(\ep_0):\,(\bR(\cE_n),\mu_n(\cE_n))\in U\}$, we choose $\cE_n\in\sC_n(\ep_0)$ such that $(\bR(\cE_n),\mu_n(\cE_n))\in U$ (if there are several possible choices, we pick one according to a deterministic rule).
To shorten the notation, we write $\mu_n$ for $\mu_n(\cE_n)$ and $E_n$ for $\bR(\cE_n)$.
Using inequality \eqref{eq:nuf}, we have 
\begin{align}\label{eq:ineqmun}
\mu_n(B(x,r))\geq \mu_n(f)\geq \nu(f)-\ep\alpha_{d-1}r ^{d-1}\geq \nu(B(x,r))-2\ep\alpha_{d-1}r ^{d-1}\,.
\end{align}
Set $r'=(1+\sqrt{\delta})r$. Note that $r'\leq 2r\leq r_0$. Thanks to inequality \eqref{eq:choixr}, we have
\begin{align*}
\cL^d((E_n\cap B(x,r'))\Delta B^-(x,r',n_E(x)))&\leq \cL^d((E\cap B(x,r'))\Delta B^-(x,r',n_E(x)))+\cL^d(E\Delta E_n)\nonumber\\&\leq 2\delta \alpha_d r'^d
\end{align*}
and for $n$ large enough 
\begin{align}\label{eq:difsym}
\card\left((E_n\cap B(x,r'))\Delta B^-(x,r',n_E(x))\cap\sZ_n^ d\right)\leq 4\delta \alpha_d r'^d\,.
\end{align}
 For each $\gamma$ in $\{1/n,\dots , (\lfloor nr\sqrt{\delta}\rfloor -1)/n\}$, we define
\[B(\gamma)=B(x,r'-\gamma,n_E(x))\,.\]
The sets $\partial^+_n B(\gamma)\cup\partial ^-_nB(\gamma)$ are pairwise disjoint for different $\gamma$.
Moreover, using inequality \eqref{eq:difsym}, we have for $n$ large enough
\begin{align*}
\sum_{\gamma=1/n,\dots,(\lfloor nr\sqrt{\delta}\rfloor -1)/n}\card(E_n\cap \partial^+_n B(\gamma))+\card(E_n^c\cap \partial^-_n B(\gamma))&\leq \card(E_n\Delta (B^-(x,r',v)\cap\sZ_n^d)\\&\leq  4\delta \alpha_d r'^d n^d\,.
\end{align*}
By a pigeon-hole principle, there exists $\gamma_0$ in $\{1/n,\dots,(\lfloor nr\sqrt{\delta}\rfloor -1)/n\}$ such that for $n$ large enough
\begin{align}\label{eq:contxgamma0}
\card(E_n\cap \partial^+_n B(\gamma_0))+\card(E_n^c\cap \partial^-_n B(\gamma_0))\leq \frac{4\delta \alpha_d r'^d n^d}{\lfloor nr\sqrt{\delta}\rfloor -1}\leq \frac{6\delta \alpha_d r^dn^{d}}{nr \sqrt{\delta}}=6\sqrt{\delta} \alpha_d r^{d-1}n^{d-1}\,.
\end{align}
 If there are several possible choices for $\gamma_0$, we pick the smallest one. We denote by $X^+$ and $X^-$ the following sets of edges:
$$X^+=\{\langle y, z\rangle\in \E_n^d: y\in \partial ^+_nB(\gamma_0)\cap E_n\}\,,\qquad X^-=\{\langle y, z\rangle\in \E_n^d: y\in \partial ^-_nB(\gamma_0)\cap E_n^c\}\,.$$
Let us control the number of edges in $X^+\cup X^-$ using inequality \eqref{eq:contxgamma0}:
\begin{align}\label{eq:controleM}
\card(X^+ \cup X^-)\leq 2d\left(\card(E_n\cap \partial^+_n B(\gamma_0))+\card(E_n^c\cap \partial^-_n B(\gamma_0))\right)\leq 12d \sqrt{\delta} \alpha_d r^{d-1}n^{d-1}\,.
\end{align}
\begin{figure}[H]
\begin{center}
\def\svgwidth{0.6\textwidth}
   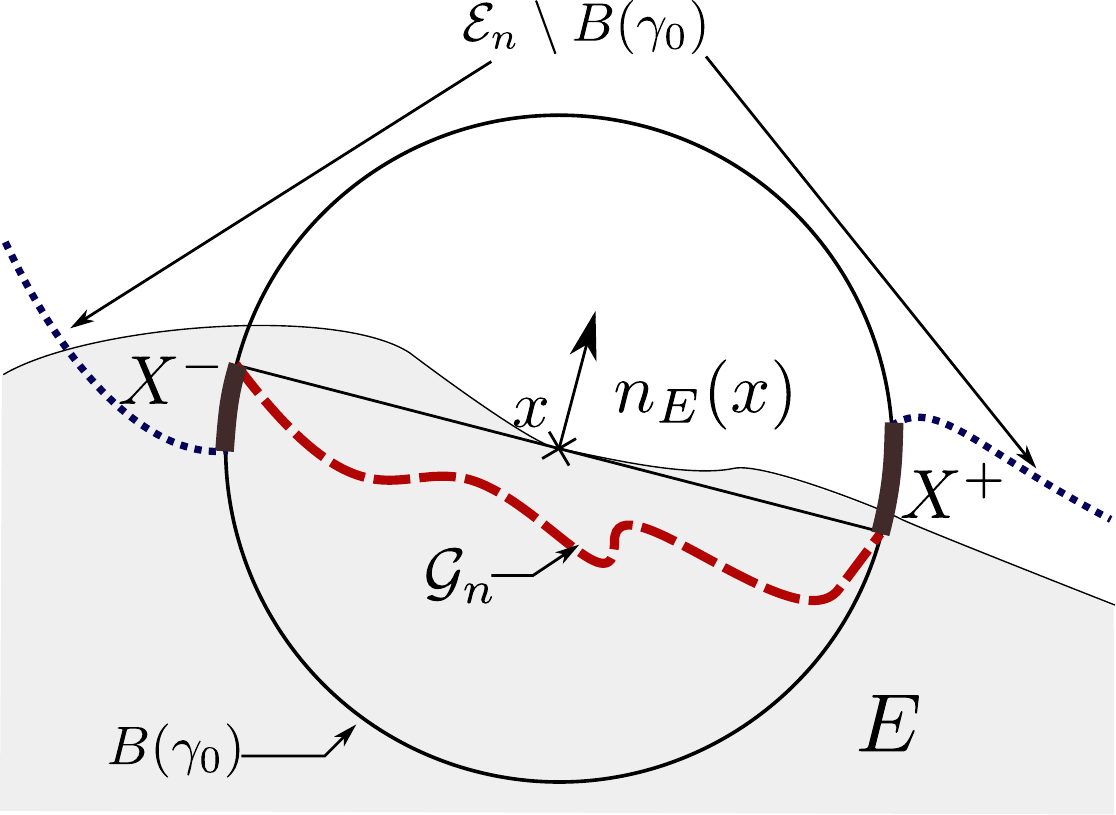
   \caption{\label{fig:repFn}Representation of the cutset $\cF_n$.}
   \end{center}
\end{figure}
\noindent Let $\cG_n$ be a minimal cutset for $\phi_n(\partial^-_n B(\gamma_0),\partial^+_nB(\gamma_0),B(\gamma_0))$.
Set $$\cF_n=(\cE_n\setminus B(\gamma_0))\cup X^+\cup X^-\cup \cG_n\,.$$ We claim that $\cF_n$ is a $(\Gamma_n^1,\Gamma_n^2)$-cutset (see figure \ref{fig:repFn}).
Let $\gamma$ be a path from $\Gamma_n^1$ to $\Gamma_n^2$ in $\Omega_n$. We write $\gamma=(x_0,e_1,x_1\dots,e_p,x_p)$.

$\triangleright$ If $\gamma\cap ((\cE_n\setminus B(\gamma_0))\cup X^+\cup X^-)\neq \emptyset$ then $\gamma\cap\cF_n\neq\emptyset$.

$\triangleright$  Let us assume that $\gamma\cap ((\cE_n\setminus B(\gamma_0))\cup X^+\cup X^-)= \emptyset$. Since $\gamma\cap\cE_n\neq \emptyset $, then we have $\gamma\cap (\cE_n\cap B(\gamma_0))\neq\emptyset$. Besides, we have $\gamma \cap E_n^ c\cap \partial_n B(\gamma_0)\neq\emptyset $ and $\gamma \cap E_n\cap \partial_n B(\gamma_0)\neq\emptyset $.
Set 
$$m=\inf\{k\geq 0: x_k \in E_n^ c\cap \partial_n B(\gamma_0)\}\,\qquad\text{and}\qquad
l=\sup\{k\leq m: x_k\in E_n\cap \partial_n B(\gamma_0)\}\,.$$
Moreover, since $\gamma\cap (X^+\cup X^-)=\emptyset$, we have $x_{l}\in \partial ^-_n B(\gamma_0)$ and $x_{m}\in \partial^+_n B(\gamma_0)$. By construction, the portion of $\gamma$ between $x_{l}$ and $x_{m}$ is strictly inside $B(\gamma_0)$.
 Since $\cG_n$ is a cutset between $\partial^-_n B (\gamma_0)$ and $\partial ^+_n B(\gamma_0)$, it follows that $\gamma\cap \cG_n\neq\emptyset$. The set $\cF_n$ is a $(\Gamma_n^1,\Gamma_n^2)$-cutset.
Since $\cE_n\in\sC_n(\ep_0)$, we have $V(\cF_n)\geq V(\cE_n)-\ep_0 n ^{d-1}$. Using inequality \eqref{eq:ineqmun}, it follows that 
\begin{align*}
V(\cG_n)+M\card(X^+\cup X^-)&\geq V(\cE_n\cap B(\gamma_0))-\ep_0n ^{d-1}\\&\geq \mu_n( B(x,r))n^{d-1}-\ep_0 n ^{d-1}\\
&\geq  \nu(B(x,r))n^{d-1}-2\ep\alpha_{d-1}r^{d-1}n^{d-1}-\ep_0 n^{d-1}
\end{align*}
and using inequality \eqref{eq:controleM}
\[ \nu(B(x,r))\leq \frac{\phi_n(\partial^-_n B(\gamma_0),\partial^+_n B(\gamma_0),B(\gamma_0))}{n^{d-1}}+ 12dM\sqrt{\delta}\alpha_d r^{d-1}+ 2\ep\alpha_{d-1}r^{d-1}+\ep_0 \,. \]
We can choose $\ep_0$ small enough depending on $r$ and $\delta$ such that 
$$\frac{\ep_0}{\alpha_{d-1}r^{d-1}}\leq \ep \,.$$
It follows that 
\begin{align*}
 \frac{\nu(B(x,r))}{\alpha_{d-1}r^{d-1}}&\leq \frac{\phi_n(\partial^-_n B(\gamma_0),\partial^+_n B(\gamma_0),B(\gamma_0))}{\alpha_{d-1}r^{d-1}n^{d-1}}+ 12dM\sqrt{\delta}\frac{\alpha_d}{\alpha_{d-1}} + 3\ep \\
 &\leq \frac{\phi_n(\partial^-_n B(\gamma_0),\partial^+_n B(\gamma_0),B(\gamma_0))}{\alpha_{d-1}(r'-\gamma_0)^{d-1}n^{d-1}}(1+\sqrt{\delta}) ^{d-1}+ 12dM\sqrt{\delta}\frac{\alpha_d}{\alpha_{d-1}} + 3\ep\,. 
 \end{align*}
 Let $\kappa_d>0$ be given by lemma \ref{lem:inclusion}.
Let us assume that 
$$\frac{\nu(B(x,r))}{\alpha_{d-1}r^{d-1}}\geq(\nu_G(n_E(x))+\ep+\kappa_d M\sqrt \ep)(1+\sqrt{\delta}) ^{d-1}+ 12dM\sqrt{\delta}\frac{\alpha_d}{\alpha_{d-1}} + 3\ep\,.$$
It follows that 
$$\frac{\phi_n(\partial^-_n B(\gamma_0),\partial^+_n B(\gamma_0),B(\gamma_0))}{\alpha_{d-1}(r'-\gamma_0)^{d-1}n^{d-1}}\geq \nu_G(n_E(x))+\ep+\kappa_d M\sqrt \ep\,.$$
Then using lemma \ref{lem:inclusion}
\begin{align*}
\{\exists \cE_n\in\sC_n(\ep_0):\,(\bR(\cE_n),\mu_n(\cE_n))\in U\}&\subset \overline G_n(x,r'-\gamma_0,n_E(x),1, \nu_G(n_E(x))+\ep+ \kappa_ dM\sqrt \ep)^ c\\&\subset \overline G_n(x,r'(1-\sqrt\ep),n_E(x),1, \nu_G(n_E(x))+\ep)^ c\,
\end{align*}
where we used that $$r'(1-\sqrt{\ep})\leq r'(1-\sqrt{\delta})\leq r'-\gamma_0\leq r'\,.$$
Thanks to lemma \ref{lem:phicassimple}, it yields that
$$\limsup_{n\rightarrow\infty}\frac{1}{n^{d-1}}\log \Prb(\exists \cE_n\in\sC_n(\ep_0):\,(\bR(\cE_n),\mu_n(\cE_n))\in U)=-\infty\,.$$
This contradicts inequality \eqref{eq:Enu}, thus
\begin{align*}
 \frac{\nu(B(x,r))}{\alpha_{d-1}r^{d-1}}&\leq (\nu_G(n_E(x))+\ep+\kappa_d M\sqrt \ep)(1+\sqrt{\delta}) ^{d-1}+ 12dM\sqrt{\delta}\frac{\alpha_d}{\alpha_{d-1}} + 3\ep\,.
 \end{align*}
It follows that 
$$\limsup_{r\rightarrow 0} \frac{\nu(B(x,r))}{\alpha_{d-1}r^{d-1}}\leq (\nu_G(n_E(x))+\ep+10d^3M\sqrt \ep)(1+\sqrt{\delta}) ^{d-1}+ 12dM\sqrt{\delta}\frac{\alpha_d}{\alpha_{d-1}} + 3\ep$$
and by letting $\delta$ and $\ep$ go to $0$:
$$\limsup_{r\rightarrow 0} \frac{\nu(B(x,r))}{\alpha_{d-1}r^{d-1}}\leq \nu_G(n_E(x))\,.$$

{\noindent $\bullet$ \bf Case $x\in \partial ^*\Omega \cap ((\Gamma^1\setminus \overline{\partial^*E} )\cup (\Gamma^2\cap \partial^*E))$} This case is handled in a similar way but with an extra issue: there is no $r$ small enough such that $B(x,r) \subset \Omega$. The set $\cF_n$ that we have built may not be included in $\Pi_n$. We will only treat the case $x\in \Gamma^1\setminus \overline{\partial ^* E}$ (the case $x\in \Gamma^2\cap \partial ^* E$ is very similar). To solve this issue, we can pick $\delta$ small enough depending on $\ep$ (small enough such that it satisfies condition \eqref{cond:deltalemphi} in lemma \ref{lem:phicassimple} where $\eta$ is replaced by $\ep$) and then by lemma \ref{lem:Omegabord}, we pick $r_0$ small enough depending on $\delta$ such that it satisfies for all $r\in]0,r_0]$
\begin{align}\label{eq:choixr2}
\cL^d((\Omega\cap B(x,r))\Delta B^-(x,r,n_\Omega(x)))\leq \delta \alpha_d r^d\,,
\end{align} 
\begin{align}\label{eq:surfacedepasse}
(\Omega\cap B(x,r,n_\Omega(x)))\Delta B^+ (x,r,n_\Omega(x))\subset\left\{\,y\in B(x,r):|(y-x)\cdot n_\Omega(x)|\leq \delta \|y-x\|_2\,\right\}
\end{align}
and
\begin{align}\label{eq:blaba}
\left|\frac{1}{\alpha_{d-1} r_0^{d-1}}\cH^{d-1}(\partial^*\Omega\cap B(x,r_0))-1\right|\leq \ep\,.
\end{align}
Moreover, up to choosing a smaller $r_0$ we can assume that $B(x,r_0)\cap (E\cup \Gamma^2)=\emptyset$ (this is possible since $d_2(x,\Gamma^2\cup E)>0$. 
Note that for $x\in \Gamma^1 \setminus\overline{ \partial ^* E}$ the set $E_n\cap B(x,r_0)$ looks like $\Omega\cap  B^+(x,r_0,n_\Omega(x))$. Whereas for $y\in\partial ^* E\cap \Gamma^2$ $E_n\cap B(y,r)$ looks like $B^-(y,r,n_\Omega(x))$. 
We have for $r\leq r_0$, for $n$ large enough, using proposition \ref{prop:minkowski}
\begin{align*}
&\cL^ d((E_n\cap B^-(x,r,n_\Omega(x)))\cup ( \cV_\infty(\Omega,1/n)\cap B^+(x,r,n_\Omega(x)))\\
&\leq \cL^ d((E_n\Delta E) \cup (E\cap B^-(x,r,n_\Omega(x))) +\cL^d(\Omega\cap B^+(x,r,n_\Omega(x)))+ \cL^ d( \cV_\infty(\Omega,1/n)\setminus \Omega)\\
&\leq \cL^d(E_n\Delta E)+\cL^d (E\cap B^-(x,r,n_\Omega(x))+\cL^d((\Omega\cap B(x,r))\Delta B^-(x,r,n_\Omega(x)))+\frac{8d}{n}\cH^{d-1}(\partial \Omega)\\
&\leq 3\delta  \alpha_d r^d\,
\end{align*}
where we use in the last inequality that $E\cap B(x,r_0)=\emptyset$,
and 
\begin{align}\label{eq:difsymcas2}
\card \left(\left(E_n\cap B^-(x,r,n_\Omega(x))\cap\sZ_n^d\right)\cup \left( \Omega_n\cap B^+(x,r,n_\Omega(x)) \right)\right)
&\leq 4\delta  \alpha_d r^d n^d\,.
\end{align}
By the same arguments (see inequality \eqref{eq:contxgamma0}), there exists $\gamma_0$ in $\{1/n,\dots,(\lfloor nr\sqrt{\delta}\rfloor -1)/n\}$ such that 
$$\card(E_n\cap \partial^-_n B(\gamma_0))+\card(E_n^c\cap \Omega_n\cap \partial^+_n B(\gamma_0))\leq 6\sqrt{\delta} \alpha_d r^{d-1}n^{d-1}\,.$$
 We denote by $X^+$ and $X^-$ the following set of edges:
$$X^+=\{\langle y, z\rangle\in \Pi_n: y\in \partial ^+_nB(\gamma_0)\cap \Omega_n\}\,,\qquad
X^-=\{\langle y, z\rangle\in \Pi_n: y\in \partial ^-_nB(\gamma_0)\cap E_n\}\,.$$ 
Let $\cG'_n$ be a minimal cutset for $\phi_n(\partial^+_n B(\gamma_0)\cup (\Gamma_n^1\setminus \partial ^-_n B(\gamma_0)),\partial ^-_n B(\gamma_0),B(\gamma_0)\cap \Omega_n)$. Set
$$\cF'_n=(\cE_n\setminus B(\gamma_0))\cup X^+ \cup X^- \cup \cG'_n\,.$$
The set $\cF'_n$ may not be included in $\Pi_n$, but we claim that $\cF'_n\cap \Pi_n$ is a $(\Gamma_n^1,\Gamma_n^2)$- cutset. Once we prove this result, the remaining of the proof is the same than in the case $x\in \partial ^* E\cap \Omega$ using lemma \ref{lem:phicassimple}.

Let $\gamma$ be a path from $\Gamma_n^1$ to $\Gamma_n^2$ in $\Omega_n$. 

$\triangleright$ If $\gamma\cap ((\cE_n\setminus B(\gamma_0))\cup X^+\cup X^-)\neq \emptyset$ then $\gamma\cap(\cF'_n\cap \Pi_n)\neq\emptyset$.

$\triangleright$  Let us assume that $\gamma\cap( (\cE_n\setminus B(\gamma_0))\cup X^+\cup X^-)= \emptyset$. Note that for $n$ large enough, we have $\Gamma_n^2\cap B(\gamma_0)=\emptyset$. We aim to prove that $\gamma\cap\cG'_n\neq\emptyset$. As in the previous case, we can extract from $\gamma$ an excursion in $B(\gamma_0)$ that starts at $z\in (\partial_n B(\gamma_0)\cup \Gamma_n^1) \cap E_n$ and ends at $w\in\partial_n B(\gamma_0)\cap E_n^c\cap\Omega_n$. Since $\gamma\cap X^+=\emptyset$ then $w\in\partial_n  ^- B(\gamma_0)$. Similarly, if $z\in\partial_n B(\gamma_0)$ since $\gamma\cap X^-=\emptyset$ we have $z\in \partial^+_n B(\gamma_0)$. Since $\cG'_n$ is a cutset that cuts $\partial^+_n B(\gamma_0)\cup (\Gamma_n^1\setminus \partial ^-_n B(\gamma_0))$ from $\partial ^-_n B(\gamma_0)$ in $B(\gamma_0)\cap \Omega_n$, we have $\cG'_n\cap \gamma\neq \emptyset$. We conclude as in the previous case that 
$$\limsup_{r\rightarrow 0} \frac{\nu(B(x,r))}{\alpha_{d-1}r^{d-1}}\leq \nu_G(n_\Omega(x))\,.$$

\noindent {\bf Step 4. We prove that $\nu$ is supported on $\fE $.} We recall that $$\fE =(\partial ^* E\cap\Omega)\cup(\partial ^* \Omega \cap (\Gamma^1\setminus \overline{\partial^* E}))\cup (\partial ^* E\cap \Gamma^2)\,.$$  Let us denote by $\lambda$ the following measure
$$\lambda=\cH^{d-1}|_{\fE}\,.$$
By Lebesgue decomposition theorem, there exists $\nu_\lambda$ and $\nu_s$ such that $\nu_\lambda$ is absolutely continuous with respect to $\lambda$, there exists a Borelian set $A$ such that $\lambda(A^c)=\nu_s(A)=0$ and
$$\nu=\nu_{\lambda}+\nu_s\,.$$
There exists $f\in L^1(\lambda)$ such that $\nu_\lambda=f \lambda$. 
Moreover, since $\lambda$ is a Radon measure, it is of Vitali type (see for instance theorem 4.3. in \cite{rigot}). 
By theorem 2.3 in \cite{rigot},
for $\lambda$-almost every $x$ we have
$$f(x)=\lim_{r\rightarrow 0}\frac{\nu(B(x,r))}{\lambda(B(x,r))}\geq 0\,.$$
Since $f\in L^1(\lambda)$, for $\lambda$-almost every $x$ we have
$$f(x)=\lim_{r\rightarrow 0}\frac{1}{\lambda(B(x,r))}\int_{B(x,r)\cap \fE}f(y)d\cH^{d-1}(y)\,.$$
We aim to prove that $\nu_s=0$. The main issue to prove this is that the reduced boundary $\partial^* E$ may be very different from the topological boundary $\partial E$, the set $\fE$ may not be a "continuous cutset". However, up to modifying $E$ on a set of null measure, we can assume that $\partial E =\overline{\partial ^* E}$ (see for instance remark 15.3 in \cite{maggi_2012}). 
We recall that $\cH^{d-1}(\fE)<\infty$.
Let $\ep>0$. By proposition \ref{prop:utilisationvitali}, there exists a finite family of disjoint closed balls $(B(x_i,r_i,v_i))_{i\in I_1\cup I_2\cup I_3}$ such that for $i\in I_1$, we have $x_i\in \partial ^* E\cap \Omega$, for $i\in I_2$, we have $x_i\in \partial ^* \Omega \cap (\Gamma^1\setminus \overline{\partial^* E})$ and for $i\in I_3$, we have $x_i\in \partial ^* E\cap \Gamma^2$, and the following properties hold
\begin{align}\label{eq:af100}
\cH^{d-1}(\fE\setminus \cup_{i\in I_1\cup I_2\cup I_3}B(x_i,r_i)))\leq \ep\,,
\end{align}
\begin{align}\label{eq:af200}
\forall i \in I_1\cup I_2\cup I_3\quad  \forall\, 0<r\leq r_i\qquad \left|\frac{1}{\alpha_{d-1} r^{d-1}}\cH^{d-1}(\fE\cap B(x_i,r))-1\right|\leq \ep\,,
\end{align}
\begin{align}\label{eq:af400}
\forall i\in I_1\cup I_2\cup I_3\qquad \left|\int_{B(x_i,r_i)\cap\fE}f(y)d\cH^{d-1}(y)-\nu(B(x_i,r_i))\right|\leq \ep\alpha_{d-1} r_i^{d-1}\,,
\end{align}
\begin{align}\label{eq:af500}
\forall i \in I_1\qquad B(x_i,r_i)\subset \Omega\quad\text{ and } \quad\cL^d((E\cap B(x_i,r_i)) \Delta B^-(x_i,r_i,v_i))\leq \ep \alpha_d r_i^d\,,
\end{align}
\begin{align}\label{eq:af600}
\forall i \in I_2\qquad d_2(B(x_i,r_i), \Gamma^2\cup E)>0 \quad\text{ and } \quad\  \cL^d((\Omega\cap B(x_i,r_i)) \Delta B^-(x_i,r_i,v_i))\leq \ep \alpha_d r_i^d\,,
\end{align}
\begin{align}\label{eq:af700}
\forall i \in I_3\qquad d_2(B(x_i,r_i), \Gamma^1)>0,  \quad\  &\cL^d((E\cap B(x_i,r_i)) \Delta B^-(x_i,r_i,v_i))\leq \ep \alpha_d r_i^d\nonumber\\\text{ and }\quad& \cL^d((\Omega\cap B(x_i,r_i)) \Delta B^-(x_i,r_i,v_i))\leq \ep \alpha_d r_i^d\,.
\end{align}
Set \[r_{min}=\min_{i\in I_1\cup I_2\cup I_3}r_i\,.\]Let $g$ be a continuous function compactly supported with values in $[0,1]$ such that $g=1$ on $\overline{\Omega}$.
For $i\in I_1\cup I_2\cup I_3$, let $g_i$ be a continuous function compactly supported with values in $[0,1]$ such that $g_i=1$ on $B(x_i,r_i(1-\sqrt{\ep}/4))$ and $g_i=0$ on $B(x_i,r_i)^c$. Hence, we have
\begin{align}\label{eq:nugi}
\nu(g_i)\leq \nu(B(x_i,r_i))\,.
\end{align}
Set $$U=\left\{F\in\cB(\sR^d): \cL^d(F\Delta E)\leq \ep\alpha_d r_{min}^d\right\}\times\left\{\begin{array}{c}\mu\in\cM(\sR^d):\qquad|\mu(g)-\nu(g)|\leq \ep\text{  and}\\\forall i\in I_1\cup I_2\cup I_3\quad |\mu(g_i)-\nu(g_i)|\leq \ep\alpha_{d-1}r_i^{d-1} \end{array}\right\}\,.
$$
Let $\ep_0>0$ we will choose later.
Thanks to inequality \eqref{eq:Enu}, there exists an increasing sequence $(a_n)_{n\geq 1}$ such that
$$\Prb(\exists \cE_{a_n}\in\sC_{a_n}(\ep_0):\,(\bR(\cE_{a_n}),\mu_{a_n}(\cE_{a_n}))\in U)>0\,.$$
For short, we will write $n$ instead of $a_n$.
On the event $\{\exists \cE_n\in\sC_n(\ep_0):\,(\bR(\cE_n),\mu_n(\cE_n))\in U\}$, we pick $\cE_n\in\sC_n(\ep_0)$ such that $(\bR(\cE_n),\mu_n(\cE_n))\in U$ (if there are several possible choices, we pick one according to a deterministic rule).
To shorten the notation, we write $\mu_n$ for $\mu_n(\cE_n)$ and $E_n$ for $\bR(\cE_n)$.
We can prove thanks to inequalities \eqref{eq:af500} as in the proof of inequality \eqref{eq:difsym} that for $n$ large enough 
\begin{align}\label{eq:difsym4b}
\forall i\in I_1 \qquad \card\left(\left((E\cap B(x_i,r_i))\Delta B^-(x_i,r_i,v_i)\right)\cap\sZ_n^d\right)\leq 4\ep\alpha_d r_i^dn^d
\end{align}
and
\begin{align}\label{eq:difsym4}
\forall i\in I_1 \qquad \card\left(\left((E_n\cap B(x_i,r_i))\Delta B^-(x_i,r_i,v_i)\right)\cap\sZ_n^d\right)\leq 4\ep\alpha_d r_i^d n^d\,.
\end{align}
For $i\in I_2$, thanks to \eqref{eq:af600}, as in the proof of inequality \eqref{eq:difsymcas2}, we have
\begin{align}\label{eq:difsymI2}
\card \left((E_n\cap B^-(x_i,r_i,v_i)\cap\sZ_n^d)\cup (\Omega_n\cap B^+(x_i,r_i,v_i) )\right)
&\leq 4\ep  \alpha_d r_i^d n^d\,.
\end{align}
Using proposition \ref{prop:minkowski} and \eqref{eq:af700}, for $i\in I_3$, for $n$ large enough
\begin{align*}
&\cL^ d((E_n^c\cap B^-(x_i,r_i,v_i))\cup ( \cV_\infty(\Omega,1/n)\cap B^+(x_i,r_i,v_i))\\
&\leq \cL^ d((E_n\Delta E) \cup (E^c\cap B^-(x_i,r_i,v_i)) +\cL^d(\Omega\cap B^+(x_i,r_i,v_i))+ \cL^ d( \cV_\infty(\Omega,1/n)\setminus \Omega)\\
&\leq \cL^d(E_n\Delta E)+\cL^d((E^c\cap B(x_i,r_i))\Delta B^-(x_i,r_i,v_i))+\cL^d((\Omega\cap B(x_i,r_i))\Delta B^-(x_i,r_i,v_i))+\frac{8d}{n}\cH^{d-1}(\partial \Omega)\\
&\leq 4\ep \alpha_d r^d\,
\end{align*}
and finally
\begin{align}\label{eq:difsymI2b}
\card \left((E_n^c\cap B^-(x_i,r_i,v_i)\cap\sZ_n^d)\cup (\Omega_n\cap B^+(x_i,r_i,v_i) )\right)
&\leq 8\ep  \alpha_d r_i^d n^d\,.
\end{align}
 For each $i\in I_1\cup I_2\cup I_3$ and $\alpha$ in $\{\lfloor nr_i\sqrt{\ep}/2\rfloor/n,\dots , (\lfloor nr_i\sqrt{\ep}\rfloor -1)/n\}$, we define
\[B_i(\alpha)=B_i(x_i,r_i-\alpha),.\]
By the same arguments as in the previous step, using a pigeon-hole principle, there exists $\alpha_i$ in \newline $\{\lfloor nr_i\sqrt{\ep}/2\rfloor/n,\dots,(\lfloor nr_i\sqrt{\ep} \rfloor -1)/n\}$ such that for $i\in I_1$
\begin{align}\label{eq:contcardI100}
\card((E_n\cup E)\cap \partial^+_n B_i(\alpha_i)\cap\sZ_n^d)+\card((E_n^c\cup E^c)\cap \partial^-_n B_i(\alpha_i)\cap\sZ_n^d)&\leq 20 \sqrt{\ep} \alpha_d r_i^{d-1}n^{d-1}\,.
\end{align}
 For $i\in I_2$ there exists 
  $\alpha_i$ in $\{\lfloor nr_i\sqrt{\ep}/2\rfloor/n,\dots,(\lfloor nr_i\sqrt{\ep} \rfloor -1)/n\}$ such that
\begin{align}\label{eq:contcardI200}
\card(\Omega_n\cap \partial^+_n B_i(\alpha_i))+\card(E_n\cap \partial^-_n B_i(\alpha_i)\cap\sZ_n^d)&\leq 20\sqrt{\ep} \alpha_d r_i^{d-1}n^{d-1}\,,
\end{align}
and for $i\in I_3$, there exists 
  $\alpha_i$ in $\{\lfloor nr_i\sqrt{\ep}/2\rfloor/n,\dots,(\lfloor nr_i\sqrt{\ep} \rfloor -1)/n\}$ such that
\begin{align}\label{eq:contcardI300}
\card(\Omega_n\cap \partial^+_n B_i(\alpha_i))+\card(E_n^c\cap \partial^-_n B_i(\alpha_i)\cap\sZ_n^d)&\leq 20\sqrt{\ep} \alpha_d r_i^{d-1}n^{d-1}\,.
\end{align}
 We denote by $X^+$ and $X^-$ the following set of edges:
\begin{align*}
X^+=&\bigcup_{i\in I_1 }\{\langle y, z\rangle\in \E_n^d: y\in \partial ^+_nB_i(\alpha_i)\cap (E\cup E_n)\}\cup \bigcup_{i\in I_2 \cup I_3}\{\langle y, z\rangle\in \E_n^d: y\in \partial ^+_nB_i(\alpha_i)\cap \Omega_n\}\,,
 \end{align*}
 \begin{align*}
X^-=&\bigcup_{i\in I_1 }\{\langle y, z\rangle\in \E_n^d: y\in \partial ^-_nB_i(\alpha_i)\cap (E^c\cup E_n^c)\}\cup \bigcup_{i\in I_2 }\{\langle y, z\rangle\in \E_n^d: y\in \partial ^-_nB_i(\alpha_i)\cap  E_n\}\\&\cup\bigcup_{i\in I_3 }\{\langle y, z\rangle\in \E_n^d: y\in \partial ^-_nB_i(\alpha_i)\cap  E_n^c\}\,.
 \end{align*}
Let us control the number of edges in $X^+\cup X^-$ using inequalities \eqref{eq:af200}, \eqref{eq:contcardI100}, \eqref{eq:contcardI200} and \eqref{eq:contcardI300}
\begin{align}\label{eq:contcardX}
\card(X^+ \cup X^-)&\leq 40d\frac{\alpha_d}{\alpha_{d-1}}\sum_{i\in I_1 \cup I_2\cup I_3}\sqrt{\ep} \alpha_{d-1} r_i^{d-1}n^{d-1}\nonumber\\
&\leq 40d \sqrt{\ep}\frac{\alpha_d}{\alpha_{d-1}} \sum_{i\in I_1 \cup I_2\cup I_3}\frac{1}{1-\ep}\cH^{d-1}(\fE\cap B(x_i,r_i))n^{d-1}\nonumber\\
&\leq 80d\sqrt{\ep}\frac{\alpha_d}{\alpha_{d-1}}\cH^{d-1}(\fE)n^{d-1}
\end{align}
where we use that $\ep\leq 1/2$.
Let us denote by $\cF_0$ the set of edges with at least one endpoint in
$\cV_2(\fE\setminus \cup_{i\in I_1\cup I_2\cup I_3}B_i(\alpha_i), d/n)$.
Using inequality \eqref{eq:af200}, we have
\begin{align}\label{eq:calculrbar}
\sum_{i\in I_1 \cup I_2\cup I_3}\cH^{d-1}(B(x_i,r_i)\cap \fE)-\cH^{d-1}(B_i(\alpha_i)\cap \fE)&\leq \sum_{i\in I_1 \cup I_2\cup I_3} (1+\ep)\alpha_{d-1}r_i^{d-1}- (1-\ep)\alpha_{d-1}(r_i-\alpha_i)^{d-1}\nonumber\\
&\leq \sum_{i\in I_1 \cup I_2\cup I_3} \left(1+ \ep -(1-\ep)(1-\sqrt{\ep})^{d-1}\right)\alpha_{d-1}r_i^{d-1}\nonumber\\
&\leq (1+\ep-(1-\ep)(1-(d-1)\sqrt\ep))\frac{1}{1-\ep}\cH^{d-1}(\fE)\nonumber\\
&\leq 2d\sqrt\ep\cH^{d-1}(\fE)
\end{align}
for small enough $\ep$ depending on $d$.
Hence, we have using proposition \ref{prop:minkowski}, inequalities \eqref{eq:af100} and \eqref{eq:calculrbar}, for $n$ large enough, 
\begin{align}\label{eq:contcardF0b}
\card(\cF_0)\leq 2d\cL^d(\cV_2(\fE\setminus \cup_{i\in I}B_i(\alpha_i), 2d/n))n^d\leq 10d^2(2d\sqrt{\ep}\cH^{d-1}(\fE)+\ep)n^{d-1}\,.
\end{align}
We claim that $\cF_n=X^+\cup X^- \cup\cF_0\cup_{i\in I_1\cup I_2\cup I_3}(\cE_n\cap B_i(\alpha_i))$ is a $(\Gamma_n^1,\Gamma_n^2)$-cutset. Let $\gamma$ be a path from $\Gamma_n^1$ to $\Gamma_n^2$ in $\Omega_n$. We write $\gamma=(x_0,e_1,x_1\dots,e_p,x_p)$. By definition of $E_n$, we have $x_0\in E_n$ and $x_p\in E_n^c$.

\noindent $\triangleright$ Let us assume that $x_0\in B_i(\alpha_i)$ for $i\in I_2$. Since for $n$ large enough $B_i(\alpha_i)\cap\Gamma_n^2=\emptyset$, the path $\gamma$ eventually exits the ball $B_i(\alpha_i)$.
 Let $x_r$ denotes the first point that reaches $\partial _nB_i(\alpha_i)$, \textit{i.e.},
$$r=\inf\left\{k\geq 0:\, x_k\in \partial_n B_i(\alpha_i)\right\}\,.$$ 
  We have $x_r\in \Omega_n$.
\begin{itemize}[$\cdot$]
\item If $x_r\in \partial^+_n B(\alpha_i)$, then $x_r\in \partial^+_n B(\alpha_i)\cap \Omega_n$ and $\langle x_r,x_{r+1}\rangle \in X^+$.
\item If  $x_r\in \partial^-_n B(\alpha_i)\cap E_n$, then $\langle x_r,x_{r+1}\rangle \in X^-$.
\item If $x_r\in \partial^-_n B(\alpha_i)\cap E_n^c$, since $x_0\in E_n$, then $\gamma \cap (\partial ^e E_n\cap B_i(\alpha_i))\neq\emptyset$ and so $\gamma \cap (\cE_n\cap B_i(\alpha_i))\neq\emptyset$.
\end{itemize}

\noindent $\triangleright$ Let us assume that $x_p\in B_i(\alpha_i)$ for $i\in I_3$. For $n$ large enough, we have $\Gamma_n^1\cap B_i(\alpha_i)=\emptyset$, and so $x_0\notin B_i(\alpha_i)$. Let $x_l$ denotes the last point to enter in $B_i(\alpha_i)$, \textit{i.e.},
\[l=\sup\left\{k\geq 0:\, x_k \in\partial_n B_i(\alpha_i)\right\}\,.\]
We have $x_l\in \Omega_n$.
\begin{itemize}[$\cdot$]
\item If $x_l\in \partial^+_n B(\alpha_i)$, then $x_l\in \partial^+_n B(\alpha_i)\cap \Omega_n$ and $\langle x_{l-1},x_{l}\rangle \in X^+$.
\item If  $x_l\in \partial^-_n B(\alpha_i)\cap E_n^c$, then $\langle x_{l-1},x_{l}\rangle \in X^-$.
\item If $x_l\in \partial^-_n B(\alpha_i)\cap E_n$, since $x_p\in E_n^c$, then $\gamma \cap (\partial ^e E_n\cap B_i(\alpha_i))\neq\emptyset$ and so $\gamma \cap (\cE_n\cap B_i(\alpha_i))\neq\emptyset$.
\end{itemize}

\noindent $\triangleright$ Let us assume $x_0\in \cV_2(\Gamma^1\setminus \overline{\partial^* E},d/n)\setminus \bigcup_{i\in I_2}B_i(\alpha_i)$ then $\langle x_0,x_1\rangle \cap \cF_0\neq \emptyset$.

\noindent$\triangleright$ Let us assume $x_p\in \cV_2(\Gamma^2\cap\partial ^* E,d/n)\setminus \bigcup_{i\in I_3}B_i(\alpha_i)$ then $\langle x_{p-1},x_p\rangle \cap \cF_0\neq \emptyset$.

\noindent$\triangleright$ Let us assume $x_0\in \cV_2(\Gamma^1\setminus \overline{\partial^* E},d/n)^c$ and $x_p\in \cV_2(\Gamma^2\cap\partial ^* E,d/n)^c$. Set $\widetilde E=E\cup (\cV_2(E,d/n)\setminus \Omega)$.
It follows that $x_0\in\widetilde E$ and $x_{p}\in \widetilde E^c$. 
Let us assume there exists an excursion of $\gamma$ inside a ball $B_i(\alpha_i)$, for $i\in I_1$ that starts at $x_r\in \partial_n B_i(\alpha_i)\cap \widetilde E=\partial_n B_i(\alpha_i)\cap  E$ and exits at $x_m\in \partial_n B_i(\alpha_i)\cap \widetilde E^c=\partial_n B_i(\alpha_i)\cap  E^c$ (where we use that $B_i(\alpha_i)\subset \Omega$).
\begin{itemize}[$\cdot$]
\item If $x_r\notin \partial ^-_n B_i(\alpha_i)\cap E_n$, then  $x_r\in \partial ^+_n B_i(\alpha_i) \cap E$ or  $x_r\in \partial ^-_n B_i(\alpha_i) \cap E_n^ c$. It follows that  $\langle x_{r-1},x_r\rangle \in X^-\cup X^+$.
\item If $x_m\notin \partial ^+_n B_i(\alpha_i)\cap E_n^c$, then  $x_r\in \partial ^-_n B_i(\alpha_i) \cap E^c$ or  $x_r\in \partial ^+_n B_i(\alpha_i) \cap E_n$. It follows that $\langle x_{m},x_{m+1}\rangle \in X^-\cup X^+$.
\item If $x_r\in E_n$ and $x_m\in E_n^c$, then $\gamma\cap (\cE_n\cap B_i(\alpha_i))\neq\emptyset$.
\end{itemize}
If no such excursion exists then the path $\gamma$ must cross $(\partial \widetilde E\setminus \partial (\cV_2(\Omega,d/n)))\setminus\cup_{i\in I_1}B_i(\alpha_i)$. We distinguish two cases.
\begin{itemize}[$\cdot$]
\item If $\gamma$ crosses $(\partial E\cap \Omega )\setminus\cup_{i\in I_1}B_i(\alpha_i)$. Since $\partial E=\overline{\partial ^* E}$, it follows that $\gamma\cap\cF_0\neq\emptyset$.
\item If $\gamma$ crosses $\partial \widetilde E$ outside $\Omega$. Then $\gamma$ must cross $\cV_2((\Gamma^1\setminus \overline{\partial ^* E})\setminus \cup_{i\in I_2}B_i(\alpha_i),d/n)$ or $ \cV_2((\Gamma^2\cap\partial ^* E)\setminus \cup_{i\in I_3}B_i(\alpha_i),d/n)$.  Hence, $\gamma\cap\cF_0\neq\emptyset$.
\end{itemize}
For all these cases, we have $\gamma\cap\cF_n\neq \emptyset$, thus $\cF_n$ is indeed a $(\Gamma_n^1,\Gamma_n^2)$-cutset.
Since $\cE_n\in\sC_n(\ep_0)$, using inequalities \eqref{eq:contcardX} and \eqref{eq:contcardF0b}, we have
\begin{align}\label{eq:contVen}
 V(\cE_n)&\leq V(\cF_n)+\ep_0n^{d-1}\nonumber\\&\leq V(X^+\cup X^- \cup\cF_0\cup_{i\in I_1\cup I_2\cup I_3}(\cE_n\cap B_i(\alpha_i)))+\ep_0 n ^{d-1}\nonumber\\
 & \leq M(\card(X^+\cup X^-)+ \card(\cF_0))+ \sum_{i\in I_1\cup I_2\cup I_3}V(\cE_n\cap B_i(\alpha_i))+\ep_0 n ^{d-1}\nonumber\\
 &\leq M\left(\left(80d\frac{\alpha_d}{\alpha_{d-1}}+20d^3\right)\sqrt{\ep}\cH^{d-1}(\fE)+10d^2\ep\right)n^{d-1}+ n ^{d-1} \sum_{i\in I_1\cup I_2\cup I_3}\mu_n(g_i)+\ep_0 n ^{d-1}
\end{align}
where we recall that by construction of $g_i$ and $\alpha_i$, we have $g_i=1$ on $B_i(\alpha_i)$.
Using the fact that $(E_n,\mu_n)\in U$, inequalities \eqref{eq:af200}, \eqref{eq:af400}, and \eqref{eq:nugi}, it follows that 
\begin{align*}
\sum_{i\in I_1\cup I_2\cup I_3}\mu_n(g_i)
&\leq \sum_{i\in I_1\cup I_2\cup I_3}\left(\nu(g_i)+\ep\alpha_{d-1}r_i^{d-1}\right)\\
&\leq \sum_{i\in I_1\cup I_2\cup I_3}\left(\nu(B(x_i,r_i))+\ep\alpha_{d-1}r_i^{d-1}\right)
\\&\leq \sum_{i\in I_1\cup I_2\cup I_3}\left(\int_{B(x_i,r_i)\cap \fE}f(y)d\cH^{d-1}(y)+2\ep \alpha_{d-1}r_i^{d-1}\right)\\
&\leq  \int_{\fE}f(y) d\cH^{d-1}(y)+\frac{2\ep}{1-\ep} \cH^{d-1}(\fE)\,.
\end{align*}
Moreover, we have $$\nu(\sR^d)\leq \nu(g)\leq \mu_n(g)+ \ep= V(\cE_n)/n^{d-1}+\ep\,.$$
Combining all the three previous inequalities, we get
\begin{align*}
\nu(\sR^d)&\leq  M\left(\left(80d\frac{\alpha_d}{\alpha_{d-1}}+20d^3\right)\sqrt{\ep}\cH^{d-1}(\fE)+10d^2\ep\right)+\int_{\fE}f(y) d\cH^{d-1}(y)+\frac{2\ep}{1-\ep} \cH^{d-1}(\fE)+\ep_0+\ep\,.
\end{align*}
By letting first $\ep_0$ go to $0$ and then $\ep$ go to $0$, we obtain that 
$$\nu(\sR^d)\leq \int_{\fE}f(y) d\cH^{d-1}(y)\,.$$
Since $\nu(\sR^d)\geq \nu_\lambda(\sR^d)=\int_{\fE}f(y) d\cH^{d-1}(y)$, we obtain that 
$$\nu(\sR^d)=\int_{\fE}f(y) d\cH^{d-1}(y)\,.$$
Thus, we have $\nu_s=0$ and $\nu$ is absolutely continuous with respect to $\lambda$.

\noindent {\bf Step 5. We prove that $(E,\nu)$ is minimal.}
Let $F\in\cB(\sR^d)$ such that $F\subset \Omega$ and $\cP(F,\Omega)<\infty$.  
Set $$\fF =(\partial ^* F\cap\Omega)\cup(\partial ^* \Omega \cap (\Gamma^1\setminus \overline{\partial^* F})\cup (\partial ^* F\cap \Gamma^2)\,.$$ 
Let us denote by $\fF_1$ and $\fF_2$ the following sets:
$$\fF_1=\fF\setminus \overline{\fE}\qquad\text{and}\qquad \fF_2=\fF\cap\fE\,.$$
Note that if $x\in\fF_1$, then $d_2(x,\fF_2)>0$. Indeed, if $d_2(x,\fF_2)=0$ then $x\in \overline{\fE}$ and this contradicts the fact that $x\in\fF_1$.
Let $\ep>0$. Let $\delta$ be small enough depending on $\ep$ as prescribed in inequality \eqref{cond:deltalemphi} in lemma \ref{lem:phicassimple}.
By proposition \ref{prop:utilisationvitali}, there exists a finite family of disjoint closed balls $(B(x_i,r_i,v_i))_{i\in I_1\cup I_2\cup I_3}$  of radius smaller than $\ep$ such that for $i\in I_1$, we have $x_i\in \partial ^* F\cap \Omega$, for $i\in I_2$, we have $x_i\in \partial ^* \Omega \cap (\Gamma^1\setminus \overline{\partial^* F})$ and for $i\in I_3$, we have $x_i\in \partial ^* F\cap \Gamma^2$. Moreover, we have that for any $i\in I_1\cup I_2\cup I_3$,
if $x_i\in \cF_1$, then $B(x_i,r_i)\cap \cF_2=\emptyset$, and the following properties hold:
\begin{align}\label{eq:am1}
\cH^{d-1}(\fF\setminus \cup_{i\in I_1\cup I_2\cup I_3}B(x_i,r_i)))\leq \ep\,,
\end{align}
\begin{align}\label{eq:am2}
\forall i \in I_1\cup I_2\cup I_3\quad  \forall\, 0<r\leq r_i\qquad \left|\frac{1}{\alpha_{d-1} r^{d-1}}\cH^{d-1}(\fF\cap B(x_i,r))-1\right|\leq \ep\,,
\end{align}
\begin{align}\label{eq:am3}
\forall i \in I_1\qquad B(x_i,r_i)\subset \Omega\text{ and } \cL^d((F\cap B(x_i,r_i)) \Delta B^-(x_i,r_i,v_i))\leq \ep \alpha_d r_i^d\,,
\end{align}
\begin{align}\label{eq:am3bis}
\forall i \in I_1\text{ such that $x_i\in\fF_2$}\qquad B(x_i,r_i)\subset \Omega\text{ and } \cL^d((E\cap B(x_i,r_i)) \Delta B^-(x_i,r_i,v_i))\leq \ep \alpha_d r_i^d\,,
\end{align}
\begin{align}\label{eq:am4}
\forall i \in I_2\qquad d_2(B(x_i,r_i), F\cup \Gamma ^2)>0 \quad\text{ and } \quad\  \cL^d((\Omega\cap B(x_i,r_i)) \Delta B^-(x_i,r_i,v_i))\leq \ep \alpha_d r_i^d\,,
\end{align}
\begin{align}\label{eq:am5}
\forall i \in I_3\qquad d_2(B(x_i,r_i), \Gamma ^1)>0, \quad  \cL^d((F\cap B(x_i,r_i)) \Delta B^-(x_i,r_i,v_i))\leq \ep \alpha_d r_i^d\nonumber\\\hfill
\text{and}\quad \cL^d((\Omega\cap B(x_i,r_i)) \Delta B^-(x_i,r_i,v_i))\leq \ep \alpha_d r_i^d\,,
\end{align}
\begin{align}\label{eq:am6}
\forall i\in I_2\cup I_3 \qquad  (\Omega\cap B(x_i,r_i,v_i))\Delta B^-(x_i,r_i,v_i))\subset\left\{\,y\in B(x_i,r_i):|(y-x_i)\cdot v_i|\leq \delta \|y-x_i\|_2\,\right\}\,,
\end{align}
\begin{align}\label{eq:am7}
\forall i\in I_1\cup I_2\cup I_3\quad\forall\, 0<r\leq r_i\qquad \left|\frac{1}{\alpha_{d-1} r^{d-1}}\int_{\fF\cap B(x_i,r)}\nu_G(n_\bullet(y))d\cH^{d-1}(y)-\nu_G(v_i)\right|\leq \ep\
\end{align}
where $\bullet= F$ if $x_i\in \partial ^* F$ and $\bullet=\Omega$ if $x_i\in \partial ^* \Omega\setminus \overline{\partial ^* F}$.
Set \[r_{min}=\min_{i\in I_1\cup I_2\cup I_3}r_i\,.\]
Let $g$ be a continuous function compactly supported such that $g=1$ on $\overline{\Omega}$.
For $i\in I_1\cup I_2\cup I_3$, let $g_i$ be a continuous function compactly supported with values in $[0,1]$ such that $g_i=1$ on $B(x_i,r_i(1-\sqrt{\ep}/4))$ and $g_i=0$ on $B(x_i,r_i)^c$. Hence, we have
\begin{align}\label{eq:nugibis}
\nu(g_i)\leq \nu(B(x_i,r_i))\,.
\end{align}
Set $$U=\left\{F\in\cB(\sR^d): \cL^d(F\Delta E)\leq \ep\alpha_d r_{min}^d\right\}\times\left\{\mu\in\cM(\sR^d):\begin{array}{c}\forall i\in I_1\cup I_2\cup I_3\\ \hfill|\mu(g_i)-\nu(g_i)|\leq \ep\alpha_{d-1}r_i^{d-1}\,,\\|\mu(g)-\nu(g)|\leq \ep \end{array}\right\}\,.
$$
Let $\ep_0>0$ we will choose later. Thanks to inequality \eqref{eq:Enu}, we have
\[\limsup_{n\rightarrow \infty}\frac{1}{n^{d-1}}\log \Prb(\exists \cE_n\in\sC_n(\ep_0):\,(\bR(\cE_n),\mu_n(\cE_n))\in U)\geq -K\,.\]
For $i\in I_1\cup I_3$ and $x_i\in \fF_1$, by the same arguments as in the previous step, there exists $\alpha_i$ in $\{\lfloor nr_i\sqrt{\ep}/2 \rfloor/n,\dots,(\lfloor nr_i\sqrt{\ep} \rfloor -1)/n\}$ such that
\begin{align}\label{eq:contcardI1p2}
\card(F\cap \partial^+_n B_i(\alpha_i)\cap\sZ_n^d)+\card(F^c\cap \partial^-_n B_i(\alpha_i)\cap\sZ_n^d)&\leq 16 \sqrt{\ep} \alpha_d r_i^{d-1}n^{d-1}\,.
\end{align}
 For $i\in I_2$ such that $x_i\in \fF_1$, there exists 
  $\alpha_i$ in $\{\lfloor nr_i\sqrt{\ep}/2 \rfloor/n,\dots,(\lfloor nr_i\sqrt{\ep} \rfloor -1)/n\}$ such that \begin{align}\label{eq:contcardI2et3}
\card(\Omega_n\cap \partial^+_n B_i(\alpha_i))+\card(\Omega_n^c\cap \partial^-_n B_i(\alpha_i))&\leq 16 \sqrt{\ep} \alpha_d r_i^{d-1}n^{d-1}\,.
\end{align}
The values of $\alpha_i$ are deterministic and depend on $n$ and $\ep$.
Let us denote by $\cE$ the following event
$$\cE=\bigcap_{\substack{i\in I_1\cup I_2\cup I_3:\\x_i\in\fF_1}}\overline G_n\left(x_i,r_i-\alpha_i,v_i,1,\nu_G(v_i)+\ep+\kappa_d M\sqrt\ep \right)\,$$
where $\kappa_d$ is given by lemma \ref{lem:inclusion}.
By lemma \ref{lem:inclusion}, we have
\begin{align*}
\Prb\left(\cE^c\right)
&\leq\sum_{\substack{i\in I_1\cup I_2\cup I_3:\\x_i\in\fF_1}}\Prb\left(\overline G_n\left(x_i,r_i-\alpha_i,v_i,1,\nu_G(v_i)+\ep+\kappa_d M\sqrt\ep\right)^ c\right)\\
&\leq\sum_{\substack{i\in I_1\cup I_2\cup I_3:\\x_i\in\fF_1}}\Prb\left(\overline G_n\left(x_i,r_i(1-\sqrt\ep),v_i,1,\nu_G(v_i)+\ep\right)^ c\right)\,.
\end{align*}
By lemma \ref{lem:phicassimple}, we have
$$\limsup_{n\rightarrow\infty}\frac{1}{n^{d-1}}\log\Prb(\cE^c)=-\infty\,.$$
Besides, we have
$$\limsup_{n\rightarrow \infty}\frac{1}{n^{d-1}}\log\Prb(\exists \cE_n\in\sC_n(\ep_0):\,(\bR(\cE_n),\mu_n(\cE_n))\in U, \cE^c)\leq \limsup_{n\rightarrow \infty}\frac{1}{n^{d-1}}\log \Prb(\cE^c)=-\infty\,.$$
Hence, by lemma \ref{lem:estimeanalyse}, 
$$\limsup_{n\rightarrow \infty}\frac{1}{n^{d-1}}\log \Prb(\exists \cE_n\in\sC_n(\ep_0):\,(\bR(\cE_n),\mu_n(\cE_n))\in U, \cE)\geq -K\,.$$
Thus, there exists an increasing sequence $(a_n)_{n\geq 1}$ such that
$$\Prb(\exists \cE_{a_n}\in\sC_{a_n}(\ep_0):\,(\bR(\cE_{a_n}),\mu_{a_n}(\cE_{a_n}))\in U,\cE)>0\,.$$
For short, we will write $n$ instead of $a_n$.
On the event $\{\exists \cE_n\in\sC_n(\ep_0):\,(\bR(\cE_n),\mu_n(\cE_n))\in U,\cE\}$, we choose $\cE_n\in\sC_n(\ep_0)$ such that $(\bR(\cE_n),\mu_n(\cE_n))\in U$ (if there are several possible choices, we pick one according to a deterministic rule).
To shorten the notation, we write $\mu_n$ for $\mu_n(\cE_n)$ and $E_n$ for $\bR(\cE_n)$.
As in the previous step, for $i\in I_1\cup I_2\cup I_3$ such that $x_i\in \fF_2$, there exists $\alpha_i$ in $\{\lfloor nr_i\sqrt{\ep}/2\rfloor/n,\dots,(\lfloor nr_i\sqrt{\ep} \rfloor -1)/n\}$ such that
if $i\in I_1$ then
\begin{align}\label{eq:contcardI1}
\card((E_n\cup F)\cap \partial^+_n B_i(\alpha_i)\cap\sZ_n^d)+\card((E_n^c\cup F^c)\cap \partial^-_n B_i(\alpha_i)\cap\sZ_n^d)&\leq 20 \sqrt{\ep} \alpha_d r_i^{d-1}n^{d-1}\,,
\end{align}
if $i\in I_2$, then
\begin{align}\label{eq:contcardI2}
\card(\Omega_n\cap \partial^+_n B_i(\alpha_i))+\card(E_n\cap \partial^-_n B_i(\alpha_i)\cap\sZ_n^d)&\leq 20 \sqrt{\ep} \alpha_d r_i^{d-1}n^{d-1}\,,
\end{align}
finally, if $i\in I_3$, then
\begin{align}\label{eq:contcardI3}
\card(\Omega_n\cap \partial^+_n B_i(\alpha_i)\cap\sZ_n^d)+\card(E_n^c\cap \partial^-_n B_i(\alpha_i)\cap\sZ_n^d)&\leq 20 \sqrt{\ep} \alpha_d r_i^{d-1}n^{d-1}\,.
\end{align}
 We denote by $X^+_1$, $X^+_2$, $X^-_1$ and $X^-_2$ the following set of edges:
 \begin{align*}
X^+_1=\bigcup_{i\in I_1:x_i\in\fF_1 }\{\langle y, z\rangle\in \E_n^d: y\in \partial ^+_nB_i(\alpha_i)\cap F\}\cup \bigcup_{i\in I_2\cup I_3 :x_i\in\fF_1}\{\langle y, z\rangle\in \E_n^d: y\in \partial ^+_nB_i(\alpha_i)\cap \Omega_n\}\,,
 \end{align*}
\begin{align*}
X^+_2=&\bigcup_{i\in I_1:x_i\in\fF_2 }\{\langle y, z\rangle\in \E_n^d: y\in \partial ^+_nB_i(\alpha_i)\cap (F\cup E_n)\}\cup \bigcup_{i\in I_2 \cup I_3:x_i\in\fF_2}\{\langle y, z\rangle\in \E_n^d: y\in \partial ^+_nB_i(\alpha_i)\cap \Omega_n\}\,,
 \end{align*}
  \begin{align*}
X^-_1=\bigcup_{i\in I_1:x_i\in\fF_1 }\{\langle y, z\rangle\in \E_n^d: y\in \partial ^-_nB_i(\alpha_i)\cap F^c\}\cup \bigcup_{i\in I_2\cup I_3 :x_i\in\fF_1}\{\langle y, z\rangle\in \E_n^d: y\in \partial ^-_nB_i(\alpha_i)\cap \Omega^c\}\,,
 \end{align*}
 \begin{align*}
X^-_2=&\bigcup_{i\in I_1:x_i\in\fF_2 }\{\langle y, z\rangle\in \E_n^d: y\in \partial ^-_nB_i(\alpha_i)\cap (F^c\cup E_n^c)\}\cup \bigcup_{i\in I_2 :x_i\in\fF_2}\{\langle y, z\rangle\in \E_n^d: y\in \partial ^-_nB_i(\alpha_i)\cap E_n\}\\
&\qquad\cup\bigcup_{i\in I_3:x_i\in\fF_2 }\{\langle y, z\rangle\in \E_n^d: y\in \partial ^-_nB_i(\alpha_i)\cap  E_n^c\}\,.
 \end{align*}
We can control the number of edges in $X^+_1\cup X^+_2\cup X^-_1\cup X^-_2$ by the same arguments as in the previous step
\begin{align}\label{eq:contcardX2}
\card(X^+_1\cup X^+_2\cup X^-_1\cup X^-_2)&\leq 80 d\frac{\alpha_d}{\alpha_{d-1}}\sqrt{\ep}\cH^{d-1}(\fF)\,.
\end{align}
Let us denote by $\cF_0$ the set of edges with at least one endpoint in 
$\cV_2(\fF\setminus \cup_{i\in I}B_i(\alpha_i), d/n)$.
We have by the same computations as in \eqref{eq:contcardF0b} 
\begin{align}\label{eq:contcardF0}
\card(\cF_0)\leq 10d^2(2d\sqrt{\ep}\cH^{d-1}(\fF)+\ep)n^{d-1}\leq 21d^3\sqrt{\ep}\cH^{d-1}(\fF)n ^{d-1}
\end{align}
for $\ep$ small enough depending on $d$ and $\fF$. On the event $\cE$,
for $i\in I_1\cup I_2\cup I_3$ such that $x_i\in\fF_1$, 
we denote by $\cG_n^i$ the cutset in the definition of the event $\overline G_n(x_i,r_i-\alpha_i,v_i,1,\nu_G(v_i)+\ep+\kappa_dM\sqrt \ep)$ (if there are several possible choices, we choose one according to a deterministic rule).
We claim that $$\cF_n=X^+_1\cup X^+_2\cup X^-_1\cup X^-_2 \cup\cF_0\bigcup_{i\in  I_1\cup I_2\cup I_3: x_i\in\fF_2}(\cE_n\cap B_i(\alpha_i))\bigcup_{i\in  I_1\cup I_2\cup I_3:x_i\in\fF_1}\cG_n^i$$ is a $(\Gamma_n^1,\Gamma_n^2)$-cutset. 

Let $\gamma$ be a path from $\Gamma_n^1$ to $\Gamma_n^2$ in $\Omega_n$. We write $\gamma=(x_0,e_1,x_1\dots,e_p,x_p)$.

\noindent $\triangleright$ Let us assume that $x_0\in B_i(\alpha_i)$ for $i\in I_2$ and $x_i\in\fF_2$ or $x_p\in B_j(\alpha_j)$ for $j\in I_3$ and $x_j\in \fF_2$. Then, by the same arguments as in the previous step, we have $\cF_n\cap \gamma \neq\emptyset$.

\noindent $\triangleright$ Let us assume that $x_0\in B_i(\alpha_i)$ for $i\in I_2$ and $x_i\in\fF_1$. If $x_0\in \partial_n^- B_i(\alpha_i)$, then $x_0 \in \partial_n^- B_i(\alpha_i)\cap \Omega^c$ and $\langle x_0,x_1\rangle\in X_1^-$. Let us assume $x_0\notin \partial_n^- B_i(\alpha_i)$.  Since for $n$ large enough $B_i(\alpha_i)\cap\Gamma_n^2=\emptyset$, the path $\gamma$ eventually exits the ball $B_i(\alpha_i)$.
Set \[r=\inf\left\{k\geq 0: \,x_k\in \partial_n B_i(\alpha_i)\right\}\,.\]
 We have $x_r\in \Omega_n$.
\begin{itemize}[$\cdot$]
\item If $x_r\in \partial^+_n B_i(\alpha_i)$, then $x_r\in \partial^+_n B_i(\alpha_i)\cap \Omega_n$ and $\langle x_r,x_{r+1}\rangle \in X^+_1$.
\item If  $x_r\in \partial^-_n B_i(\alpha_i)$, then since $\cG_n^i$ is a cutset between $\Gamma_n^1\setminus  \partial^-_n B_i(\alpha_i)$ and $ \partial^-_n B_i(\alpha_i)$ in $B_i(\alpha_i)$ then $\gamma\cap\cG_n^i \neq \emptyset$.
\end{itemize}

\noindent $\triangleright$ Let us assume that $x_p\in B_i(\alpha_i)$ for $i\in I_3$ and $x_i\in\fF_1$. By the same reasoning as in the previous case, we can prove that $\gamma\cap \cF_n\neq\emptyset$.

\noindent $\triangleright$ Let us assume $x_0\in \cV_\infty(\Gamma^1\setminus \overline{\partial ^* F},1/n)\setminus \bigcup_{i\in I_2}B_i(\alpha_i)$ then $\langle x_0,x_1\rangle \cap \cF_0\neq \emptyset$.

\noindent$\triangleright$ Let us assume $x_p\in \cV_\infty(\Gamma^2\cap\partial ^* E,1/n)\setminus \bigcup_{i\in I_3}B_i(\alpha_i)$ then $\langle x_{p-1},x_p\rangle \cap \cF_0\neq \emptyset$.

\noindent$\triangleright$ Let us assume $x_0\in \cV_\infty(\Gamma^1\setminus \overline{\partial ^* F},1/n)^c$ and $x_p\in \cV_\infty(\Gamma^2\cap\partial ^* F,1/n)^c$. Set $\widetilde{F}=F\cup (\cV_\infty(F,1/n)\setminus \Omega)$.
It follows that $x_0\in \widetilde F$ and $x_p\in\widetilde F^c$.
Let us assume there exists an excursion of $\gamma$ inside a ball $B_i(\alpha_i)$, for $i\in I_1$ and $x_i\in\fF_1$ that starts at $x_r\in \partial_n B_i(\alpha_i)\cap \widetilde F=\partial_n B_i(\alpha_i)\cap F$ and exits at $x_m\in \partial_n B_i(\alpha_i)\cap \widetilde F^c=\partial_n B_i(\alpha_i)\cap F^c$, where we use that $B_i(\alpha_i)\subset\Omega$.
\begin{itemize}[$\cdot$]
\item If $x_r\in \partial ^+_n B_i(\alpha_i)\cap F$, then $\langle x_{r-1},x_r\rangle \in  X^+_1$.
\item If $x_m\in \partial ^-_n B_i(\alpha_i)\cap F^c$, then $\langle x_{m},x_{m+1}\rangle \in X^-_1$.
\item If $x_r\in \partial ^-_n B_i(\alpha_i)$ and $x_m\in \partial ^+_n B_i(\alpha_i)$, then since $\cG_n^i$  cuts $\partial ^-_n B_i(\alpha_i)$ from $ \partial ^+_n B_i(\alpha_i)$ in $B_i(\alpha_i)$, we have $\gamma\cap \cG_n^i\neq\emptyset$.
\end{itemize}
Let us assume there exists an excursion of $\gamma$ inside a ball $B_i(\alpha_i)$, for $i\in I_1$ and $x_i\in\fF_2$ that starts at $x_r\in \partial_n B_i(\alpha_i)\cap F$ and exits at $x_m\in \partial_n B_i(\alpha_i)\cap F^c$.
\begin{itemize}[$\cdot$]
\item If $x_r\in \partial ^+_n B_i(\alpha_i)\cap F$, then $\langle x_{r-1},x_r\rangle \in  X^+_2$. If $x_r\in \partial ^-_n B_i(\alpha_i)\cap E_n^c$, then $\langle x_{r-1},x_r\rangle \in  X^-_2$. 
\item If $x_m\in \partial ^-_n B_i(\alpha_i)\cap F^c$, then $\langle x_{m},x_{m+1}\rangle \in X^-_2$. If $x_m\in \partial ^+_n B_i(\alpha_i)\cap E_n$, then $\langle x_{m-1},x_m\rangle \in  X^+_2$. 
\item Finally, if $x_r\in E_n$ and $x_m\in E_n^c$, then $\gamma\cap (\cE_n\cap B_i(\alpha_i))\neq \emptyset$.
\end{itemize}

If no such excursion exists then the path $\gamma$ must cross $(\partial \widetilde F\setminus \partial \cV_\infty(\Omega,1))\setminus\cup_{i\in I_1}B_i(\alpha_i)$. We conclude as in the previous step that $\gamma\cap\cF_0\neq\emptyset$.
For all these cases, we have $\gamma\cap\cF_n\neq \emptyset$, thus $\cF_n$ is a $(\Gamma_n^1,\Gamma_n^2)$-cutset.

Note that $B_i(\alpha_i)\subset B(x_i,r_i(1-\sqrt\ep/4))$ and so $g_i=1$ on $B_i(\alpha_i)$. On the event $\cE$, using inequalities \eqref{eq:am2} and \eqref{eq:am7}, we have
\begin{align}\label{eq:controlecapa}
\frac{1}{n ^{d-1} }V&\left(\bigcup_{\substack{i\in  I_1\cup I_2\cup I_3:\\ x_i\in\fF_2}}(\cE_n\cap B_i(\alpha_i))\bigcup_{\substack{i\in  I_1\cup I_2\cup I_3:\\x_i\in\fF_1}}\cG_n^i\right)\nonumber\\
 & \leq  \frac{1}{n ^{d-1} }\sum_{\substack{i\in I_1\cup I_2\cup I_3:\\ x_i\in \fF_2}}V(\cE_n\cap B_i(\alpha_i))+\frac{1}{n ^{d-1} } \sum_{\substack{i\in I_1\cup I_2\cup I_3:\\ x_i\in \fF_1}}V(\cG_n^i)\nonumber\\
 &\leq  \sum_{\substack{i\in I_1\cup I_2\cup I_3:\\ x_i\in \fF_2}}\mu_n(g_i)+\sum_{\substack{i\in I_1\cup I_2\cup I_3:\\ x_i\in \fF_1}}\left(\nu_G(v_i)+\ep+M\kappa_d\sqrt\ep\right)\alpha_{d-1}r_i^{d-1}\nonumber
\\
 &\leq   \sum_{\substack{i\in I_1\cup I_2\cup I_3:\\ x_i\in \fF_2}}\left(\nu(B(x_i,r_i))+\ep\alpha_{d-1}r_i^{d-1}\right)\nonumber\\
 &\qquad+ \sum_{\substack{i\in I_1\cup I_2\cup I_3:\\ x_i\in \fF_1}}\left(\int_{\fF\cap B(x_i,r_i)}\nu_G(n_\bullet(y))d\cH^{d-1}(y) +2M\kappa_d\sqrt \ep \alpha_{d-1}r_i^{d-1}\right)\nonumber\\
 &\leq \int_{\cV_2(\fF_2,\ep)\cap \fE}f(y)d\cH^{d-1}(y)+ \int_{\fF_1}\nu_G(n_\bullet(y))d\cH^{d-1}(y)+ 4M\kappa_d\sqrt \ep \cH^{d-1}(\fF)
 \,.
\end{align}
for small enough $\ep$ depending on $d$ and $M$. We have used the two following facts: 
\begin{enumerate}[(i)]
 \item for $i\in  I_1\cup I_2\cup I_3$ such that $x_i\in \fF_1$, we have by construction of the covering 
$$\fF\cap B(x_i,r_i)=\fF_1\cap B(x_i,r_i)\,,$$
\item for $x_i\in \fF_2$, since $r_i\leq \ep$, we have
$$\fF\cap B(x_i,r_i)\subset\cV_2(\fF_2,\ep)\cap B(x_i,r_i)\,.$$
\end{enumerate}
Besides, we have 
\begin{align}\label{eq:contcardvenbis}
V(\cE_n)=n ^{d-1}\mu_n(g)\geq n^{d-1}( \nu(g)-\ep)= n^{d-1}\left(\int_{\fE}f(y)d\cH^{d-1}(y)-\ep\right)\,.
\end{align}
Using that $\cE_n\in \sC_n(\ep_0)$, inequalities \eqref{eq:contcardX2}, \eqref{eq:contcardF0}, \eqref{eq:controlecapa} and \eqref{eq:contcardvenbis}, it follows  that
\begin{align*}
\int_{\fE}f(y)d\cH^{d-1}(y)-\ep &\leq  \kappa'\sqrt{\ep}\cH^{d-1}(\fF)+\int_{\cV_2(\fF_2,\ep)\cap \fE}f(y)d\cH^{d-1}(y)+\int_{\fF_1}\nu_G(n_\bullet(y))d\cH^{d-1}(y)+\ep_0
 \end{align*}
 where $\kappa'$ depends on $d$ and $M$.
Note that by the dominated convergence theorem, we have
$$\lim_{\ep\rightarrow 0}\int_{\cV_2(\fF_2,\ep)\cap \fE}f(y)d\cH^{d-1}(y)=\int_{\overline{\fF_2}\cap \fE}f(y)d\cH^{d-1}(y)=\int_{\overline{\fF}\cap \fE}f(y)d\cH^{d-1}(y)\,.$$
By letting first $\ep_0$ go to $0$, then $\delta$ go to $0$ and finally $\ep$ go to $0$, we obtain that 
$$\int_{\fE}f(y)d\cH^{d-1}(y)\leq \int_{\overline{\fF}\cap \fE}f(y)d\cH^{d-1}(y)+ \int_{\fF\setminus \overline{\fE}}\nu_G(n_\bullet(y))d\cH^{d-1}(y)\,.$$
The minimality of $(E,f)$ follows. 
\end{proof}
\begin{proof}[Proof of lemma \ref{lem:Zhang}] This proof is just an adaptation of the proof of Zhang in \cite{Zhang2017}. In \cite{Zhang2017}, Zhang controls the cardinal of a minimal cutset that cuts a given box from infinity. Actually, his proof may be adapted to any connected cutset (not necessarily minimal) with a control on its capacity. In our context, we have the following control on the capacity $V(\cE_n)\leq 11d^2M \cH^{d-1}(\Gamma^1) n^{d-1}$. We need the cutset to be connected to be able to upper-bound the number of possible realizations of the cutset $\cE_n$. To solve this issue, we do as in remark 19 in \cite{CT1}, we consider the union of $\cE_n$ with the edges that lie along $\Gamma$: it is always connected, and the number of edges we have added is upper-bounded by $c n^{d-1}$ for a constant $c$ that depends only on the domain $\Omega$ since $\Gamma$ is piecewise of class $\sC^1$ (see proposition \ref{prop:minkowski}). 
 Let $K\geq 0$, by an adaptation of Zhang's theorem, there exist positive constants $C_1$ and $C_2$ such that 
 \begin{align*}
 \forall n\geq 1\qquad \Prb(\exists \cE_n \text{ $(\Gamma_n^1\cup\Gamma_n^2)$-cutset such that $V(\cE_n)\leq 11d^2M \cH^{d-1}(\Gamma^1) n^{d-1}$ and $\card(\cE_n)\geq  \beta n ^{d-1}$})\\\hfill\leq C_1 \exp(-C_2\beta n^{d-1})\,.\end{align*}
 We set $\beta=K/C_2$. The result follows.
\end{proof}

 \section{Lower bound}\label{sec:lb}
 This section corresponds to the step 2 of the sketch of the proof in the introduction. For $(E,\nu)\in\fT_{\cM}$ such that $\widetilde{I}(E,\nu)<\infty$, we cover almost $\fE$ by a family of disjoint balls. In each ball, we build a cutset that is almost flat and has almost the same local capacity than $(E,\nu)$, the cost of this operation may be controlled by lower large deviations for the maximal flow in balls (see lemma \ref{lem:phicassimple}). We then fill the holes by adding a negligible amount of edges to merge all these cutsets inside the balls into a cutset in $\sC_n(\Gamma^1,\Gamma^2,\Omega)$. The main technical difficulty is to ensure that the cutset we have built is almost minimal. To do so, we have to ensure that anywhere outside a small region around $\fE$, the local maximal flow is not abnormally low.
 \begin{prop}[Lower bound]\label{prop:lowerbound}
 Let $(E,\nu)\in\fT_{\cM}$ such that $\widetilde{I}(E,\nu)<\infty$. Then, for any neighborhood $U$ of $(E,\nu)$, we have
 \[\lim_{ \ep\rightarrow 0}\liminf_{n\rightarrow\infty}\frac{1}{n^{d-1}}\log\Prb\left(\exists\cE_n\in\sC_n(\ep):(\bR(\cE_n),\mu_n(\cE_n))\in U\right)\geq -\widetilde{I}(E,\nu)\,.\]
  \end{prop}
\begin{proof}
Let $(E,\nu)\in\fT_{\cM}$ such that $\widetilde{I}(E,\nu)<\infty$. We write $\nu=f\cH^{d-1}|_{\fE}$. Let $\cU$ be a neighborhood of $(E,\nu)$ for the product topology $\cO'\times \cO$. There exists $E$, $\xi_0>0$, there exist $p$ continuous functions $g_1,\dots, g_p$ having compact support and $\xi_1>0$ such that 
$$\cU_0=\{F\in\cB(\sR^d): \cL^d(F\Delta E)\leq \xi_0\}\times\left\{\,\rho\in\cM(\sR^d):\forall i \in \{1,\dots,p\}\quad |\rho(g_i)-\nu(g_i)|\leq \xi_1\,\right\}\subset\cU\,.$$
The functions $g_1,\dots,g_p$ are uniformly continuous. Let $\eta>0$. Let $\ep>0$, we will choose later depending on $\eta$. There exists $\delta_0>0$ such that 
\begin{align}\label{eq:unifcontg}
\forall i\in \{1,\dots,p\}\quad\forall \,x,y\in \sR^d\qquad \|x-y\|_2\leq \delta_0\implies |g_i(x)-g_i(y)|\leq \ep\,.
\end{align}
Thanks to lemma \ref{lem:Zhang}, there exists $\beta>0$ such that
$$\liminf_{n\rightarrow \infty}\frac{1}{n^{d-1}}\log \Prb(\exists \cE_n\in \sC_n(\Gamma^1,\Gamma^ 2,\Omega):\, V(\cE_n)\leq 11d^2M \cH^{d-1}(\Gamma^1) n^{d-1},\,\card(\cE_n)\geq \beta n ^{d-1})\leq -2\widetilde{I}(E,\nu)\,.$$
We define the event
$$\cE^{(0)}=\left\{\exists \cE_n\in \sC_n(\Gamma^1,\Gamma^ 2,\Omega):\, V(\cE_n)\leq 11d^2M \cH^{d-1}(\Gamma^1) n^{d-1},\,\card(\cE_n)\geq \beta n ^{d-1}\right\}^c\,.$$ 

\noindent {\bf Step 1. Build a covering of $\sC_\beta$.}
Let $F\in\sC_\beta$ where we recall that $\sC_\beta$ was defined in \eqref{eq:defcbeta}. By dominated convergence theorem, we have
\begin{align*}
\lim_{\delta\rightarrow 0}\int_{(\fF\setminus \cV_2(\overline{\fE},\delta))}\nu_G(n_{\bullet}(y))d\cH^{d-1}(y)=\int_{(\fF\setminus \overline{\fE})}\nu_G(n_{\bullet}(y))d\cH^{d-1}(y)\,.
\end{align*}
Consequently, there exists $\delta_F >0$ such that
\begin{align}\label{eq:choixdeltaf}
\int_{(\fF\setminus \cV_2(\overline{\fE},\delta_F))}\nu_G(n_{\bullet}(y))d\cH^{d-1}(y)\geq \int_{(\fF\setminus \overline{\fE})}\nu_G(n_{\bullet}(y))d\cH^{d-1}(y)-\ep\,.
\end{align}
We recall that the notation $n_\bullet(x)$ is to lighten the expressions, $\bullet$ will correspond to $E$, $F$ or $\Omega$ depending on $x$.
By proposition \ref{prop:utilisationvitali}, there exists a finite covering $(B(x_i^F,r_i^F,v_i ^F))_{i\in I^F}$ of $(\fF\setminus \cV_2(\fE,\delta_F))\cup (\fE\cap \fF)$ such that $r_i ^F\leq \delta_F/4$ for any $i\in I^F$,
\begin{align}\label{eq:lbf1}
\cH^{d-1}(((\fF\setminus \cV_2(\fE,\delta_F))\cup (\fE\cap \fF))\setminus \cup_{i\in I^F}B(x _i^F,r _i^F))\leq \ep\,,
\end{align}
\begin{align}\label{eq:lbf2}
\forall i \in I^F\quad & \forall\, 0<r\leq r _i^F\qquad \left|\frac{1}{\alpha_{d-1} r^{d-1}}\cH^{d-1}(\fF\cap B(x _i^F,r))-1\right|\leq \ep\nonumber\\
&\text{and   if }x_i^F\in\fF\cap \fE\quad \forall\, 0<r\leq r _i^F\qquad \left|\frac{1}{\alpha_{d-1} r^{d-1}}\cH^{d-1}(\fE\cap B(x _i^F,r))-1\right|\leq \ep\,,
\end{align}
\begin{align}\label{eq:lbf3}
\forall i \in I^F \quad\text{such that}\quad x _i^F\in\partial^*F\cap\Omega\qquad B(x _i^F,r _i^F)\subset \Omega\quad\text{and}\nonumber\\\hfill \cL^d((F\cap B(x _i^F,r _i^F)) \Delta B^-(x _i^F,r _i^F,v _i^F))\leq \ep \alpha_d (r _i^F)^d\,,
\end{align}

\begin{align}\label{eq:lbf5}
\forall i \in I^F\quad\text{such that}\quad x _i^F\in \Gamma^1\setminus \overline{\partial ^* F}\qquad d_2(B(x _i^F,r _i^F), F\cup \Gamma ^2)>0 \quad\text{ and }\nonumber\\\hfill \quad\  \cL^d((\Omega\cap B(x _i^F,r _i^F)) \Delta B^-(x _i^F,r _i^F,v _i^F))\leq \ep \alpha_d (r _i^F)^d\,,
\end{align}
\begin{align}\label{eq:lbf6}
\forall i \in I^F\quad\text{such that}\quad x _i^F\in\partial^*\Omega \cap \Gamma^2\qquad  d_2(B(x _i^F,r _i^F),  \Gamma ^1)>0 \quad\text{ and }\nonumber \\\hfill\quad\  \cL^d((F\cap B(x _i^F,r _i^F)) \Delta B^-(x _i^F,r _i^F,v _i^F))\leq \ep \alpha_d (r _i^F)^d\,,
\end{align}
\begin{align}\label{eq:lbf8}
\forall i\in I^F\quad\text{such that}\quad& x_i^F\in\fF\setminus \fE\quad\forall\, 0<r\leq r _i^F \nonumber\\&\left|\frac{1}{\alpha_{d-1}r^{d-1}}\int_{\fF\cap B(x_i^F,r)}\nu_G(n_\bullet(y))d\cH^{d-1}(y)-\nu_G(v _i^F)\right|\leq \ep\,,
\end{align}
\begin{align}\label{eq:lbf9}
\forall i\in I^F\quad\text{such that}\quad x_i^F\in\fF\cap \fE\quad\forall\, 0<r\leq r _i^F\qquad \left|\frac{1}{\alpha_{d-1}r^{d-1}}\int_{\fE\cap B(x_i^F,r)}f(y)d\cH^{d-1}(y)-f(x_i^F)\right|\leq \ep\,,
\end{align}
\begin{align}\label{eq:lbf10}
\forall i\in I^F\quad\text{such that}&\quad x_i^F\in\fF\cap \fE\quad\forall\, 0<r\leq r _i^F\nonumber\\& \left|\frac{1}{\alpha_{d-1}r^{d-1}}\int_{\fE\cap B(x_i^F,r)}\cJ_{n_\bullet(y)}(f(y))d\cH^{d-1}(y)-\cJ_{v_i^F}(f(x_i^F))\right|\leq \ep\,.
\end{align}
Set $$\ep_F=\ep \alpha_d \min_{i\in I ^F}(r_i^F)^d\,.$$
Since $\sC_\beta$ is compact, we can extract from $(B_{\dis}(F,\ep_F), F\in\sC_\beta)$ a finite covering $(B_{\dis}(F,\ep_{F}),F \in \cA)$ of $\sC_\beta$ with $\card(\cA)<\infty$.
Since $\widetilde{I}(E,\nu)<\infty$ and $\cH^{d-1}(\fE)<\infty$,
we can choose $\delta_1$ small enough such that 
\begin{gather}\label{eq:choixdelta1}
\delta_1\leq \frac{1}{4}\min_{F\in\cA}\delta_F,\nonumber\\
\forall F\in\cA\quad \forall j\in I ^F\text{ such that }x_j^F\in\fF\cap\fE\quad \int_{\fE\cap( B(x_j^F,r_j^F+\delta_1) \setminus B(x_j^F,r_j^F))} \cJ_{n_\bullet(x)}(f(x))d\cH^{d-1}(x)\leq \ep\alpha_{d-1}(r_j^F)^{d-1}\nonumber\\
 \text{and}\quad \cH^{d-1}\left(\fE\cap\left( B(x_j^F,r_j^F+\delta_1) \setminus B(x_j^F,r_j^F)\right)\right)\leq \ep\alpha_{d-1}(r_j^F)^{d-1} \,.
\end{gather}
Set $$r_{min}= \min_{F\in\cA}\min_{i\in I^{F}}r_i^{F}\,.$$

\noindent{\bf Step 2. We build a covering of $\fE$ that depends on the covering of $\sC_\beta$.}
By proposition \ref{prop:utilisationvitali}, there exists a finite covering $(B(x_i^E,r_i ^E,v_i^E))_{i\in I ^E}$ of $\fE$ such that for all $i\in I^E$, $r_i^E\leq \min(r_{min},\delta_0,\delta_1)$ and the following properties hold:
\begin{align}\label{eq:lbe1}
\cH^{d-1}\left(\fE\setminus \bigcup_{i\in I^E}B(x_i^E,r_i^E)\right)\leq \ep\,,
\end{align}
\begin{align}\label{eq:lbe2}
\forall i \in I^E\quad  \forall\, 0<r\leq r_i\qquad \left|\frac{1}{\alpha_{d-1} r^{d-1}}\cH^{d-1}(\fE\cap B(x_i^E,r))-1\right|\leq \ep\,,
\end{align}
\begin{align}\label{eq:lbe3}
\forall i \in I^E \quad\text{such that}\quad x_i^E\in\partial^*E\cap\Omega\qquad B(x_i^E,r_i^E)\subset \Omega\quad\text{and }\nonumber\\\hfill \cL^d((E\cap B(x_i^E,r _i^E)) \Delta B^-(x_i^E,r _i^E,v _i^E))\leq \ep \alpha_d (r _i^E)^d\,,
\end{align}
\begin{align}\label{eq:lbe4}
\forall i \in I^E\quad\text{such that}\quad x_i^E\in \Gamma^1\setminus \overline{\partial ^* E}\qquad d_2(B(x_i^E,r _i^E), \Gamma^2\cup E)>0 \quad\text{ and }\nonumber\\\hfill \quad\  \cL^d((\Omega\cap B(x_i^E,r _i^E)) \Delta B^-(x_i^E,r _i^E,v _i^E))\leq \ep \alpha_d (r _i^E)^d\,,
\end{align}
\begin{align}\label{eq:lbe5}
\forall i \in I^E\quad\text{such that}\quad x_i^E\in\partial^*E \cap \Gamma^2\qquad d_2(B(x_i^E,r _i^E), \Gamma^1)>0,\nonumber\\\hfill \quad\  \cL^d((E\cap B(x_i^E,r _i^E)) \Delta B^-(x_i^E,r _i^E,v _i^E))\leq \ep \alpha_d (r _i^E)^d\\
\text{and}\quad \cL^d((\Omega\cap B(x_i^E,r _i^E)) \Delta B^-(x_i^E,r _i^E,v _i^E))\leq \ep \alpha_d (r _i^E)^d\,,
\end{align}
\begin{align}\label{eq:lbe6}
\forall i\in I^E\quad\text{such that}\quad x_i^E\in\partial^* \Omega \qquad  (\Omega\cap B(x_i^E,r_i^E)\Delta B^+ (x_i^E,r_i^E,v_i^E))\nonumber\\ \hfill \subset\left\{\,y\in B(x_i^E,r_i^E):|(y-x_i^E)\cdot v_i^E|\leq \delta \|y-x_i^E\|_2\right\}\,,
\end{align}
\begin{align}\label{eq:lbe7}
\forall i\in I^E\quad\forall\, 0<r\leq r _i^E\qquad \left|\frac{1}{\alpha_{d-1}r^{d-1}}\int_{\fE\cap B(x_i^E,r)}f(y)d\cH^{d-1}(y)-f(x_i^E)\right|\leq \ep\,,
\end{align}
\begin{align}\label{eq:lbe8}
\forall i\in I^E\quad\forall\, 0<r\leq r _i^E\qquad \left|\frac{1}{\alpha_{d-1}r^{d-1}}\int_{\fE\cap B(x_i^E,r)}\cJ_{n_\bullet(y)}(f(y))d\cH^{d-1}(y)-\cJ_{v_i^E}(f(x_i^E))\right|\leq \ep\,.
\end{align}
We need to slightly reduce the $r _i^E$.
Using inequality \eqref{eq:lbe3}, 
for $i\in I^E$ and $x_i^E\in\fE\setminus \partial ^* \Omega$, by the same arguments as in the previous section, there exists $\alpha_i$ in $\{1/n,\dots,(\lfloor nr _i^E\sqrt{\ep} \rfloor -1)/n\}$ such that
\begin{align}\label{eq:contcardI1p2b}
\card(E\cap \partial^+_n B_i(\alpha_i)\cap\sZ_n^d)+\card(E^c\cap \partial^-_n B_i(\alpha_i)\cap\sZ_n^d)&\leq 8 \sqrt{\ep} \alpha_d (r _i^E)^{d-1}n^{d-1}\,,
\end{align}
where $B_i(\alpha_i)=B(x_i^E,r_i^E-\alpha_i,v_i^E)$.
Using inequalities \eqref{eq:lbe4} and \eqref{eq:lbe5}, for $i\in I^E$ such that $x_i\in \partial^* \Omega$, there exists 
  $\alpha_i$ in $\{1/n,\dots,(\lfloor nr _i^E\sqrt{\ep} \rfloor -1)/n\}$ such that 
\begin{align}\label{eq:contcardI2et3b}
\card(\Omega_n\cap \partial^+_n B_i(\alpha_i))+\card(\Omega^c\cap \partial^-_n B_i(\alpha_i)\cap\sZ_n^d)&\leq 8 \sqrt{\ep} \alpha_d (r _i^E)^{d-1}n^{d-1}\,.
\end{align}
The values of $\alpha_i$ are deterministic and depend on $n$.
We denote by $X^+$ and $X^-$ the following set of edges:
\begin{align*}
X^+=\bigcup_{i\in I^E: x_i^E\in \fE\setminus \partial^* \Omega }\{\langle y, z\rangle\in \E_n^d: y\in \partial ^+_nB_i(\alpha_i)\cap E\}\cup \bigcup_{i\in I^E:x_i^E\in\partial ^* \Omega }\{\langle y, z\rangle\in \E_n^d: y\in \partial ^+_nB_i(\alpha_i)\cap \Omega_n\}\,,
 \end{align*}
 \begin{align*}
X^-=\bigcup_{i\in I^E: x_i^E\in\fE\setminus \partial ^* \Omega }\{\langle y, z\rangle\in \E_n^d: y\in \partial ^-_nB_i(\alpha_i)\cap E^c\}\cup \bigcup_{i\in I^E:x_i^E\in\partial ^* \Omega }\{\langle y, z\rangle\in \E_n^d: y\in \partial ^-_nB_i(\alpha_i)\cap \Omega^c\}\,.
 \end{align*}
By similar computations as in \eqref{eq:contcardX}, using \eqref{eq:lbe2},
we have
\begin{align}\label{eq:lbcontx}
\card(X^-\cup X^+)\leq 80 d\frac{\alpha_d}{\alpha_{d-1}}\sqrt{\ep}\cH^{d-1}(\fE)n^{d-1}\leq \eta n^{d-1}
\end{align} 
 for $\ep$ small enough depending on $\eta$, $d$ and $\fE$.

\noindent {\bf Step 3. We build a configuration on which the event $\{\exists \cE_n \in\sC_n(u(\eta)):(\bR(\cE_n),\mu_n(\cE_n))\in U\}$ occurs.}
  Set \[\ep'= \frac{\alpha_d}{\alpha_{d-1}} 2^{-d-1}\ep\]
and $\bar r_i^{E}=(1-\sqrt{\ep})r_i^E$, for $i\in I^E$.
Set
$$\cE^{(1)}=\bigcap_{F\in \cA}\bigcap_{\substack{i\in I ^F:\\ x_i^F\in \fF \setminus \fE}} G_n(x_i ^F,r_i^F,v_i^F,\ep,(1-\eta)\nu_G(v_i^F))^c\,,$$ 
$$\cE^{(2)}=\bigcap_{F\in \cA}\bigcap_{\substack{i\in I ^F:\\ x_i^F\in \fF \cap  \fE, \,f(x_i^F)\geq \eta}} G_n(x_i ^F,r_i^F,v_i^F,\ep,(1-\eta) f(x_i^F))^c\,,$$
$$\cE^{(3)}= \bigcap_{i\in I^E:\, f(x_i^E)\geq \eta} G_n(x_i ^E,\bar r_i ^E,v_i^E,\ep,(1-\eta) f(x_i^E))^c$$
and 
$$\cE^{(4)}= \bigcap_{i\in I^E} \overline G_n(x_i ^E,\bar r_i ^E,v_i^E,\ep',f(x_i^E)+\eta)\,.$$
The first event ensures that the flow is not abnormally low outside $\fE$. The three other events ensure that the local flow near $\fE$ is close to $f$.
For any $F\in\cA$, for any $i\in I ^F$ such that $x_i^F\in \fF \setminus \fE$, we have by construction $x_i^F \in \fF\setminus \cV_2(\fE,\delta_F)$, $r_i^F\leq \delta_F/4$ and for all $j\in I^E$, $r_j^E\leq \delta_F/4$. It follows that
$$B(x_i^F,r_i^F)\cap \left(\bigcup_{j\in I^E}B(x_j^E,\bar r_j ^E)\right)=\emptyset\,.$$
Hence, the event $\cE^{(1)}$ is independent of the event $\cE^{(4)}$. 
We claim that on the event 
$\cE^{(0)}\cap\cE^{(1)}\cap \cE^{(2)}\cap\cE^{(3)}\cap\cE^{(4)}$,
the following event occurs
$$\left\{\exists \cE_n \in\sC_n(u(\eta)):(\bR(\cE_n),\mu_n(\cE_n))\in U\right\}\,$$
for some function $u$ that goes to $0$ when $\eta$ goes to $0$.

 {\bf We build the cutset.} By lemma \ref{lem:inclusion}, we have $\overline G_n(x_i ^E,\bar r_i ^E,v_i^E,\ep',f(x_i^E)+\eta)\subset \overline G_n(x_i ^E,r_i ^E-\alpha_i,v_i^E,\ep',f(x_i^E)+\eta+M\kappa_d\sqrt{\ep})$ where $\kappa_d\geq 1$ is given by lemma \ref{lem:inclusion}. On the event $\cE^{(4)}$, for any $i\in I^E$, denote by $\cE_n(i)$ the cutset corresponding to the event $ \overline G_n(x_i ^E, r_i ^E-\alpha_i,v_i^E,\ep',f(x_i^E)+\eta+M\kappa_d\sqrt\ep)$.
Denote by $\cF_0$ the edges that have at least one endpoint in 
$$\cV_2\left(\fE\setminus \bigcup_{i\in I^E}B(x_i^E,r_i^E-\alpha_i),\frac{d}{n}\right)\,.$$
The set $\cE_n=\cF_0 \cup_{i\in I^E}\cE_n(i)\cup X^+ \cup X^-$ is in $\sC_n(\Gamma^1,\Gamma^2,\Omega)$. We do not prove that it is a $(\Gamma_n^1,\Gamma_n^2)$-cutset since the proof is very similar to the previous proofs.

{\bf We control the capacity of the cutset $\cE_n$ inside the balls $B(x_i^E,r_i^E-\alpha_i)$, for $i\in I^E$.}
Since  $\cE_n(i)$ is a cutset corresponding to the event $ \overline G_n(x_i ^E, r_i ^E-\alpha_i,v_i^E,\ep',f(x_i^E)+\eta+M\kappa_d\sqrt\ep)$, it follows that
$$\frac{V(\cE_n(i))}{\alpha_{d-1}( r_i ^E-\alpha_i)^{d-1} n^{d-1}}\leq f(x_i^E)+\eta+M\kappa_d\sqrt{\ep}\,.$$
If $f(x_i^E)<\eta$, it is enough to conclude that 
$$\left|\frac{V(\cE_n(i))}{\alpha_{d-1}( r_i ^E-\alpha_i)^{d-1} n^{d-1}}-f(x_i^E)\right|\leq \eta+M\kappa_d\sqrt \ep$$
occurs. Otherwise, let us define
\begin{align*}
\br(\cE_n(i))=\left\{\,x\in \sZ_n^d\cap B(x_i^E,\bar r_i^E,v_i^E): \begin{array}{c}\text{ there exists a path from $x$ to $\partial^-_n B(x_i^E,\bar r_i^E,v_i^E)$}\\ \text{ in $(\sZ_n^d\cap B(x_i^E,\bar r_i^E,v_i^E),\E_n^d\setminus \cE_n(i))$}\end{array}\right\}\,.
\end{align*}
Since $\cE_n(i)\cap B(x_i^E,\bar r_i^E)\subset \cyl(\disc(x_i^E,\bar r_i^E,v_i^E),2\ep' r_i^E)$, we have
\[\br(\cE_n(i))\Delta B^-(x_i^E,\bar r_i^E,v_i^E)\subset \cyl(\disc(x_i^E,\bar r_i^E,v_i^E),2\ep' r_i^E)\cap \sZ_n^d\,.\]
It follows that for $n$ large enough
\[\card(\br(\cE_n(i))\Delta B^-(x_i^E,\bar r_i^E,v_i^E))\leq 8\ep' \alpha_{d-1} (r_i^E)^d n^d \leq  4\ep \alpha_d (\bar r_i ^E)^dn^d\]
where we recall that 
\begin{align}\label{eq:defep'}
\ep'= \frac{\alpha_d}{\alpha_{d-1}} 2^{-d-1}\ep\,.
\end{align}
Hence, if $V(\cE_n(i)\cap B(x_i^E,\bar r_i^E))\leq (1-\eta) f(x_i^E)\alpha_{d-1}(\bar r_i^ E)^{d-1}n^{d-1}$ then the event 
$G_n(x_i ^E,\bar r_i ^E,v_i^E,\ep,(1-\eta) f(x_i^E))$ occurs. 
Consequently, on the event $ \overline G_n(x_i ^E,\bar r_i ^E,v_i^E,\ep',f(x_i^E)+\eta)\cap  G_n(x_i ^E,\bar r_i ^E,v_i^E,\ep,(1-\eta) f(x_i^E))^c$, we have 
$$V(\cE_n(i))\geq (1-\eta) f(x_i^E)\alpha_{d-1}(\bar r_i ^E)^{d-1} n^{d-1}\,$$
and
$$\frac{V(\cE_n(i))}{\alpha_{d-1}( r_i ^E-\alpha_i)^{d-1} n^{d-1}}\geq (1-\eta)(1-\sqrt\ep)^{d-1}f(x_i^E)\geq (1-\eta)(1-(d-1)\sqrt\ep)f(x_i^E)\geq f(x_i^E)-(\eta-(d-1)\sqrt\ep)M\,.$$
It follows that on the event $\cE^{(3)}\cap\cE^{(4)}$
\begin{align}\label{eq:controlevei}
\forall i\in I^E\qquad \left|\frac{V(\cE_n(i))}{\alpha_{d-1}( r_i ^E-\alpha_i)^{d-1} n^{d-1}}-f(x_i^E)\right| \leq \eta(1+M\kappa_d)\,.
\end{align}
where we choose $\ep$ small enough depending on $\eta$ and $d$ (we recall that $\kappa_d\geq 1$).

{\bf We prove that $(E_n,\mu_n(\cE_n))\in\cU_0$.}
For $i\in I^E$ such that $x_i^E\in \partial ^* E \cap\Omega$,  we have by construction that 
$\bR (\cE_n)\cap B(x_i^E, r_i ^E))\Delta  B^-(x_i^E, r_i ^E,v_i^E))\subset \cyl(\disc(x_i^E, r_i ^E,v_i^E),2\ep'  r_i ^E)$.
Using \eqref{eq:lbe3}, we have
\begin{align*}
\cL^d&((\bR(\cE_n)\cap B(x_i^E, r_i ^E))\Delta (E\cap B(x_i^E, r_i ^E)))\\
&\leq \cL ^d((\bR(\cE_n)\cap B(x_i^E, r_i ^E))\Delta  B^-(x_i^E, r_i ^E,v_i^E))+\cL^d(E\cap B(x_i^E, r_i ^E))\Delta  B^-(x_i^E, r_i ^E,v_i^E))\\
&\leq (4\ep' \alpha_{d-1}+\ep \alpha_ d)(r_i^E)^d\,.
\end{align*}
It is easy to check that for any $i\in I^E$
$$\cL^d((\bR(\cE_n)\cap B(x_i^E, r_i ^E))\Delta (E\cap B(x_i^E, r_i ^E)))\leq (4\ep' \alpha_{d-1}+\ep \alpha_ d)(r_i^E)^d\,.$$
It follows that for $n$ large enough, we have using proposition \ref{prop:minkowski}
\begin{align*}
\cL^d(\bR(\cE_n)\Delta E)&\leq \cL^d(\cV_2(\fE,d/n))+\sum_{i\in I^E}\cL^d((\bR(\cE_n)\cap B(x_i^E, r_i ^E))\Delta (E\cap B(x_i^E, r_i ^E)))\\
&\leq \frac{4d}{n}\cH^{d-1}(\fE)+  \sum_{i\in I^E}
(4\ep' \alpha_{d-1}+\ep \alpha_ d)(r_i^E)^d\\
&\leq 2 \ep\sum_{i\in I^E}\alpha_{d}( r_i ^E) ^d\leq 2\ep \cL^d(\cV_2(\Omega,1))
\end{align*}
where we use in the last inequality that the balls are disjoint and we recall that $\ep'$ was defined in \eqref{eq:defep'}. Then, we choose $\ep$ small enough depending on $\Omega$, $d$ and $\xi_0$ such that 
\begin{align}\label{eq:choixep0xi}
2\ep \cL^d(\cV_2(\Omega,1))\leq \xi_0\,.
\end{align}
Let us compute the cardinal of the set $\cF_0$.
By proposition \ref{prop:minkowski} and inequality \eqref{eq:lbe1} and \eqref{eq:lbe2} and by the same computations as in \eqref{eq:contcardF0b}, for $n$ large enough
\begin{align}\label{eq:contfnlb}
\card(\cF_0)&\leq 10d^2\cH^{d-1}\left(\fE\setminus \cup_{i\in I^E}B(x_i^E,r_i ^E-\alpha_i)\right)n^{d-1}\leq 10d^2\left(2d\sqrt \ep \cH^{d-1}(\fE)+\ep \right)n^{d-1}\leq \eta\, n^{d-1}
\end{align}
for $\ep$ small enough depending on $\eta$, $d$ and $\cH^{d-1}(\fE)$.
Let $j\in\{1,\dots,p\}$. 
Using \eqref{eq:unifcontg} and the fact that $ r_i ^E\leq \delta_0$ for any $i\in I^E$, we have by similar computations as in \eqref{eq:calculrbar}, using \eqref{eq:lbe1} and \eqref{eq:lbe2}
\begin{align}\label{eq:comblb1}
&\left|\nu(g_j)-\sum_{i\in I_E}\nu(B(x_i^E,r_i^E-\alpha_i))g_j(x_i^E)\right|\nonumber\\
&\quad\leq \|g_j\|_\infty \nu\left(\fE\setminus \cup_{i\in I^E}B(x_i^E,r_i^E-\alpha_i)\right)+\sum_{i\in I^E}\int_{B(x_i^E,r_i^E-\alpha_i)\cap\fE}|g_j(x)-g_j(x_i^E)|f(x)d\cH^{d-1}(x)\nonumber\\
&\quad\leq \|g_j\|_\infty M \cH^{d-1}(\fE\setminus \cup_{i\in I^E}B(x_i^E, r_i^E-\alpha_i))+\ep \int_{\fE}f(x)d\cH^{d-1}(x)\nonumber\\
&\quad\leq \|g_j\|_\infty M \left( \cH^{d-1}(\fE\setminus \cup_{i\in I^E}B(x_i^E, r_i ^E))+\sum_{i\in I^E}\left(\cH^{d-1}(B(x_i^E, r_i ^E)\cap\fE)-\cH^{d-1}(B(x_i^E, \bar r_i^E))\cap\fE)\right)\right)\nonumber\\&\qquad\quad+
\ep\nu(\overline \Omega)\nonumber\\
&\quad\leq \|g_j\|_\infty M (\ep +2d\sqrt \ep \cH^{d-1}(\fE))+ \ep\nu(\overline \Omega)\,.
\end{align}
Similarly, we have using \eqref{eq:contfnlb}
\begin{align}\label{eq:comblb2}
&\left|\mu_n(g_j)-\sum_{i\in I_E}\mu_n(B(x_i^E,r_i^E-\alpha_i))g_j(x_i^E)\right|\nonumber\\
&\qquad\leq \|g_j\|_\infty \mu_n\left(\Omega\setminus \cup_{i\in I^E}B(x_i^E,r_i^E-\alpha_i)\right)+\sum_{i\in I_E}\left(\mu_n(g_j\ind_{B(x_i^E,r_i^E-\alpha_i)})-\mu_n(B(x_i^E,r_i^E-\alpha_i))g_j(x_i^E)\right)\nonumber\\
&\qquad\leq M\|g_j\|_\infty \frac{1}{n ^{d-1}}(\card(\cF_0)+\card(X^+ \cup X^-))+\ep\mu_n(\overline \Omega)\,.
\end{align}
Let $i\in I^E$. Using \eqref{eq:lbe7} and \eqref{eq:controlevei}, we have
\begin{align*}
&\left|\mu_n(B(x_i^E,r_i^E-\alpha_i))-\nu(B(x_i^E,r_i^E-\alpha_i))\right|\\
&\qquad\leq \left|\mu_n(B(x_i^E,r_i^E-\alpha_i))-f(x_i^E)\alpha_{d-1}(r_i^E-\alpha_i)^{d-1}\right|+\left|f(x_i^E)\alpha_{d-1}(r_i^E-\alpha_i)^{d-1}-\nu(B(x_i^E,r_i^E-\alpha_i))\right|\\
&\qquad\leq \left(\eta \left(1+M\kappa_d\right)+ \ep \right)\alpha_{d-1}(r_i^E-\alpha_i)^{d-1}\,.
\end{align*} 
It follows using \eqref{eq:lbe2} that 
\begin{align}\label{eq:comblb3}
\sum_{i\in I ^E}\left|\mu_n(B(x_i^E,r_i^E-\alpha_i))g_j(x_i^E)-\nu(B(x_i^E,r_i^E-\alpha_i))g_j(x_i^E)\right|\leq 8\eta(1+M\kappa_d)\cH^{d-1}(\fE)\|g_j\|_\infty\,
\end{align}
where we recall that $\ep\leq \eta$.
Combining inequalities \eqref{eq:lbcontx}, \eqref{eq:contfnlb}, \eqref{eq:comblb1}, \eqref{eq:comblb2} and \eqref{eq:comblb3}, we have
\begin{align}\label{eq:defC0}
|\mu_n(g_j)-\nu(g_j)|\leq C_0\eta
\end{align}
where $C_0$ depends on $d$, $\fE$, $g_j$ and $M$.
For $\eta\leq \xi_1/C_0$, we have
$|\mu_n(g_j)-\nu(g_j)|\leq\xi_1$. Finally, we have $(E_n,\mu_n)\in\cU_0\subset \cU$.

{\bf We prove that $\cE_n\in\sC_n(u(\eta))$.} Let $\cG_n\in\sC_n(\Gamma^1,\Gamma^2,\Omega)$. If $V(\cG_n)\geq 11d^2M \cH^{d-1}(\Gamma^1) n^{d-1}$ then $V(\cG_n)\geq V(\cE_n)$. If $V(\cG_n)\leq 11d^2M \cH^{d-1}(\Gamma^1) n^{d-1}$, then on the event $\cE^{(0)}$, we have $\card(\cG_n)\leq \beta n^{d-1}$ and $\bR(\cG_n)\in\sC_\beta$. Let $F\in\cA$ such that $\cL^d(\bR(\cG_n)\Delta F)\leq \ep_F$. 
Let $i\in I^F$ such that $x_i^F\in\fF$. Thanks to the choice of $\ep_F$ and inequality \eqref{eq:lbf3},
when $x_i^F\in\Omega$ and $B(x_i^F,r_i ^F)\subset \Omega$, we can prove that
if $$V(\cG_n \cap B(x_i^F,r_i^F))\leq (1-\eta)\nu_G(v_i^F)\alpha_{d-1}(r_i^F)^{d-1}n^{d-1}\,,$$ then, the event $G_n(x_i ^F,r_i^F,v_i^F,\ep,(1-\eta)\nu_G(v_i^F))$ occurs. 
 By inequalities \eqref{eq:lbf5} and \eqref{eq:lbf6}, when $x_i^F\notin \Omega$, we can prove as in section 5.2 in \cite{CT3} that the same result holds.
Similarly, for $x_i^F\in\fF\cap \fE$ such that $f(x_i^F)\geq \eta$,  
if $$V(\cG_n \cap B(x_i^F,r_i^F))\leq (1-\eta)f(x_i^F)\alpha_{d-1}(r_i^F)^{d-1}n^{d-1}$$ then the event $G_n(x_i ^F,r_i^F,v_i^F,\ep,(1-\eta)f(x_i^F))$ occurs.
Hence, on the event $\cE^{(1)}\cap\cE^{(2)}$, using the fact that the balls are disjoint, we have 
\begin{align*}
\frac{V(\cG_n)}{n^{d-1}}&\geq \sum_{i\in I ^F}\frac{V(\cG_n\cap B(x_i^F,r_i^F))}{n^{d-1}}\\
&\geq (1-\eta)\sum_{i\in I^F: x_i ^F \in\fF\setminus \fE}\nu_G(v_i^F)\alpha_{d-1}(r_i^F)^{d-1}+ (1-\eta)\sum_{\substack{i\in I^F:\\ x_i ^F \in\fF\cap \fE,\,f(x_i^F)\geq \eta}}f(x_i^F)\alpha_{d-1}(r_i^F)^{d-1}\,.
\end{align*}
Besides, we have using inequalities \eqref{eq:lbf1}, \eqref{eq:lbf2} and \eqref{eq:lbf9}
\begin{align*}
\sum_{\substack{i\in I^F:\\ x_i ^F \in\fF\cap \fE,\,f(x_i^F)\geq \eta}}f(x_i^F)\alpha_{d-1}(r_i^F)^{d-1}&=\sum_{\substack{i\in I^F:\\ x_i ^F \in\fF\cap \fE}}f(x_i^F)\alpha_{d-1}(r_i^F)^{d-1}- \sum_{\substack{i\in I^F:\\ x_i ^F \in\fF\cap \fE,\,f(x_i^F)< \eta}}f(x_i^F)\alpha_{d-1}(r_i^F)^{d-1}\\
&\geq \sum_{\substack{i\in I^F:\\ x_i ^F \in\fF\cap \fE}}\left(\int_{B(x_i^F,r_i^F)\cap \fF}f(x)d\cH^{d-1}(x) - \ep\alpha_{d-1}(r_i^F)^{d-1}\right) \\
&\qquad-\eta \sum_{\substack{i\in I^F:\\ x_i ^F \in\fF\cap \fE,\,}}\alpha_{d-1}(r_i^F)^{d-1}\\
&\geq \int_{\fF\cap \fE}f(x)d\cH^{d-1}(x) -\ep M  -2(\ep+\eta) \cH^{d-1}(\fF\cap\fE)\\
&\geq \int_{\fF\cap \fE}f(x)d\cH^{d-1}(x) -\eta( M  +4(\beta +\cH^{d-1}(\partial \Omega))
\end{align*}
where we recall that $\ep\leq \eta$ and $E,F\in\sC_\beta$.
Using inequalities \eqref{eq:choixdeltaf}, \eqref{eq:lbf1}, \eqref{eq:lbf2} and \eqref{eq:lbf8} we have
\begin{align*}
\sum_{i\in I^F: x_i ^F \in\fF\setminus \fE}&\nu_G(v_i^F)\alpha_{d-1}(r_i^F)^{d-1}\\
&\geq \sum_{i\in I^F: x_i ^F \in\fF\setminus \fE}\left(\int_{\fF\cap B(x_i^F,r_i^F)}\nu_G(n_\bullet(y))d\cH^{d-1}(y)-\ep\alpha_{d-1}(r_i^F)^{d-1}\right)\\
&\geq\int_{\fF\setminus \cV_2(\overline{\fE},\delta_F)}\nu_G(n_\bullet(y))d\cH^{d-1}(y)-M\cH^{d-1}\left(\left(\fF\setminus \cV_2(\fE,\delta_F)\right)\setminus\bigcup_{i\in I^F}B(x_i^F,r_i^F)\right)-2\ep\cH^{d-1}(\fF)\\
&\geq\int_{\fF\setminus \cV_2(\overline{\fE},\delta_F)}\nu_G(n_\bullet(y))d\cH^{d-1}(y)-2\eta(M+\beta+\cH^{d-1}(\partial \Omega))\\
&\geq  \int_{(\fF\setminus \overline{\fE})}\nu_G(n_{\bullet}(y))d\cH^{d-1}(y)-\ep-2\eta(M+\beta+\cH^{d-1}(\partial \Omega))
\end{align*}
where we use in the second last inequality that $F\in\sC_\beta$.
Combining the three previous inequalities, using the fact that $(E,f)$ is minimal, we obtain
\begin{align*}
\frac{V(\cG_n)}{n^{d-1}}&\geq (1-\eta)\capa(E,f)-\eta(3M+6\beta+6\cH^{d-1}(\partial\Omega))\,.
\end{align*}
Note that $\capa(E,f)\leq 10dM\cH^{d-1}(\Gamma^1)$. It follows that
\begin{align}\label{eq:lbcompengn1}
\frac{V(\cG_n)}{n^{d-1}}&\geq \capa(E,f)-\eta(3M+6\beta+6\cH^{d-1}(\partial^*\Omega)-10dM\cH^{d-1}(\Gamma^1))\,.
\end{align}
Besides, on the event $\cE^{(4)}$, using \eqref{eq:lbe7}, \eqref{eq:lbcontx}, \eqref{eq:controlevei} and \eqref{eq:contfnlb},
we have
\begin{align}\label{eq:lbcompengn2}
\frac{V(\cE_n)}{n^{d-1}}&\leq \sum_{i\in I^E}\frac{V(\cE_n(i))}{n^{d-1}}+\frac{M}{n^{d-1}}(\card(\cF_0)+\card(X^ +\cup X^-))\nonumber\\
&\leq\sum_{i\in I^E}(f(x_i^E)+\eta(1+M\kappa_d))\alpha_{d-1}( r_i ^E) ^{d-1}+2\eta M\nonumber\\
&\leq \capa(E,f)+\sum_{i\in I^E}(\ep+\eta(1+M\kappa_d))\alpha_{d-1}( r_i ^E) ^{d-1}+2M\eta\nonumber\\
&\leq \capa(E,f)+\eta(3+M\kappa_d) \cH^{d-1}(\fE)+2M\eta\nonumber\\
&\leq \capa(E,f)+\eta\left((3+M\kappa_d) (\beta+\cH^{d-1}(\partial  \Omega))+2M\right)\,.
\end{align}
Finally, combining inequalities \eqref{eq:lbcompengn1} and \eqref{eq:lbcompengn2}, there exists a function $u$ (depending on $M$, $\beta$, $\Omega$) such that $\lim_{\eta\rightarrow 0 }u(\eta)=0$ and on the event $\cE^{(1)}\cap\cE^{(2)}\cap\cE^{(4)}$, for any $\cG_n\in\sC_n(\Gamma^1,\Gamma^2,\Omega)$
$$V(\cG_n)\geq V(\cE_n)-u(\eta)n^{d-1}\,.$$
 It follows that $\cE_n\in\sC_n(u(\eta))$. The claim is thus proved.

\noindent {\bf Step 4. We conclude by estimating the probabilities of the events.} We have 
\begin{align}\label{eq:basicub}
\Prb\left(\cE^{(4)}\cap \cE^{(1)}\right)=\Prb(\cE^{(2)}\cap\cE^{(3)}\cap \cE^{(4)}\cap \cE^{(1)})+\Prb \left(((\cE^{(2)})^c\cap\cE^{(4)}\cap \cE^{(1)})\cup((\cE^{(3)})^c\cap\cE^{(4)}\cap \cE^{(1)})\right)\,.
\end{align}
Moreover, we have
\begin{align}\label{eq:unionboundprob}
\Prb &\left(((\cE^{(2)})^c\cap\cE^{(4)}\cap \cE^{(1)})\cup((\cE^{(3)})^c\cap\cE^{(4)}\cap \cE^{(1)})\right)\nonumber\\
&\leq \Prb ((\cE^{(2)})^c\cap\cE^{(4)})+\Prb((\cE^{(3)})^c\cap\cE^{(4)})\nonumber\\
&\leq \sum_{F\in \cA}\sum_{\substack{j\in I ^F:\\ x_j^F\in \fF \cap  \fE\,,f(x_j^F)\geq\eta}} \Prb \left(\bigcap_{i\in I^E} \overline G_n(x_i^E,\bar r_i ^E,v_i^E,\ep',f(x_i^E)+\eta) \cap  G_n(x_j ^F,r_j^F,v_j^F,\ep,(1-\eta) f(x_j^F))\right)\nonumber\\
&\quad + \sum_{\substack{j\in I ^E:\\f(x_j^E)\geq\eta}}\Prb\left(\bigcap_{i\in I^E} \overline G_n(x_i^E,\bar r_i ^E,v_i^E,\ep',f(x_i^E)+\eta)\cap G_n(x_j ^E,\bar r_j ^E,v_j^E,\ep,(1-\eta) f(x_j^E))\right)\,.
\end{align}
Let us estimate $\Prb((\cE^{(3)})^c\cap\cE^{(4)})$.
Let $j\in I^E$ such that $f(x_j^E)\geq\eta$, since the balls are disjoint, we have using the independence 
\begin{align}\label{eq:431}
\log\Prb&\left(\bigcap_{i\in I^E} \overline G_n(x_i^E,\bar r_i ^E,v_i^E,\ep,f(x_i^E)+\eta)\cap G_n(x_j ^E,\bar r_j ^E,v_j^E,\ep,(1-\eta) f(x_j^E))\right)-\log\Prb(\cE^{(4)})\nonumber\\
&\leq \log\Prb\left(\bigcap_{i\in I^E: i\neq j} \overline G_n(x_i^E,\bar r_i ^E,v_i^E,\ep',f(x_i^E)+\eta)\cap G_n(x_j ^E,\bar r_j ^E,v_j^E,\ep,(1-\eta) f(x_j^E))\right)-\log\Prb(\cE^{(4)})\nonumber\\
&=  \log \Prb(G_n(x_j ^E,\bar r_j ^E,v_j^E,\ep,(1-\eta) f(x_j^E)))-\log\Prb(\overline G_n(x_j ^E,\bar r_j ^E,v_j^E,\ep', f(x_j^E)+\eta))\,.
\end{align}
Let $\kappa_0>0$ be given by lemma \ref{lem: Gxrv}.
Using lemma \ref{lem: Gxrv}, we have
\begin{align}\label{eq:432}
\limsup_{n\rightarrow\infty }\frac{1}{n^{d-1}}\log\Prb(G_n(x_j ^E,\bar r_j ^E,v_j^E,\ep,(1-\eta) f(x_j^E)))\leq  -g(\ep)\alpha_{d-1}(\bar r_j ^E)^{d-1}\cJ_{v_j^E}\left(\frac{ (1-\eta)f(x_j^E)+\kappa_0\sqrt{\ep}}{g(\ep)}\right)\,.
\end{align}
Let us choose $\ep'$ (and so $\ep$) small enough depending as $\delta$ in \eqref{cond:deltalemphi} (see lemma \ref{lem:phicassimple}). Using lemma \ref{lem:phicassimple}, we have
\begin{align}\label{eq:433}
\liminf_{n\rightarrow\infty }\frac{1}{n^{d-1}}\log\Prb(\overline G_n(x_j ^E,\bar r_j ^E,v_j^E,\ep',f(x_j^E)+\eta))\geq -\alpha_{d-1}(\bar r_j ^E) ^{d-1}\cJ_{v_j^E}(f(x_j^E))\,.
\end{align}
Pick $\ep$ small enough such that 
\begin{align}\label{eq:choixeplb}
\frac{1}{g(\ep)}\leq 1+\eta\qquad\text{and}\qquad 2\kappa_0\sqrt{\ep}\leq \frac{\eta^3}{2}\,,
\end{align}
(remember that $g(\ep)$ goes to $0$ when $\ep$ goes to $0$).
We recall that $f(x_j^E)\geq \eta$.
Using that the map $\cJ_{v_j^E}$ is non increasing, we get
\begin{align}\label{eq:434}
\cJ_{v_j^E}\left(\frac{ (1-\eta)f(x_j^E)+\kappa_0\sqrt{\ep}}{g(\ep)}\right)\geq\cJ_{v_j^E}\left(f(x_j^E)+2\kappa_0\sqrt{\ep}-\eta^2f(x_j^E)\right)\geq \cJ_{v_j^E}\left(f(x_j^E)-\frac{\eta^3}{2}\right)\,.
\end{align}
Using that $\cJ_{v_j^E}(\nu_G(v_j^E))=0$ and the convexity of the map $\cJ_{v_j^E}$, we obtain
\begin{align}\label{eq:appconv}
\cJ_{v_j^E}(f(x_j^E))\leq \lambda \cJ_{v_j^E}\left(f(x_j^E)-\frac{\eta^3}{2}\right)
\end{align}
where $\lambda\in]0,1[$ is chosen such that
$$f(x_j^E)=\lambda\left(f(x_j^E)-\frac{\eta^3}{2}\right)+(1-\lambda)\nu_G(v_j^E)\,.$$
We recall that $f(x_j^E)\leq \nu_G(v_j^E)$ since $(E,f)\in \fT$. If $f(x_j^E)=\nu_G(v_j^E)$ then inequality \eqref{eq:appconv} is trivial. Otherwise, such a $\lambda\in ]0,1[$ exists.
Choose $\ep$ small enough depending on $\eta$ and $\nu_G$ such that
$$\frac{ \sup_{v\in\sS^{d-1}}\nu_G(v)}{ \sup_{v\in\sS^{d-1}}\nu_G(v)+\eta^3/2}<g(\ep)\,.$$
Using that the map $x\mapsto x/(x+\eta^3/2)$ is non-decreasing on $\sR^+$, it follows that
\begin{align}\label{eq:435}
\lambda = \frac{ \nu_G(v_j^E)-f(x_j^E)}{ \nu_G(v_j^E)-f(x_j^E)+\eta^3/2}\leq \frac{ \sup_{v\in\sS^{d-1}}\nu_G(v)}{ \sup_{v\in\sS^{d-1}}\nu_G(v)+\eta^3/2}<g(\ep)\,.
\end{align}
Hence, we have
$\lambda<g(\ep)$.
Combining inequalities \eqref{eq:431}, \eqref{eq:432}, \eqref{eq:433}, \eqref{eq:434}, \eqref{eq:appconv} and \eqref{eq:435}, we obtain that
\begin{align}\label{eq:probub2}
\limsup_{n\rightarrow\infty }&\frac{1}{n^{d-1}}\left(\log \Prb\left(\bigcap_{i\in I^E} \overline G_n(x_i^E,\bar r_i ^E,v_i^E,\ep',f(x_i^E)+\eta)\cap G_n(x_j ^E,\bar r_j ^E,v_j^E,\ep,(1-\eta) f(x_j^E))\right)-\log\Prb(\cE^{(4)})\right)\nonumber\\
&\leq -g(\ep)\alpha_{d-1}(\bar r_j ^E)^{d-1}\cJ_{v_j^E}\left(f(x_j^E)-\frac{\eta^3}{2}\right)+ \lambda\alpha_{d-1}(\bar r_j ^E)^{d-1}\cJ_{v_j^E}\left(f(x_j^E)-\frac{\eta^3}{2}\right)<0\,.
\end{align}
Let us estimate $\Prb ((\cE^{(2)})^c\cap\cE^{(4)})$.
Let $F\in \cA$ and $j\in I ^F$ such that $ x_j^F\in \fF \cap  \fE$ and $f(x_j^F)\geq\eta$. We have
\begin{align*}
\Prb &\left(\bigcap_{i\in I^E} \overline G_n(x_i^E,\bar r_i ^E,v_i^E,\ep',f(x_i^E)+\eta) \cap  G_n(x_j ^F,r_j^F,v_j^F,\ep,(1-\eta) f(x_j^F))\right)\\
&\leq \Prb \left(\bigcap_{i\in I^E: B(x_i^E,\bar r_i ^E)\cap B(x_j^F,r_j^F)=\emptyset } \overline G_n(x_i^E,\bar r_i ^E,v_i^E,\ep',f(x_i^E)+\eta) \cap  G_n(x_j ^F,r_j^F,v_j^F,\ep,(1-\eta) f(x_j^F)) \right)\,.
\end{align*}
Using the independence and the previous inequality, it follows that 
\begin{align}\label{eq:eqprbe1}
\log\Prb &\left(\bigcap_{i\in I^E} \overline G_n(x_i^E,\bar r_i ^E,v_i^E,\ep',f(x_i^E)+\eta) \cap  G_n(x_j ^F,r_j^F,v_j^F,\ep,(1-\eta) f(x_j^F))\right)-\log\Prb(\cE^{(4)})\nonumber\\
&\leq \log\Prb(G_n(x_j ^F,r_j^F,v_j^F,\ep,(1-\eta) f(x_j^F)))-\log \Prb\left(\bigcap_{\substack{i\in I^E: \\B(x_i^E,\bar r_i ^E)\cap B(x_j^F,r_j^F)\neq\emptyset }} \overline G_n(x_i^E,\bar r_i ^E,v_i^E,\ep',f(x_i^E)+\eta)\right)
\end{align}
and thanks to lemma \ref{lem: Gxrv} and the choice of $\ep$ as in \eqref{eq:choixeplb}, we have
\begin{align}\label{eq:eqprbe2}
\limsup_{n\rightarrow\infty }\frac{1}{n^{d-1}}\log\Prb(G_n(x_j ^F,r_j^F,v_j^F,\ep,(1-\eta) f(x_j^F)))\leq  -g(\ep)\alpha_{d-1}(r_j^F)^{d-1}\cJ_{v_j^F}\left(f(x_j^F)- \frac{\eta^3}{2}\right)\,.
\end{align}
Besides, we have using the independence, lemma \ref{lem:phicassimple} and inequalities \eqref{eq:lbf2}, \eqref{eq:lbf10}, \eqref{eq:lbe2} and \eqref{eq:lbe8}
\begin{align}\label{eq:calcprbe3}
\liminf_{n\rightarrow\infty }&\frac{1}{n^{d-1}}\log \Prb\left(\bigcap_{i\in I^E: B(x_i^E,\bar r_i ^E)\cap B(x_j^F,r_j^F)\neq\emptyset } \overline G_n(x_i^E,\bar r_i ^E,v_i^E,\ep',f(x_i^E)+\eta)\right)\nonumber\\
&\geq -\sum_{i\in I^E: B(x_i^E,\bar r_i ^E)\cap B(x_j^F,r_j^F)\neq\emptyset }\alpha_{d-1}(\bar r_i ^E) ^{d-1}\cJ_{v_i^E}(f(x_i^E))\nonumber\\
&\geq -\sum_{i\in I^E: B(x_i^E,\bar r_i ^E)\cap B(x_j^F,r_j^F)\neq\emptyset }\left(\int_{B(x_i^E,\bar r_i ^E)\cap\fE}\cJ_{n_\bullet(x)}(f(x))d\cH^{d-1}(x) +\ep \alpha_{d-1}(\bar r_i^E)^{d-1}\right)\nonumber\\
&\geq -\int_{B(x_j^F,r_j^F)\cap \fE}\cJ_{n_\bullet(x)}(f(x))d\cH^{d-1}(x)- \int_{\fE\cap B(x_j^F,r_j^F+\delta_1) \setminus B(x_j^F,r_j^F)} \cJ_{n_\bullet(x)}(f(x))d\cH^{d-1}(x)\nonumber\\
&\quad -2\ep\left(\cH^{d-1}(\fE\cap B(x_i^F,r_i^F))+ \cH^{d-1}(\fE\cap( B(x_i^F,r_i^F+\delta_1)\setminus B(x_i^F,r_i^F))\right)\nonumber\\
&\geq -\alpha_{d-1}(r_j^F) ^{d-1}\cJ_{v_j^F}(f(x_j^F))- 2\ep \alpha_{d-1}(r_j^F) ^{d-1} -2\ep\left(\alpha_{d-1}(r_j^F) ^{d-1}+ 2\ep \alpha_{d-1}(r_j^F) ^{d-1}\right)\nonumber\\
&\geq -\alpha_{d-1}(r_j^F) ^{d-1}\cJ_{v_j^F}(f(x_j^F))- 7\ep \alpha_{d-1}(r_j^F) ^{d-1} 
\end{align}
where we use that $r_i^E\leq \delta_1$ and that $\delta_1$ was chosen such that it satisfies \eqref{eq:choixdelta1}.
We can choose $\ep$ small enough depending on $\eta$ such that
\begin{align}\label{eq:choixeplb2}
7\ep< \inf_{v\in\sS^{d-1}}\cJ_{v}\left(\nu_G(v)-\frac{\eta ^3}{2}\right)\left( g(\ep)-\frac{ \sup_{v\in\sS^{d-1}}\nu_G(v)}{ \sup_{v\in\sS^{d-1}}\nu_G(v)+\eta^3/2}\right) \,.
\end{align}
We recall that $g(\ep)$ goes to $1$ when $\ep$ goes to $0$. Hence, when $\ep$ goes to $0$, the limit of right hand side is positive.
Using the convexity as above, we obtain
\begin{align}\label{eq:utilisationconv2}
\cJ_{v_j^F}(f(x_j^F))\leq\lambda  \cJ_{v_j^F}\left(f(x_j^F)-\frac{\eta^3}{2}\right)
\end{align}
where \[\lambda\leq   \frac{ \sup_{v\in\sS^{d-1}}\nu_G(v)}{ \sup_{v\in\sS^{d-1}}\nu_G(v)+\eta^3/2}\,.\]
Combining inequalities \eqref{eq:eqprbe1}, \eqref{eq:eqprbe2}, \eqref{eq:calcprbe3} and \eqref{eq:utilisationconv2}, we get thanks to the choice of $\ep$ as in \eqref{eq:choixeplb2}
\begin{align}\label{eq:probub3}
\limsup_{n\rightarrow\infty }&\frac{1}{n^{d-1}}\log\left( \Prb\left(\bigcap_{i\in I^E} \overline G_n(x_i^E,\bar r_i ^E,v_i^E,\ep',f(x_i^E)+\eta)\cap G_n(x_j ^E,\bar r_j ^E,v_j^E,\ep,(1-\eta) f(x_j^E))\right)/\,\Prb(\cE^{(4)})\right)\nonumber\\
&\leq -g(\ep)\alpha_{d-1}(r_j^F)^{d-1}\cJ_{v_j^F}\left(f(x_j^F)-\frac{\eta^3}{2}\right)+ \lambda\alpha_{d-1}(r_j^F)^{d-1}\cJ_{v_j^F}\left(f(x_j^E)-\frac{\eta^3}{2}\right)+ 7\ep \alpha_{d-1}(r_j^F) ^{d-1}\nonumber\\
& <0\,.
\end{align}
Combining inequalities \eqref{eq:unionboundprob}, \eqref{eq:probub2}, \eqref{eq:probub3} with lemma \ref{lem:estimeanalyse} that
\begin{align}\label{eq:eqintub}
\limsup_{n\rightarrow\infty}\frac{1}{n^{d-1}}\log\Prb((\cE^{(2)}\cap\cE^{(3)})^c\cap\cE^{(4)}\cap \cE^{(1)})<\liminf_{n\rightarrow\infty}\frac{1}{n^{d-1}}\log\Prb(\cE^{(4)}) \,.\end{align}

Let us now estimate $\Prb(\cE^{(1)})$.
Using FKG inequality, we have
\begin{align*}
\Prb \left(\cE^{(1)}\right)\geq \prod_{F\in \cA}\prod_{\substack{i\in I ^F:\\ x_i^F\in \fF \setminus \fE}}\Prb(G_n(x_i ^F,r_i^F,v_i^F,\ep, (1-\eta)\nu_G(v_i^F))^c)\,.
\end{align*}
Let $F\in\cA$ and $i\in I^F$.
Thanks to the choice of $\ep$ in \eqref{eq:choixeplb} and lemma \ref{lem: Gxrv}, we have
$$\limsup_{n\rightarrow\infty}\frac{1}{n^{d-1}}\log\Prb(G_n(x_i ^F,r_i^F,v_i^F,\ep, (1-\eta)\nu_G(v_i^F)))<-g(\ep)\alpha_{d-1}(r_i^F)^{d-1}\cJ_{v_i^F}\left(\nu_G(v_i^F)-\frac{\eta^3}{2}\right)<0\,.$$
It follows by lemma \ref{lem:estimeanalyse} that 
$$\liminf_{n\rightarrow\infty}\frac{1}{n^{d-1}}\log\Prb(G_n(x_i ^F,r_i^F,v_i^F,\ep, (1-\eta)\nu_G(v_i^F))^c)=0$$
and 
\[\liminf_{n\rightarrow \infty}\frac{1}{n^{d-1}}\log 
\Prb \left(\cE^{(1)}\right)\geq \sum_{F\in \cA}\sum _{\substack{i\in I ^F:\\ x_i^F\in \fF \setminus \fE}}\liminf_{n\rightarrow \infty}\frac{1}{n^{d-1}}\log\Prb(G_n(x_i ^F,r_i^F,v_i^F,\ep, (1-\eta)\nu_G(v_i^F))^c) =0\,.\]
It follows that
\begin{align}\label{eq:prbE10}
\liminf_{n\rightarrow \infty}\frac{1}{n^{d-1}}\log 
\Prb \left(\cE^{(1)}\right)=0\,.
\end{align}
Using that $\cE^{(1)}$ and $\cE^{(4)}$ are independent equalities \eqref{eq:basicub} and \eqref{eq:prbE10}, inequality \eqref{eq:eqintub}
and lemma \ref{lem:estimeanalyse}
\begin{align}\label{eq:egalliminf}
\liminf_{n\rightarrow\infty}\frac{1}{n^{d-1}}\log\Prb\left(\cE^{(4)}\right)\leq \liminf_{n\rightarrow\infty}\frac{1}{n^{d-1}}\log\Prb\left(\cE^{(4)}\cap \cE^{(1)}\right)=\liminf_{n\rightarrow\infty}\frac{1}{n^{d-1}}\log\Prb(\cE^{(1)}\cap\cE^{(2)}\cap\cE^{(3)}\cap \cE^{(4)})\,.
\end{align}
Using \eqref{eq:egalliminf}, lemma \ref{lem:phicassimple}, inequalities \eqref{eq:lbe2} and \eqref{eq:lbe8}, we have
\begin{align*}
\liminf_{n\rightarrow\infty }\frac{1}{n^{d-1}}\log \Prb(\cE^{(1)}\cap\cE^{(2)}\cap\cE^{(3)}\cap \cE^{(4)})&\geq \liminf_{n\rightarrow\infty }\frac{1}{n^{d-1}}\log\Prb(\cE^{(4)})\\
&\geq  \liminf_{n\rightarrow\infty }\frac{1}{n^{d-1}}\log \Prb\left(\bigcap_{i\in I^E} \overline G_n(x_i^E,\bar r_i ^E,v_i^E,\ep',f(x_i^E)+\eta)\right)\\&\geq -\sum_{i\in I^E}\alpha_{d-1}\bar r_i ^E\cJ_{v_i^E}(f(x_i^E))\\
&\geq -\sum_{i\in I^E}\left( \int_{\fE\cap B(x_i^E,r_i^E)}\cJ_{n_\bullet(y)}(f(y))d\cH^{d-1}(y)+\ep\alpha_{d-1}( r_i ^E)^ {d-1}\right)\\
&\geq -\widetilde {I}(E,\nu)-2\ep\cH^{d-1}(\fE)\,.
\end{align*}
Besides, we have up to choosing a smaller $\ep$
$$\limsup_{n\rightarrow\infty }\frac{1}{n^{d-1}}\log \Prb((\cE^{(0)})^c)\leq -2\widetilde {I}(E,\nu)<-\widetilde {I}(E,\nu)-2\ep\cH^{d-1}(\fE)\,.$$
By lemma \ref{lem:estimeanalyse}, it follows that for small enough $\ep$
$$\liminf_{n\rightarrow\infty }\frac{1}{n^{d-1}}\log \Prb(\cE^{(0)}\cap \cE^{(1)}\cap\cE^{(2)}\cap\cE^{(3)}\cap \cE^{(4)})= \liminf_{n\rightarrow\infty }\frac{1}{n^{d-1}}\log \Prb(\cE^{(1)}\cap\cE^{(2)}\cap\cE^{(3)}\cap \cE^{(4)})\,.$$
For any $\eta>0$ small enough and $\ep$ small enough depending on $\eta$, we have
\begin{align*}
\liminf_{n\rightarrow\infty}\frac{1}{n^{d-1}}&\log\Prb\left(\exists\cE_n\in\sC_n(u(\eta)):(\bR(\cE_n),\mu_n(\cE_n))\in \cU\right)\\&\geq \liminf_{n\rightarrow\infty }\frac{1}{n^{d-1}}\log \Prb(\cE^{(0)}\cap \cE^{(1)}\cap\cE^{(2)}\cap\cE^{(3)}\cap \cE^{(4)})\geq -\widetilde{I}(E,\nu)-2\ep\cH^{d-1}(\fE)\,.
\end{align*}
By letting first $\ep$ go to $0$ and then $\eta$ go to $0$ we obtain the expected result.

\noindent{\bf Calibration of constants.} We explain here in which order the constants are chosen.
We first choose $\eta$ such that it satisfies
$$\eta\leq\frac{\xi_1}{C_0}$$ where the constant $C_0$ was defined in \eqref{eq:defC0}.
Then we choose $\ep$ small enough depending on $\eta$, $d$, $\Omega$ and $\xi_0$ such that it satisfies \eqref{eq:choixep0xi}, \eqref{eq:choixeplb} and \eqref{eq:choixeplb2}. Once $\ep$ is fixed, we can define $\delta_0$ as in \eqref{eq:unifcontg}, $\delta_F$ for each $F\in\sC_\beta$ as in \eqref{eq:choixdeltaf}, the covering of $\sC_\beta$. Once the covering of $\sC_\beta$ is chosen, we can define $\delta_1$ as in \eqref{eq:choixdelta1}. Finally, we can build a covering of $\fE$ depending on $\ep$ and on the covering of $\sC_\beta$.

\end{proof}


 \section{Upper bound} \label{sec:ub}
 The aim of this section is to prove the following result.
 
 \begin{prop}\label{prop:upperbound} Let $(E,f)\in\fT$. Write $\nu=f\cH^{d-1}|_{\fE}$.
 If $\cI(E,f)<\infty$, then for any $\delta_0\in]0,1[$, there exists a neighborhood $U$ of $(E,\nu)$ such that
 $$\limsup_{n\rightarrow\infty}\frac{1}{n^{d-1}}\log\Prb\left(\exists \cE_n\in \sC_n(\Gamma^1,\Gamma^2,\Omega): (\bR(\cE_n),\mu_n(\cE_n))\in U\right)\leq -(1-\delta_0)\cI(E,f)\,.$$
 If $\cI(E,f)=+\infty$, then for any $t>0$ there exists a neighborhood $U$ of $(E,\nu)$ such that
 $$\limsup_{n\rightarrow\infty}\frac{1}{n^{d-1}}\log\Prb\left(\exists \cE_n\in \sC_n(\Gamma^1,\Gamma^2,\Omega): (\bR(\cE_n),\mu_n(\cE_n))\in U\right)\leq -t\,.$$
 \end{prop}
 Before proving this result, let us introduce some properties of the rate function $\cJ_v$ that will be useful in the proof of proposition \ref{prop:upperbound}.
 \subsection{Properties of the rate function $\cJ_\vv$}
 We already know that for a fixed $v\in\sS^{d-1}$, the map $\lambda\mapsto \cJ_v(\lambda)$ is lower semi-continuous. We aim to prove in this section that the map $(\lambda,v)\mapsto \cJ_v(\lambda)$ is also lower semi-continuous.
 The rate function $\cJ$ satisfies the weak triangle inequality in the following sense.
 \begin{prop}[weak triangle inequality for $\cJ$]\label{prop:triangineq}
Let $(ABC)$ be a non-degenerate triangle in $\sR^d$ and let $\vv_A,\vv_B,\vv_C$ be the exterior normal unit vectors to the sides $[BC],[AC],[AB]$ in the plane spanned by $A,B,C$. Then, for any $\lambda,\mu \geq 0$
\[\cH^1([BC]) \cJ_{\vv_A}\left( \frac{\lambda\cH^1([AC])+\mu\cH^1([AB])}{\cH^1([BC])}\right)\leq \cH^1([AC])\cJ_{\vv_B}(\lambda)+\cH^1([AB])\cJ_{\vv_C}( \mu)\,.\]
\end{prop}
\begin{proof}The proof is an adaptation of the proof of Proposition 11.6 in \cite{Cerf:StFlour}. This result was proved for the dimension $2$ in lemma 3.1. in \cite{RossignolTheretd2lower}. We only treat the case where the triangle $(ABC)$ is such that $\overrightarrow{ BA }\cdot \overrightarrow{BC}\geq 0$ and $\overrightarrow{CA}\cdot \overrightarrow{CB}\geq 0$ (actually we will only need this case in what follows). Once this case is proved, the other cases may be obtained as a straightforward adaptation of the end of the proof of Proposition 11.6. Let $(e_1,\dots,e_d)$ be an orthonormal basis such that $e_1$ and $e_2$ belong to the space spanned by $A,B,C$. Let $\ep,h$ be such that $0<\ep\leq 1\leq h$ and $\lambda,\mu\geq 0$. Let $K$ be the compact convex set defined by 
$$K=\left\{\,x+\sum_{i=3}^du_ie_i:\,x\in(ABC), \,(u_3,\dots,u_d)\in[0,h]^{d-2}\,\right\}\,.$$
The boundary of $K$ consists of the three following hyperrectangles 
$$R_A=\left\{\,x+\sum_{i=3}^du_ie_i:\,x\in[BC], \,(u_3,\dots,u_d)\in[0,h]^{d-2}\,\right\},$$
$$R_B=\left\{\,x+\sum_{i=3}^du_ie_i:\,x\in[AC], \,(u_3,\dots,u_d)\in[0,h]^{d-2}\,\right\},$$
$$R_C=\left\{\,x+\sum_{i=3}^du_ie_i:\,x\in[AB], \,(u_3,\dots,u_d)\in[0,h]^{d-2}\,\right\}$$
and the set 
$$T=\bigcup_{3\leq j\leq d}\left\{x+\sum_{i=3}^du_ie_i:\,x\in(ABC), \,u_j\in\{0,h\},\,(u_3,\dots ,u_{j-1},u_{j+1},\dots,u_d)\in[0,h]^{d-3}\,\right\}\,.$$
We can define $R_B^\ep$ and $R_C^\ep$ such that (see figure \ref{fig:r}):
\begin{itemize}[$\bullet$]
\item $R_B^\ep\subset R_B$, $R_C^\ep\subset R_C$ and $\cyl(R_B^\ep,\ep)\cap \cyl(R_C^\ep,\ep)=\emptyset$,
\item for $h'\geq 0$ large enough, $\cyl(R_B^\ep,\ep)\cup \cyl(R_C^\ep,\ep)\subset \cyl(R_A,h')$,
\item we have for any $h>0$, $\lim_{\ep\rightarrow 0}\cH^{d-1}(R_B\setminus R_B^\ep) =0$ and $\lim_{\ep\rightarrow 0}\cH^{d-1}(R_C\setminus R_C^\ep) =0$.
\end{itemize}
\begin{figure}[H]
\begin{center}
\def\svgwidth{0.5\textwidth}
   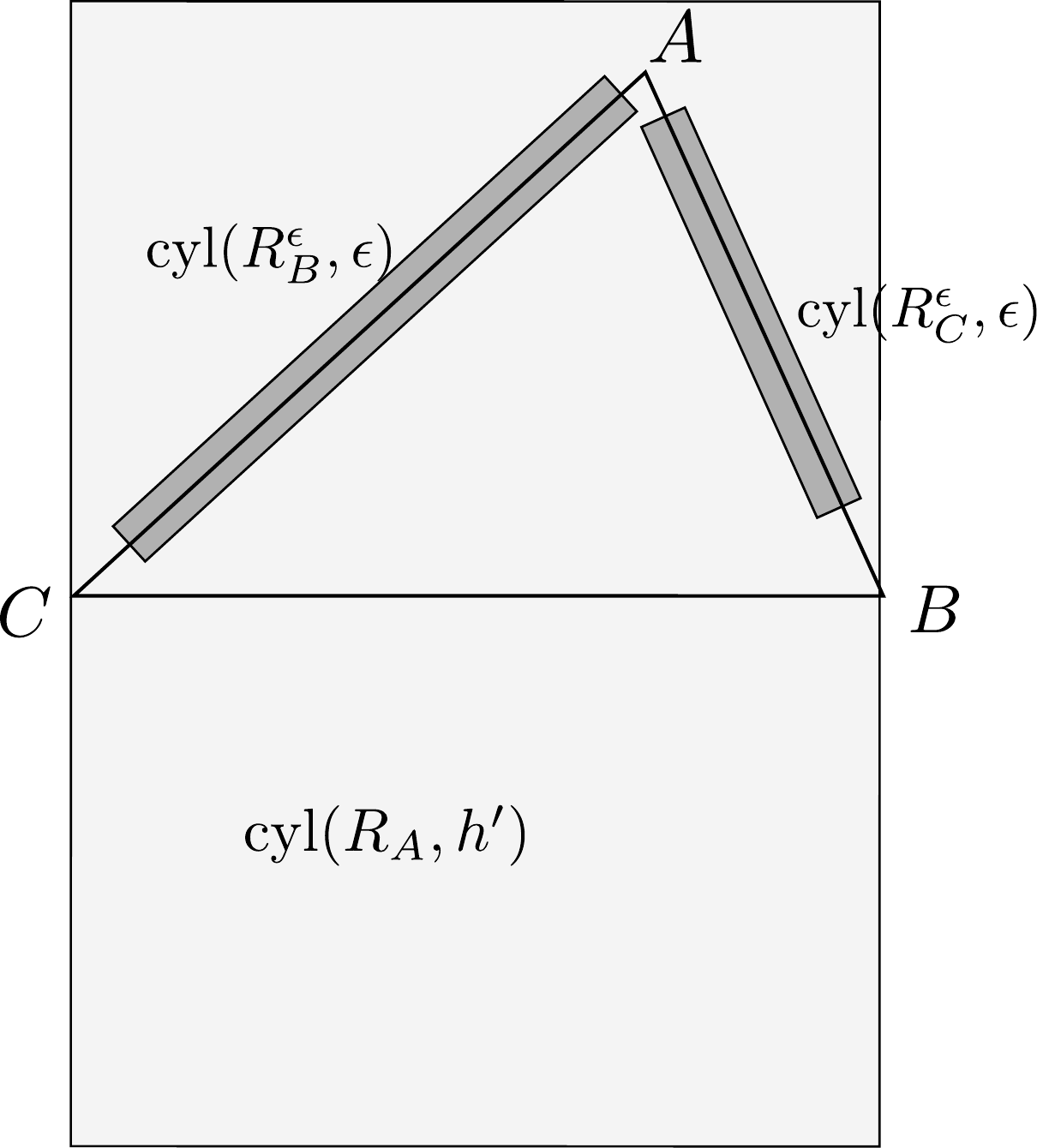
   \caption{\label{fig:r}Representation of $R_B^\ep$ and $R_C^\ep$.}
   \end{center}
\end{figure}
We denote by $\mathrm{E}_n^\ep$ the edges in $\E_n^d$ that have at least one endpoint in 
$$\cV_2( R_B\setminus R_B^\ep,d/n)\cup \cV_2( R_C\setminus R_C^\ep,d/n)\cup\cV_2(T,d/n)\,.$$
Using proposition \ref{prop:minkowski}, there exists a positive constant $c_d$ depending only on $d$ such that
\begin{align*}
\card(\mathrm{E}_n^\ep)&\leq c_d\left(\cH^{d-1}(R_B\setminus R_B^\ep)+ \cH^{d-1}(R_C\setminus R_C^\ep) + \cH^{d-1}(T)\right)n^{d-1}\\&\leq c_d\left(\cH^{d-1}(R_B\setminus R_B^\ep)+ \cH^{d-1}(R_C\setminus R_C^\ep) +2(d-2) \cH^{2}(ABC)h ^{d-3}\right)n^{d-1}\,.
\end{align*}
Set
$$\cE=\{\tau_n(R_B^\ep ,\ep)\leq \lambda\cH^{d-1}(R_B^\ep )n^{d-1}\}\cap \{\tau_n(R_C^\ep ,\ep)\leq \mu\cH^{d-1}(R_C^\ep )n^{d-1}\}\,.$$
Note that if $\mathrm E_B^\ep$ (respectively $\mathrm E_C^\ep$) is  a cutset for $\tau_n(R_B,\ep)$ (resp. $\tau_n(R_C,\ep)$), then $\mathrm E_B^\ep\cup \mathrm E_C^\ep \cup \mathrm E_n^\ep$ is a cutset for $\tau_n(R_A,h')$.
Hence, we have
$$\cE\subset \{\tau_n(R_A, h')\leq (\lambda\cH^{d-1}(R_B)+ \mu\cH^{d-1}(R_C))n^{d-1}+ \card(\mathrm{E}_n^\ep)M\}\,.$$
Hence, we have using the independence
\begin{align*}
\Prb&\left(\frac{\tau_n(R_A,h)}{n^{d-1}}\leq \lambda\cH^{d-1}(R_B)+ \mu\cH^{d-1}(R_C)+ c_d\left(f(\ep) +2(d-2) \cH^{2}(ABC)h ^{d-3}\right)M\right)\\
&\qquad \geq\Prb\left(\tau_n(R_A,h)\leq (\lambda\cH^{d-1}(R_B)+ \mu\cH^{d-1}(R_C))n^{d-1}+ \card (\mathrm{E}_n^\ep)M\right)\\
&\qquad \geq \Prb(\tau_n(R_B^\ep,\ep)\leq \lambda\cH^{d-1}(R_B^\ep)n^{d-1})\Prb(\tau_n(R_C^\ep,\ep)\leq \mu\cH^{d-1}(R_C^\ep)n^{d-1})\,
\end{align*}
where $f(\ep)=\cH^{d-1}(R_B\setminus R_B^\ep)+ \cH^{d-1}(R_C\setminus R_C^\ep)$.
Using theorem \ref{thm:lldtau}, we have
\begin{align*}
\cH^{d-1}(R_A)\cJ_{\vv_A}&\left(\frac{\lambda\cH^{d-1}(R_B)+ \mu\cH^{d-1}(R_C)+c_d\left(f(\ep) +2(d-2) \cH^{2}(ABC)h ^{d-3}\right)M)}{\cH^{d-1}(R_A)}\right)\\
&\leq  \cH^{d-1}(R_B^\ep)\cJ_{\vv_B}(\lambda)+ \cH^{d-1}(R_C^\ep)\cJ_{\vv_C}(\mu)\,.
\end{align*}
We recall that $\cH^{d-1}(R_A)=h^{d-2}\cH^1([BC])$, $\cH^{d-1}(R_B)=h^{d-2}\cH^1([AC])$ and $\cH^{d-1}(R_C)=h^{d-2}\cH^1([AB])$. We also recall that $\cJ_{\vv_A}$ is lower semi-continuous.
By letting $\ep$ goes to $0$ we obtain
\begin{align*}
&\cH^{d-1}(R_A)\cJ_{\vv_A}\left(\frac{\lambda\cH^{1}([AC])+ \mu\cH^{1}([AB])+2c_d(d-2) \cH^{2}(ABC)h ^{-1}M}{\cH^{1}([BC])}\right)\\
&\leq \liminf_{\ep\rightarrow 0} \cH^{d-1}(R_A)\cJ_{\vv_A}\left(\frac{\lambda\cH^{d-1}(R_B)+ \mu\cH^{d-1}(R_C)+c_d\left(f(\ep) +2(d-2) \cH^{2}(ABC)h ^{d-3}\right)M)}{\cH^{d-1}(R_A)}\right)\\
&\leq  \cH^{d-1}(R_B)\cJ_{\vv_B}(\lambda)+ \cH^{d-1}(R_C)\cJ_{\vv_C}(\mu)\,.
\end{align*}
By dividing the inequality by $h^{d-2}$, we obtain
\begin{align*}
\cH^{1}([BC])\cJ_{\vv_A}&\left(\frac{\lambda\cH^{1}([AC])+ \mu\cH^{1}([AB])+2c_d(d-2) \cH^{2}(ABC)h ^{-1})M)}{\cH^{1}([BC])}\right)\\
&\leq  \cH^{1}([AC])\cJ_{\vv_B}(\lambda)+ \cH^{1}([AB])\cJ_{\vv_C}(\mu)\,.
\end{align*}
Letting $h$ go to infinity, using again the fact that $\cJ_{\vv_A}$ is lower semi-continuous, yields the result:
\begin{align*}
\cH^{1}([BC])\cJ_{\vv_A}&\left(\frac{\lambda\cH^{1}([AC])+ \mu\cH^{1}([AB])}{\cH^{1}([BC])}\right)\leq  \cH^{1}([AC])\cJ_{\vv_B}(\lambda)+ \cH^{1}([AB])\cJ_{\vv_C}(\mu)\,.
\end{align*}
\end{proof}

We can deduce from proposition \ref{prop:triangineq} the lower semi-continuity of $\cJ$.
\begin{cor}[The function $\cJ$ is lower semi-continuous]\label{cor:liminf}For any $\lambda\geq 0$, $\vv\in\sS^{d-1}$, for any sequence $(\lambda_n)_{n\geq 1}$  of positive real numbers that converges towards $\lambda$, for any sequence $(\vv_n)_{n\geq1}$ of $\sS^{d-1}$ that converges towards $\vv$, we have
$$\liminf_{n\rightarrow\infty}\cJ_{\vv_n}(\lambda_n)\geq \cJ_{\vv}(\lambda)\,.$$ 
\end{cor}
\begin{proof}Let $\lambda\geq 0$. Let $\vv\in\sS^{d-1}$. Let $(\lambda_n)_{n\geq 1}$ be a sequence of positive real numbers that converges towards $\lambda$. Let $(\vv_n)_{n\geq1}$ be a sequence of $\sS^{d-1}$ that converges towards $\vv$. Let $\cP_n$ be the plan spanned by $(O,\vv,\vv_n)$. Let $A$ and $A_n$ be points of $\sR^d$ in $\cP_n$ such that $\|\overrightarrow{ OA}\|_2=\|\overrightarrow{ OA_n}\|_2=1$,
$\overrightarrow{OA} \cdot \vv =\overrightarrow{OA_n} \cdot \vv_n=0$.
Note that $$\|\overrightarrow{ AA_n}\|_2=\|\overrightarrow{OA_n}-\overrightarrow{OA}\|_2=\|\vv_n-\vv\|_2$$ where we use that there is an isometry (rotation of $\pi/2$ centered at O in the plane $\cP_n$) that sends $\overrightarrow{OA}$ on $\vv$ and $\overrightarrow{OA_n}$ on $\vv_n$ and such that $\overrightarrow{OA_n}\cdot \overrightarrow{OA}\geq 0$.
 It follows that $$\lim_{n\rightarrow 0}\cH^1([AA_n])=0\,.$$
Let $\vv'_n$ be the exterior normal unit vector to the side $[AA_n]$ of the triangle $OAA_n$ in $\cP_n$ (see figure \ref{fig:triangle}).
\begin{figure}[H]
\begin{center}
\def\svgwidth{0.3\textwidth}
   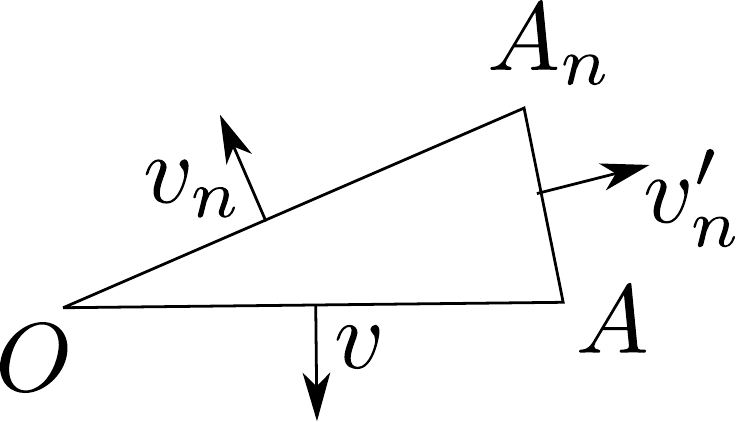
   \caption{\label{fig:triangle}The triangle $OAA_n$ in the plane $\cP_n$}
   \end{center}
\end{figure}
 
Applying proposition \ref{prop:triangineq}, we have
\[\cH^1([OA])\cJ_{\vv}\left(\frac{\cH^1([OA_n])\lambda_n +\cH^1([AA_n])\nu_G(v'_n)}{\cH^1([OA])}\right)\leq \cH^1([AA_n])\cJ_{\vv'_n}(\nu_G(v'_n))+\cH^1([OA_n])\cJ_{\vv_n}(\lambda_n)\]
and it yields that
$$\cJ_{\vv}(\lambda_n +\cH^1([AA_n])\nu_G(v'_n))\leq \cH^1([AA_n])\cJ_{\vv'_n}(\nu_G(v'_n))+\cJ_{\vv_n}(\lambda_n)=\cJ_{\vv_n}(\lambda_n)\,.$$
Besides, we have
$$\lambda_n + \cH^1([AA_n])\inf_{\vv_0\in\sS^{d-1}}\nu_G(v_0)\leq \lambda_n +\cH^1([AA_n])\nu_G(v'_n)\leq \lambda_n + \cH^1([AA_n])\sup_{\vv_0\in\sS^{d-1}}\nu_G(v_0)\,.$$
As a result, we have
$$\lim_{n\rightarrow \infty}\lambda_n +\cH^1([AA_n])\nu_G(v'_n)=\lambda\,.$$
Since the function $\alpha\mapsto \cJ_{\vv}(\alpha)$ is lower semi-continuous, 
we have 
$$\cJ_{\vv}(\lambda)\leq\liminf_{n\rightarrow\infty}\cJ_{\vv}(\lambda_n +\cH^1([AA_n])\nu_G(v'_n))\leq \liminf_{n\rightarrow\infty}\cJ_{\vv_n}(\lambda_n)\,.$$
The result follows.
\end{proof}

 \subsection{Proof of proposition \ref{prop:upperbound}}
 This section corresponds to the step 3 of the sketch of the proof in the introduction. For $(E,f)\in \fT$, we cover $\fE$ by small disjoint balls. Then we prove that we can build a neighborhood $U$ of $(E,\nu)$ adapted to this covering such that we can upperbound the probability $\Prb(\exists \cE_n\in\sC_n(\Gamma^1,\Gamma^2,\Omega): (\bR(\cE_n),\mu_n(\cE_n))\in U)$ using estimates on lower large deviations for the maximal flow in a ball.  
 \begin{proof}[Proof of proposition \ref{prop:upperbound}]
Let $\ep>0$. By proposition \ref{prop:utilisationvitali}, there exists a finite family of disjoint closed balls $(B(x_i,r_i,v_i))_{i\in I_1\cup I_2\cup I_3}$ such that for $i\in I_1$, we have $x_i\in \partial ^* E\cap \Omega$, for $i\in I_2$, we have $x_i\in \partial ^* \Omega \cap (\Gamma^1\setminus \overline{\partial^* E})$ and for $i\in I_3$, we have $x_i\in \partial ^* E\cap \Gamma^2$ and the following properties hold:
\begin{align}\label{eq:ubcond1}
\cH^{d-1}(\fE\setminus \cup_{i\in I_1\cup I_2\cup I_3}B(x_i,r_i)))\leq \ep\,,
\end{align}
\begin{align}\label{eq:ubcond2}
\forall i \in I_1\cup I_2\cup I_3\quad  \forall\, 0<r\leq 2 r_i\qquad \left|\frac{1}{\alpha_{d-1} r^{d-1}}\cH^{d-1}(\fE\cap B(x_i,r))-1\right|\leq \ep\,,
\end{align}
\begin{align}\label{eq:ubcond3}
\forall i\in I_1\qquad \left\|\frac{1}{\cH^{d-1}(\partial^* E\cap B(x_i,r_i))}\int_{\partial^* E\cap B(x_i,r_i)}n_E(x)d\cH^{d-1}(x) -n_{E}(x_i)\right\|_2\leq \ep\,,
\end{align}
\begin{align}\label{eq:ubcond3bis}
\forall i\in I_2\cup I_3\qquad \left\|\frac{1}{\cH^{d-1}(\partial^* \Omega\cap B(x_i,r_i))}\int_{\partial^* \Omega\cap B(x_i,r_i)}n_\Omega(x)d\cH^{d-1}(x) -n_{\Omega}(x_i)\right \|_2\leq \ep\,,
\end{align}
\begin{align}\label{eq:ubcond4}
\forall i\in I_1\cup I_2\cup I_3\qquad \left|\frac{1}{\cH^{d-1}(\fE\cap B(x_i,r_i))}\nu(B(x_i,r_i))-f(x_i)\right|\leq \ep\,,
\end{align}
\begin{align}\label{eq:ubcond5}
\forall i \in I_1\qquad B(x_i,r_i)\subset \Omega\quad\text{ and } \quad\cL^d((E\cap B(x_i,r_i)) \Delta B^-(x_i,r_i,v_i))\leq \ep \alpha_d r_i^d\,,
\end{align}
\begin{align}\label{eq:ubcond6}
\forall i \in I_2\qquad d_2(B(x_i,r_i), \Gamma^2 \cup E)>0 \quad\text{ and } \quad\  \cL^d((\Omega\cap B(x_i,r_i)) \Delta B^-(x_i,r_i,v_i))\leq \ep \alpha_d r_i^d\,,
\end{align}
\begin{align}\label{eq:ubcond7}
\forall i \in I_3\qquad d_2(B(x_i,r_i), \Gamma^1)>0 \quad\text{ and } \quad\  \cL^d((E\cap B(x_i,r_i)) \Delta B^-(x_i,r_i,v_i))\leq \ep \alpha_d r_i^d\,.
\end{align}
Let $i\in I_1\cup I_2\cup I_3$. Using inequality \eqref{eq:ubcond2}, we have
\begin{align}\label{eq:deltai}
\nu(B(x_i,r_i(1+\ep))\setminus B(x_i,r_i))&=\int_{(B(x_i,r_i(1+\ep))\setminus B(x_i,r_i))\cap\fE}f(y)d\cH^{d-1}(y)\nonumber\\ &\leq M\left(\cH^{d-1}(\fE\cap B(x_i,r_i(1+\ep)))-\cH^{d-1}(\fE\cap B(x_i,r_i))\right)\nonumber\\
&\leq M\left((1+\ep)^{d}- 
(1-\ep)\right)\alpha_{d-1}r_i^{d-1}\nonumber\\
& \leq M (1+ (2^d-1)\ep -1+\ep)\alpha_{d-1}r_i^{d-1}\leq M 2^d\ep\alpha_{d-1}r_i^{d-1}\,.
\end{align}
Let $f_i$ be a continuous function taking its values in $[0,1]$ with support included in $B(x_i,r_i(1+\ep))$ and such that $f_i=1$ on $B(x_i,r_i)$.
Using inequalities \eqref{eq:ubcond2} and \eqref{eq:ubcond4}, we have
\begin{align*}
\nu(f_i)\leq \nu(B(x_i,r_i))+\nu(B(x_i,r_i(1+\ep))\setminus B(x_i,r_i))&\leq  (f(x_i)+\ep)\cH^{d-1}(\fE\cap B(x_i,r_i))+ 2^dM\ep\alpha_{d-1}r_i ^{d-1}
\\&\leq ((f(x_i)+\ep)(1+\ep)+ 2^dM\ep)\alpha_{d-1}r_i ^{d-1}\\
&\leq (f(x_i)+\ep(2+2^dM +M))\alpha_{d-1}r_i ^{d-1}\,.
\end{align*}
We set \[\ep_E= \ep\alpha_ d\min_{i\in I_1\cup I_2\cup I_3}r_i^d\,.\]
Let $U$ be the weak neighborhood of $(E,\nu)$ in $\cB(\sR^d)\times\cM(\sR^d)$ defined by
\begin{align}\label{eq:defU}
U=\{F\in\cB(\sR^d): \cL^d(E\Delta F)\leq\ep_E\}\times\left\{\,\rho\in\cM(\sR^d):\forall i \in I_1\cup I_2\cup I_3\quad |\rho(f_i)-\nu(f_i)|\leq \ep\alpha_{d-1}r_i ^{d-1}\,\right\}\,.
\end{align}
On the event $\{\exists \cE_n\in \sC_n(\Gamma^1,\Gamma^2,\Omega): (\bR(\cE_n),\mu_n(\cE_n))\in U\}$, we pick $\cE_n\in \sC_n(\Gamma^1,\Gamma^2,\Omega)$ such that $(\bR(\cE_n),\mu_n(\cE_n))\in U\,.$
If there are several possible choices, we choose according to a deterministic rule.
We write $\mu_n=\mu_n(\cE_n)$ and $E_n=\bR(\cE_n)$.
 Let $i\in I_1\cup I_2\cup I_3$. We have 
$$\mu_n(B(x_i,r_i))\leq\mu_n(f_i)\leq \nu(f_i)+\ep\alpha_{d-1}r_i ^{d-1}\leq (f(x_i)+(3+ 2^d M+M)\ep)\alpha_{d-1}r_i ^{d-1}\,.$$ 
Moreover, using that $\cL^d(E_n\Delta E)\leq \ep \alpha_ dr_i ^d$, we can prove as in section 5.2. in \cite{CT3} that the following event occurs
$$\bigcap_{i\in I_1\cup I_3} G_n(x_i,r_i,n_{E}(x_i),\ep, f(x_i)+\kappa\ep)\bigcap _{i\in I_2}G_n(x_i,r_i,n_{\Omega}(x_i),\ep, f(x_i)+\kappa\ep)$$ where we recall that the event $G_n$ was defined in section \ref{section:upperboundball} and where we set $\kappa=3+2^dM+M$.
Using that the balls are disjoint, we have by independence
\begin{align*}
\Prb&\left(\exists \cE_n\in \sC_n(\Gamma^1,\Gamma^2,\Omega): (\bR(\cE_n),\mu_n(\cE_n))\in U\right)\\
&\qquad\leq \prod_{i\in I_1\cup I_3}\Prb(G_n(x_i,r_i,n_{E}(x_i),\ep, f(x_i)+\kappa\ep)\prod_{i\in I_2}\Prb(G_n(x_i,r_i,n_{\Omega}(x_i),\ep, f(x_i)+\kappa\ep)\,.
\end{align*}
Let $\kappa_0>0$ be given by lemma \ref{lem: Gxrv}.
By lemma \ref{lem: Gxrv}, we have
\begin{align*}
\limsup_{n\rightarrow\infty}&\frac{1}{n^{d-1}}\log\Prb\left(\exists \cE_n\in \sC_n(\Gamma^1,\Gamma^2,\Omega): (\bR(\cE_n),\mu_n(\cE_n))\in U\right)\\
&\leq -g(\ep)\sum_{i\in I_1\cup I_3}\alpha_{d-1}r_i ^{d-1}\cJ_{n_E(x_i)}\left(\frac{f(x_i)+\kappa\ep+\kappa_0\sqrt{\ep}}{g(\ep)}\right)\\
&\qquad-g(\ep)\sum_{i\in  I_2}\alpha_{d-1}r_i ^{d-1}\cJ_{n_\Omega(x_i)}\left(\frac{f(x_i)+\kappa\ep+\kappa_0\sqrt{\ep}}{g(\ep)}\right)\\
&\leq -\frac{g(\ep)}{1+\ep}\sum_{i\in I_1\cup I_3}\cH^{d-1}(\partial^* E\cap B(x_i,r_i))\cJ_{n_E(x_i)}\left(\frac{f(x_i)+\kappa\ep+\kappa_0\sqrt{\ep}}{g(\ep)}\right)\\
&\qquad-\frac{g(\ep)}{1+\ep}\sum_{i\in  I_2}\cH^{d-1}(\partial^* \Omega\cap B(x_i,r_i))\cJ_{n_\Omega(x_i)}\left(\frac{f(x_i)+\kappa\ep+\kappa_0\sqrt{\ep}}{g(\ep)}\right)
\end{align*}
where we use inequality \eqref{eq:ubcond2} in the last inequality.
Set $\fE_\ep=\cup_{i\in I_1\cup I_2\cup I_3}B(x_i,r_i)\cap(\partial ^* E\cup\partial ^* \Omega)$.
For any $x\in \fE_\ep$, we denote by $c_\ep(x)$ the unique $x_i\in I_1\cup I_2\cup I_3$ such that $x\in B(x_i,r_i)$.
It follows that
\begin{align*}
\limsup_{n\rightarrow\infty}&\frac{1}{n^{d-1}}\log\Prb\left(\exists \cE_n\in \sC_n(\Gamma^1,\Gamma^2,\Omega): (\bR(\cE_n),\mu_n(\cE_n))\in U\right)\\
&\leq -\frac{g(\ep)}{1+\ep}\int_{\fE}\cJ_{n_\bullet(c_\ep(x))}\left(\frac{f(c_\ep(x))+\kappa\ep+\kappa_0\sqrt{\ep}}{g(\ep)}\right)\ind_{\fE_\ep}(x)d\cH^{d-1}(x)
\end{align*}
where $n_\bullet(x)=n_\Omega(x)$ for $x\in\partial ^* \Omega\setminus \partial^* E$ and $n_\bullet(x)=n_E(x)$ for $x\in\partial^* E$.
We aim to study the limit of the right hand side when $\ep$ goes to $0$ along a given sequence.  
For any $i\in I_1\cup I_2\cup I_3$, thanks to inequality \eqref{eq:ubcond4}, we have
$$\cH ^{d-1}\left(\{y\in B(x_i,r_i)\cap \fE:\,|f(y)-f(x_i)| \geq \sqrt{\ep}\}\right)\leq \sqrt{\ep}\,\cH^{d-1}(B(x_i,r_i)\cap \fE)\,.$$
Moreover, thanks to inequalities \eqref{eq:ubcond3} and \eqref{eq:ubcond3bis}, we have for $i\in I_1\cup I_3$
$$\cH ^{d-1}\left(\{y\in B(x_i,r_i)\cap \partial ^* E :\,\|n_E(y)-n_E(x_i)\|_2 \geq \sqrt{\ep}\}\right)\leq \sqrt{\ep}\,\cH^{d-1}(B(x_i,r_i)\cap \fE)$$
and for $i\in I_2$
$$\cH ^{d-1}\left(\{y\in B(x_i,r_i)\cap (\partial ^* E \cup\partial ^* \Omega):\,\|n_\Omega(y)-n_\Omega(x_i)\|_2 \geq \sqrt{\ep}\}\right)\leq \sqrt{\ep}\,\cH^{d-1}(B(x_i,r_i)\cap \fE)\,.$$
Set 
\begin{align*}
\cN_\ep= &\bigcup_{i\in I_1\cup I_2\cup I_3} \left\{y\in B(x_i,r_i)\cap \fE:\,|f(y)-f(x_i)| \geq \sqrt{\ep}\right\}\\
&\qquad\cup\bigcup_{i\in I_1\cup I_3}\left \{y\in B(x_i,r_i)\cap \partial ^* E:\,\|n_E(y)-n_E(x_i)\|_2 \geq \sqrt{\ep}\right\}\\
&\quad \cup \bigcup_{i\in I_2}\left \{y\in B(x_i,r_i)\cap \partial ^* \Omega\setminus \overline{\partial ^* E}:\,\|n_\Omega(y)-n_\Omega(x_i)\|_2 \geq \sqrt{\ep}\right\} \,.
\end{align*}
Since the balls are disjoint, we have 
\begin{align}\label{eq:contrnep}
\cH^{d-1} (\cN_\ep)&\leq 2 \sqrt{\ep}\sum _{i\in I_1\cup I_2\cup I_3}\cH ^{d-1}(B(x_i,r_i)\cap\fE )\leq 2\sqrt{\ep}\cH^{d-1}(\fE)&\leq 2 \sqrt{\ep}(\cP(E,\Omega)+\cH^ {d-1}(\partial ^* \Omega))\,.
\end{align}
For any $p\geq 1$, set $\ep_p=p ^{-4}$. Let $ x\in\liminf_{p\rightarrow\infty}(\fE_{\ep_p}\setminus \cN_{\ep_p})=\cup_{p\geq 0}\cap_{m\geq p}(\fE_{\ep_m}\setminus \cN_{\ep_m})$. Hence, for $p$ large enough, we have $x\in \cap_{m\geq p}(\fE_{\ep_m}\setminus \cN_{\ep_m})$ and by construction for $m\geq p$
$$ |f(x)-f(c_{\ep_m}(x))|\leq \sqrt{\ep_m}\,.$$
It implies that 
$$\lim_{m\rightarrow \infty} f(c_{\ep_m}(x))=f(x)\,.$$
Similarly, if $x\in \partial ^* E\cap  \liminf_{p\rightarrow\infty}(\fE_{\ep_p}\setminus \cN_{\ep_p})$,
$$\lim_{p\rightarrow \infty} n_E(c_{\ep_p}(x))=n_E(x)\,$$
and if $x\in \partial ^* \Omega\setminus \overline{\partial ^* E}\cap  \liminf_{p\rightarrow\infty}(\fE_{\ep_p}\setminus \cN_{\ep_p})$,
$$\lim_{p\rightarrow \infty} n_\Omega(c_{\ep_p}(x))=n_\Omega(x)\,.$$
Using inequality \eqref{eq:ubcond1} and inequality \eqref{eq:contrnep}, we have 
\begin{align*}
\sum_{p\geq 1}\left(\cH^{d-1}(\cN_{\ep_p})+\cH^{d-1}(\fE\setminus \fE_{\ep_p})\right)\leq \sum_{p \geq 1}\left(\frac{2}{p^2}(\cP(E,\Omega)+\cH^ {d-1}(\partial ^* \Omega))+\frac{1}{p^4}\right)<\infty\,.
\end{align*}
By Borel-Cantelli lemma, it yields
$$\cH^{d-1}\left(\limsup_{p\rightarrow \infty}\,\cN_{\ep_p}\cup (\fE\setminus \fE_{\ep_p})\right)=0\,.$$
Hence, we get
\[\cH^{d-1}\left(\fE\cap (\limsup_{p\rightarrow \infty}\,\cN_{\ep_p}\cup (\fE\setminus \fE_{\ep_p}))^c\right)=\cH^{d-1}(\fE)\]
and
\begin{align}\label{eq:negneg}
\liminf_{p\rightarrow\infty}\ind_{\fE_{\ep_p}\setminus \cN_{\ep_p}}=\ind_{\fE} \qquad\text{$\cH^{d-1}$-almost everywhere}\,.
\end{align}
We have
\begin{align*}
\limsup_{\ep\rightarrow 0}&\frac{g(\ep)}{1+\ep}\int_{\fE}\cJ_{n_\bullet(c_\ep(x))}\left(\frac{f(c_\ep(x))+\kappa\ep+\kappa_0\sqrt{\ep}}{g(\ep)}\right)\ind_{\fE_\ep}(x)d\cH^{d-1}(x)\\
&\geq \limsup_{p\rightarrow\infty}\frac{g(\ep_p)}{1+\ep_p}\int_{\fE}\cJ_{n_\bullet(c_{\ep_p}(x))}\left(\frac{f(c_{\ep_p}(x))+\kappa\ep_p+\kappa_0\sqrt{\ep_p}}{g(\ep_p)}\right)\ind_{\fE_{\ep_p}\setminus \cN_{\ep_p}}(x)d\cH^{d-1}(x)\\
&\geq \liminf_{p\rightarrow\infty}\frac{g(\ep_p)}{1+\ep_p}\int_{\fE}\cJ_{n_\bullet(c_{\ep_p}(x))}\left(\frac{f(c_{\ep_p}(x))+\kappa\ep_p+\kappa_0\sqrt{\ep_p}}{g(\ep_p)}\right)\ind_{\fE_{\ep_p}\setminus \cN_{\ep_p}}(x)d\cH^{d-1}(x)\\
&\geq\int_{\fE}\liminf_{p\rightarrow\infty}\cJ_{n_\bullet(c_{\ep_p}(x))}\left(\frac{f(c_{\ep_p}(x))+\kappa\ep_p+\kappa_0\sqrt{\ep_p}}{g(\ep_p)}\right)\ind_{\fE_{\ep_p}\setminus \cN_{\ep_p}}(x)d\cH^{d-1}(x)
\end{align*}
where we use Fatou lemma in the last inequality.
Using corollary \ref{cor:liminf} and \eqref{eq:negneg}, we have
for $\cH^{d-1}$ almost every $x\in\fE$
$$\liminf_{p\rightarrow\infty}\cJ_{n_{\bullet}(c_{\ep_p}(x))}\left(\frac{f(c_{\ep_p}(x))+\kappa\ep_p+\kappa_0\sqrt{\ep_p}}{g(\ep_p)}\right)\geq \cJ_{n_{\bullet}(x)}(f(x))\,. $$
Finally, it follows that
\begin{align*}
\limsup_{\ep\rightarrow 0}&\frac{g(\ep)}{1+\ep}\int_{\fE}\cJ_{n_\bullet(c_\ep(x))}\left(\frac{f(c_\ep(x))+\kappa\ep+\kappa_0\sqrt{\ep}}{g(\ep)}\right)\ind_{\fE_\ep}(x)d\cH^{d-1}(x)\geq \cI(E,f)\,.
\end{align*}
\noindent {\bf Calibration of constants.}
Let $\delta_0>0$. If $\cI(E,f)<\infty$, we can choose $\ep$ such that 
\begin{align*}
&\frac{g(\ep)}{1+\ep}\int_{\fE}\cJ_{n_\bullet(c_\ep(x))}\left(\frac{f(c_\ep(x))+\kappa\ep+\kappa_0\sqrt{\ep}}{g(\ep)}\right)\ind_{\fE_\ep}(x)d\cH^{d-1}(x)\geq (1-\delta_0)\cI(E,f)\,.
\end{align*}
If $\cI(E,f)=\infty$, for any $t>0$, there exists $\ep$ such that  
\begin{align*}
\frac{g(\ep)}{1+\ep}\int_{\fE}\cJ_{n_\bullet(c_\ep(x))}\left(\frac{f(c_\ep(x))+\kappa\ep+\kappa_0\sqrt{\ep}}{g(\ep)}\right)\ind_{\fE_\ep}(x)d\cH^{d-1}(x)\geq t\,.
\end{align*}
The result follows by choosing the covering associated to this $\ep$ and its associated neighborhood as defined in \eqref{eq:defU}.
\end{proof}

\section{Lower large deviation principle\label{sect:goodtaux}}

We recall that we endow $\cM(\sR^d)$ with the weak topology and $\cB(\sR^d)$ with the topology associated with the distance $\dis$.
We denote by $\cU$ the basis of neighborhood of the origin of $\cB(\sR^d)\times \cM(\sR^d)$ for the associated product topology.

\begin{prop}[Lower semi-continuity of the rate function]\label{prop:lsc}The function $(E,\nu)\mapsto\widetilde{I}(E,\nu)$ is lower semi-continuous, \textit{i.e.}, for any $(E_0,\nu_0)\in\cB(\sR^d)\times \cM(\sR^d)$, for any $t\geq 0$ such that $t<\widetilde{I}(E_0,\nu_0)$, there exists a neighborhood $U(E_0,\nu_0)$ such that for any $(E,\nu)\in U(E_0,\nu_0)$, we have $\widetilde{I}(E,\nu)\geq t$.
\end{prop}

\begin{proof} Let $(E_0,\nu_0)\in\cB(\sR^d)\times \cM(\sR^d)$.
 Let $t>0$ such that $\widetilde I(E_0,\nu_0)>t$. We claim that there exists 
a neighborhood $U$ of $(E_0,\nu_0)$ such that 
\begin{align}\label{eq:sec6lsc}
\lim_{\ep\rightarrow 0}\liminf_{n\rightarrow \infty}\frac{1}{n^{d-1}}\log \Prb(\exists \cE_n\in\sC_n(\ep):\,(\bR(\cE_n),\mu_n(\cE_n))\in U)\leq -t\,.
\end{align}
We first admit this claim and show how it implies proposition \ref{prop:lsc}.
Let $U$ be a neighborhood of $(E_0,\nu_0)$ such that \eqref{eq:sec6lsc} holds.
Let $(E,\nu)\in U$. 
 If $\widetilde{I}(E,\nu)=+\infty$, then we have trivially $\widetilde{I}(E,\nu)\geq t$. If $\widetilde{I}(E,\nu)<\infty$, since $U$ is also a neighborhood of $(E,\nu)$, then we have using proposition \ref{prop:lowerbound}
\begin{align*}
-\widetilde{I}(E,\nu)&\leq\lim_{\ep\rightarrow 0}\liminf_{n\rightarrow\infty}\frac{1}{n^{d-1}}\log\Prb\left(\exists \cE_n\in \sC_n(\ep): (\bR(\cE_n),\mu_n(\cE_n))\in U\right)\leq -t\,.
\end{align*}
It follows that $\widetilde{I}(E,\nu)\geq t$. Hence, \[\forall (E,\nu)\in U\qquad \widetilde{I}(E,\nu)\geq t\,.\] The result follows.
It remains to prove the existence of $U$ such that inequality \eqref{eq:sec6lsc} holds.
We first consider the case where $(E_0,\nu_0)\in\fT_{\cM}$. By a straightforward application of the proposition \ref{prop:upperbound}, there exists a neighborhood $U$ of $(E_0,\nu_0)$ such that
\[\limsup_{n\rightarrow\infty}\frac{1}{n^{d-1}}\log\Prb\left(\exists \cE_n\in \sC_n(\Gamma^1,\Gamma^2,\Omega): (\bR(\cE_n),\mu_n(\cE_n))\in U\right)\leq -t\,.\]
It follows that
\[\lim_{\ep\rightarrow 0}\liminf_{n\rightarrow\infty}\frac{1}{n^{d-1}}\log\Prb\left(\exists \cE_n\in \sC_n(\ep): (\bR(\cE_n),\mu_n(\cE_n))\in U\right)\leq -t\,.\]
 We now consider the case where $(E_0,\nu_0)\notin\fT_{\cM}$. By proposition \ref{prop:admissible}, there exists 
a neighborhood $U$ of $(E_0,\nu_0)$ such that 
$$\lim_{\ep\rightarrow 0}\liminf_{n\rightarrow \infty}\frac{1}{n^{d-1}}\log \Prb(\exists \cE_n\in\sC_n(\ep):\,(\bR(\cE_n),\mu_n(\cE_n))\in U)\leq -t\,.$$
\end{proof}
To prove theorem \ref{thm:pgd}, it is sufficient to prove that $\widetilde{I}$ is a good rate function, a tightness result and that the local estimates are satisfied (see section 6.2 in \cite{Cerf:StFlour}, even if it is not a real large deviation principle, we can follow exactly the same steps as in section 6.2). Theorem \ref{thm:pgd} is thus a direct consequence of the following proposition.

\begin{prop}\label{prop:goodettightness}

The function $\widetilde{I}$ is a good rate function. There exist positive constants $c$ and $\lambda_0$ such that
\[\forall \lambda\geq \lambda_0 \quad\forall U \in\mathcal{U}\quad \lim_{\ep\rightarrow 0}\limsup_{n\rightarrow \infty}\frac{1}{n^{d-1}}\log \Prb\left(\exists \cE_n\in \sC_n(\ep): (\bR(\cE_n),\mu_n(\cE_n))\notin \widetilde{I}^{-1} ([0,\lambda])+U\right)\leq -c\lambda\,.\]
Moreover, the following local estimates are satisfied 
\begin{align*}
\forall (E,\nu)\in \cB(\sR^d)\times \cM(\sR^d)&\quad\forall U\in\cU \\ &\lim_{ \ep\rightarrow 0}\liminf_{n\rightarrow\infty}\frac{1}{n^{d-1}}\log\Prb\left(\exists\cE_n\in\sC_n(\ep):(\bR(\cE_n),\mu_n(\cE_n))\in (E,\nu)+U\right)\geq -\widetilde{I}(E,\nu)\,,
\end{align*}
\begin{align*}
\forall (E,\nu)\in \cB(\sR^d)&\times \cM(\sR^d) \text{  such that  } \widetilde I (E,\nu)<\infty, \,\forall \ep>0 \quad \exists U\in \cU \\ &\hfill\lim_{ \ep\rightarrow 0}\limsup_{n\rightarrow\infty}\frac{1}{n^{d-1}}\log\Prb\left(\exists\cE_n\in\sC_n(\ep):(\bR(\cE_n),\mu_n(\cE_n))\in (E,\nu)+U\right)\leq -(1-\ep)\widetilde{I}(E,\nu)\,.
\end{align*}
\end{prop}

\begin{proof}[Proof of proposition \ref{prop:goodettightness}]

{\noindent \bf Step 1. We prove that $\widetilde{I}$ is a good rate function.} 
 Let us prove that its level sets are compact.
Set
\begin{align}\label{eq:setM}
\mathbb{M}=\left\{\nu\in\cM(\sR^d):\nu(\cV_2(\Omega,1))\leq 11d^2M \cH^{d-1}(\Gamma^1), \,\nu(\cV_2(\Omega,1)^c)=0\right\}\,.
\end{align} 
 The set $\mathbb{M}$ is relatively compact for the weak topology by Prohorov theorem. Let us check that this set is closed. Let $(\nu_n)_{n\geq 1}$ be a sequence of elements in $\mathbb{M}$ that converges weakly towards $\nu$.
By Portmanteau theorem, we have 
$$\nu(\cV_2(\Omega,1)^c)\leq \liminf_{n\rightarrow\infty}\nu_n(\cV_2(\Omega,1)^c)=0\,$$
and
$$\nu(\sR^d)\leq \liminf_{n\rightarrow\infty}\nu_n(\sR^d)\leq 11d^2M\cH^{d-1}(\Gamma^1)\,.$$
It follows that 
$$\nu(\cV_2(\Omega,1))\leq 11d^2M\cH^{d-1}(\Gamma^1)$$
and $\nu\in\mathbb{M}$. The set $\mathbb{M}$ is compact for the weak topology.
 Let $\lambda>0$. 
By lemma \ref{lem:Zhang}, there exists $\beta>0$ depending on $\lambda$ and $\Omega$ such that  
$$\forall n\geq 1\qquad \Prb(\exists \cE_n \in\sC_n(\Gamma^1,\Gamma^2,\Omega): V(\cE_n)\leq 11d^2M \cH^{d-1}(\Gamma^1) n^{d-1},\,\card(\cE_n)\geq \beta n ^{d-1})\leq \exp(-2\lambda n^{d-1})$$
and so
\begin{align}\label{eq:contradiper}
\liminf_{n\rightarrow \infty }\frac{1}{n^{d-1}}\log \Prb\left(\exists \cE_n \in\sC_n(\Gamma^1,\Gamma^2,\Omega):\, V(\cE_n)\leq 11d^2M \cH^{d-1}(\Gamma^1) n^{d-1},\, \card(\cE_n)\geq \beta n ^{d-1}\right)\leq -2\lambda\,.
\end{align}
Let us assume there exists $(E,\nu)\in\fT_{\cM}$ such that $\cP(E,\Omega)>\beta$ and $\widetilde{I}(E,\nu)\leq \lambda$.
Since $F\mapsto \cP(F,\Omega)$ is lower semi-continuous, there exists a neighborhood $U_0$ of $E$ such that for any $F\in U_0$, $\cP(F,\Omega)\geq \beta$.
Let $U_1$ be a neighborhood of $\nu$. By proposition \ref{prop:lowerbound}, we have 
$$ -\widetilde I (E,\nu)\leq \lim_{\ep\rightarrow 0}\liminf_{n\rightarrow\infty}\frac{1}{n^{d-1}}\log \Prb(\exists \cE_n\in\sC_n(\ep):\,(\bR(\cE_n),\mu_n(\cE_n))\in U_0\times U_1)\,.$$
Using that $\widetilde I (E,\nu)\leq \lambda$, it follows that 
\begin{align*}
-\lambda&\leq \lim_{\ep\rightarrow 0}\liminf_{n\rightarrow\infty}\frac{1}{n^{d-1}}\log \Prb(\exists \cE_n\in\sC_n(\ep):\,(\bR(\cE_n),\mu_n(\cE_n))\in U_0\times U_1)\\&\leq \liminf_{n\rightarrow \infty }\frac{1}{n^{d-1}}\log \Prb(\exists \cE_n \in\sC_n(\Gamma^1,\Gamma^2,\Omega):\, V(\cE_n)\leq 11d^2M \cH^{d-1}(\Gamma^1) n^{d-1},\, \bR(\cE_n)\in U_0)\\
&\leq \liminf_{n\rightarrow \infty }\frac{1}{n^{d-1}}\log \Prb(\exists \cE_n \in\sC_n(\Gamma^1,\Gamma^2,\Omega):\, V(\cE_n)\leq 11d^2M \cH^{d-1}(\Gamma^1) n^{d-1},\,\cP(\bR(\cE_n),\Omega)\geq \beta)\\
&=\liminf_{n\rightarrow \infty }\frac{1}{n^{d-1}}\log \Prb(\exists \cE_n \in\sC_n(\Gamma^1,\Gamma^2,\Omega):\, V(\cE_n)\leq 11d^2M \cH^{d-1}(\Gamma^1) n^{d-1},\, \card(\cE_n)\geq \beta n ^{d-1})\,,
\end{align*}
this contradicts inequality \eqref{eq:contradiper}. Consequently, if $\widetilde{I}(E,\nu)\leq \lambda$, then $\cP(E,\Omega)\leq \beta$.
It follows that the set $\widetilde{I}^{-1}([0,\lambda])\subset \sC_\beta\times \mathbb{M}$. Since the set $\sC_\beta$ is compact for the topology associated to the distance $\dis$ and $\mathbb{M}$ is compact for the weak topology, then the set $\sC_\beta\times \mathbb{M}$ is compact for the associated product topology. Besides, since $\widetilde{I}$ is lower semi-continuous, its level sets are closed. It follows that $\widetilde{I}^{-1}([0,\lambda])$ is compact. 
This implies that $\widetilde{I}$ is a good rate function.

{\noindent \bf Step 2. We prove the $\widetilde{I}$-tightness.} Let $\lambda \geq 0$. Let $U \in\mathcal{U}$. Let $K>0$. Let $\ep_0>0$.
For any $(E,\nu)\in \fT_{\cM}$ such that $\widetilde{I}(E,\nu)<+\infty$, by proposition \ref{prop:upperbound}, there exists a neighborhood $U_{(E,\nu)}$ such that 
$$\limsup_{n\rightarrow \infty }\frac{1}{n^{d-1}}\log \Prb(\exists\cE_n\in\sC_n(\Gamma^1,\Gamma^2,\Omega): \,(\bR(\cE_n),\mu_n(\cE_n))\in U_{(E,\nu)})\leq -(1-\ep_0)\widetilde{I}(E,\nu)\,.$$
For $(E,\nu)$ such that $\widetilde{I}(E,\nu)\leq \lambda$, up to taking $U_{(E,\nu)}\cap (\widetilde{I}^{-1}([0,\lambda]+ U)$, we can assume that $U_{(E,\nu)}\subset (\widetilde{I}^{-1}([0,\lambda]+U)$.
By proposition \ref{prop:upperbound}, for $(E,\nu)\in\fT_{\cM}$ such that $\widetilde{I}(E,\nu)=+\infty$, there exists a neighborhood $U_{(E,\nu)}$ such that 
$$\limsup_{n\rightarrow \infty }\frac{1}{n^{d-1}}\log \Prb(\exists\cE_n\in\sC_n(\Gamma^1,\Gamma^2,\Omega): \,(\bR(\cE_n),\mu_n(\cE_n))\in U_{(E,\nu)})\leq -K\,.$$
Using lemma \ref{lem:Zhang}, there exists $\beta>0$ depending on $\lambda$ and $\Omega$ such that  
$$\limsup_{n\rightarrow \infty }\frac{1}{n^{d-1}}\log \Prb \left(\exists \cE_n \in\sC_n(\Gamma^1,\Gamma^2,\Omega):\, V(\cE_n)\leq 11d^2M \cH^{d-1}(\Gamma^1) n^{d-1},\, \card(\cE_n)\geq \beta n ^{d-1}\right)\leq -\lambda\,.$$
For $(E,\nu)\in\sC_\beta \times \mathbb{M}\setminus \fT_{\cM}$, by proposition \ref{prop:admissible},
there exists a neighborhood $U_{(E,\nu)}$ such that 
$$\lim_{\ep\rightarrow 0}\limsup_{n\rightarrow \infty }\frac{1}{n^{d-1}}\log \Prb(\exists\cE_n\in\sC_n(\ep): \,(\bR(\cE_n),\mu_n(\cE_n))\in U_{(E,\nu)})\leq -K\,.$$
The set $\sC_\beta=\{E\in\cB(\sR^d): \cP(E,\Omega)\leq \beta\}$ is compact for the topology associated to the distance $\dis$. The set $\mathbb{M}$ defined in \eqref{eq:setM} is compact for the weak topology.
Therefore, we can extract from $(U_{(E,\nu)},(E,\nu)\in\sC_\beta \times \mathbb{M})$ a finite covering 
$(U_{(E_i,\nu_i)})_{i=1,\dots, N}$ of $\sC_\beta\times \mathbb{M}$. Let $\ep>0$.
We have
\begin{align*}
\Prb&\left(\exists \cE_n\in\sC_n(\ep):\,(\bR(\cE_n),\mu_n(\cE_n))\notin \widetilde{I}^{-1} ([0,\lambda])+U\right)\\&\leq \sum_{i=1}^N\Prb\left(\exists \cE_n\in\sC_n(\ep):\,(\bR(\cE_n),\mu_n(\cE_n))\notin \widetilde{I}^{-1} ([0,\lambda])+U,\,(\bR(\cE_n),\mu_n(\cE_n))\in U_{(E_i,\nu_i)} \right)\\
&\quad+\Prb\left( \exists \cE_n \in\sC_n(\Gamma^1,\Gamma^2,\Omega):\, V(\cE_n)\leq 11d^2M \cH^{d-1}(\Gamma^1) n^{d-1},\, \card(\cE_n)\geq \beta n ^{d-1}\right)\,.
\end{align*}
If $\widetilde{I}(E_i,\nu_i)\leq\lambda$, since by construction $U_{(E_i,\nu_i)}\subset (\widetilde{I}^{-1} ([0,\lambda])+U)$, then
$$\Prb\left(\exists \cE_n\in\sC_n(\ep):\,(\bR(\cE_n),\mu_n(\cE_n))\in (\widetilde{I}^{-1} ([0,\lambda])+U) ^c\cap  U_{(E_i,\nu_i)} \right)=0\,.$$
By lemma \ref{lem:estimeanalyse}, it follows that
\begin{align*}
\lim_{\ep\rightarrow 0}\limsup_{n\rightarrow\infty}\frac{1}{n^{d-1}}&\log\Prb\left(\exists \cE_n\in\sC_n(\ep):\,(\bR(\cE_n),\mu_n(\cE_n))\notin \widetilde{I}^{-1} ([0,\lambda])+U\right)\\
&\leq -\min\left((1-\ep_0)\min\left\{\widetilde{I}(E_i,\nu_i):\,i=1,\dots,N,\,\widetilde{I}(E_i,\nu_i)\in]\lambda,\infty[\right\},K\right)\\&\leq -\min((1-\ep_0)\lambda,K)\,.
\end{align*}
By letting first $K$ go to infinity and then $\ep_0$ go to $0$, we obtain 
\begin{align*}
\lim_{\ep\rightarrow 0}\limsup_{n\rightarrow\infty}\frac{1}{n^{d-1}}\log\Prb\left(\exists \cE_n\in\sC_n(\ep):\,(\bR(\cE_n),\mu_n(\cE_n))\notin \widetilde{I}^{-1} ([0,\lambda])+U\right)\leq -\lambda\,.
\end{align*}
\noindent {\bf Step 3. We prove the local estimates.}
Finally, note that the local estimates are direct consequences of propositions \ref{prop:lowerbound} and \ref{prop:upperbound}.
This concludes the proof.

\end{proof}

We now study the lower large deviations for the maximal flow $\phi_n$. 
We recall the definition of $J$:
$$\forall \lambda\geq 0\qquad J(\lambda)=\inf\left\{\widetilde{I}(E,\nu):\,(E,\nu)\in\cB(\sR^d)\times \cM(\sR^d),\,\nu(\sR^d)=\lambda\right\}\,.$$
Let $\lambda_{min}$ be defined by $$\lambda_{min}=\inf\{\lambda\geq 0: J(\lambda)<\infty\}\,.$$
We first prove this intermediate result on $J$.
\begin{prop}\label{prop:propJ}
The function $J$ is lower semi-continuous. The function $J$ is finite on $]\lambda_{min},\phi_\Omega]$ Moreover, $J$ is decreasing on $]\lambda_{min},\phi_\Omega]$, $J(\lambda)=\infty$ for $\lambda\in[0,\lambda_{min}[\cup ]\phi_{\Omega},+\infty[$ and
\begin{align*}
\forall \lambda\in]0,\phi_\Omega[\qquad -J(\lambda^-)&\leq \liminf_{n\rightarrow\infty}\frac{1}{n^{d-1}}\log\Prb(\phi_n(\Gamma^1,\Gamma^2,\Omega)\leq\lambda n ^{d-1})\\
&\leq \limsup_{n\rightarrow\infty}\frac{1}{n^{d-1}}\log\Prb(\phi_n(\Gamma^1,\Gamma^2,\Omega)\leq\lambda n ^{d-1})\leq-J(\lambda)\,,
\end{align*}
where $J(\lambda^-)$ is the left hand limit of $J$ at $\lambda$.
\end{prop}

\begin{proof}[Proof of proposition \ref{prop:propJ}]
{\noindent \bf Step 1. We prove that the infimum in the definition of $J$ is attained and that the function $J$ has the desired properties.} Let $\lambda\in[0,\phi_\Omega]$ such that $J(\lambda)<\infty$.
Let $t>J(\lambda)$. Let $\beta$ be such that for any $(E,\nu)\in\fT_{\cM}$ such that $\widetilde{I}(E,\nu)\leq t $ then $\cP(E,\Omega)\leq \beta$ (see proof of proposition \ref{prop:goodettightness}, step 1).
It follows that
$$J(\lambda)=\inf\left\{\widetilde{I}(E,\nu):\,(E,\nu)\in\sC_\beta\times\mathbb{M},\,\nu(\sR^d)=\lambda\right\}\,.$$
Let $(\nu_n)_{n\geq 1}$ be a sequence such that for any $n\geq 1$, $\nu_n(\sR^d)=\lambda$ and the sequence weakly converges towards $\nu$. It yields that $\nu(\sR^d)=\lambda$.
Consequently, the set 
$$\left\{\,(E,\nu)\in\sC_\beta\times\mathbb{M},\,\nu(\sR^d)=\lambda\right\}$$
is compact.
Since the function $\widetilde{I}$ is lower semi-continuous, it attains its minimum over this set: there exists $(E,\nu)\in\fT_{\cM}$ such that $\nu(\sR^d)=\lambda$ and $J(\lambda)=\widetilde{I}(E,\nu)$.

Let $\lambda\in]\lambda_{min}, \phi_{\Omega}[$.
Let $(E,\mu)\in\fT_{\cM}$ such that $\mu(\sR^d)=\lambda'\in]\lambda_{min},\lambda[$ . We write $\mu=f \cH^{d-1}|_{\fE}$ with $(E,f)\in\fT$.
Let $f_E$ be the function defined by
\begin{align}\label{eq:deffE}
\forall x\in\fE \qquad f_E(x)=\left\{\begin{array}{lll} \nu_G(n_E(x))&\mbox{if}& x\in \partial ^* E\cap \Omega \\
\nu_G(n_\Omega(x))&\mbox{if}& x \in \partial^* \Omega  \cap( (\Gamma^1\setminus \overline{\partial^*E})\cup (\Gamma ^2\cap \partial ^*E))
\end{array}
\right.
\end{align}
We set $\mu_E=f_E\cH^{d-1}|_{\fE}$.
By construction, we have 
$$\mu_E(\sR^d)=\cI_{\Omega}(E)\geq \phi_{\Omega}\,.$$
Since $\lambda\in ]\lambda',\phi_{\Omega}[$, there exists $\alpha \in]0,1[$ such that
$(\alpha \mu+ (1-\alpha)\mu_E)(\sR^d)=\lambda$.
By convexity of $\cJ_{\vv}$ we have
$$J(\lambda)\leq \cI(E,\alpha f+(1-\alpha)f_E)\leq \alpha \cI(E,f)+(1-\alpha)\cI(E,f_E)\leq \alpha  \cI(E,f)\,$$
where we use the fact that $\cJ_{n_E(x)}(\nu_G(n_E(x))=0$ and $\cJ_{n_\Omega(x)}(\nu_G(n_\Omega(x))=0$.
It follows that 
\begin{align}\label{eq:decJ}
\widetilde{I}(E,\mu)\geq \frac{1}{\alpha}J(\lambda)> J(\lambda)\,.
\end{align}
It follows that for $\lambda\leq \phi_\Omega$,
\begin{align*}
J(\lambda)
&= \inf\left\{\widetilde{I}(E,\nu):\,(E,\nu)\in\cB(\sR^d)\times \cM(\sR^d),\,\nu(\sR^d)\leq \lambda\right\}\,.
\end{align*}
Taking the infimum in \eqref{eq:decJ} for any $(E,\nu)$ such that $\mu(\sR^d)=\lambda'$, we obtain
$J(\lambda')>J(\lambda)$.
Hence, the function $J$ is decreasing on $]\lambda_{min},\phi_{\Omega}[$. It follows that $J$ is finite on $]\lambda_{min},\phi_{\Omega}]$. We claim that $J(\phi_\Omega)=0$. By theorem \ref{thm:CerfTheret}, the set $\Sigma^a$ is not empty, there exists $F\subset \Omega$ such that $\cI_\Omega(F)=\phi_\Omega=\capa(F,\nu_G(n_\bullet)\ind_\fF)$. It is easy to check that $(F,\nu_G(n_\bullet)\ind_\fF)\in\fT$ and $\cI(F,\nu_G(n_\bullet)\ind_\fF)=0$ using that $\cJ_{n_\bullet}(\nu_G(n_\bullet))=0$ (see theorem \ref{thm:lldtau}). It follows that $J(\phi_\Omega)=0$.

Moreover for any $(E,\nu)\in\fT_{\cM}$, the condition of minimality implies that $\nu(\sR^d)\leq\capa(F,\nu_G(n_\bullet)\ind_\fF)= \phi_\Omega$. It follows that for any $\lambda>\phi_\Omega$, we have $J(\lambda)=+\infty$.

{\noindent \bf Step 2. We prove a lower bound.}
Let $\lambda>0$. 
We aim to prove that 
\begin{align}\label{eq:pgdflowliminf}
\forall \delta>0 \qquad -J(\lambda)\leq \liminf_{n\rightarrow\infty}\frac{1}{n^{d-1}}\log \Prb\left(\phi_n(\Gamma^1,\Gamma^2,\Omega)\leq(\lambda+\delta)n ^{d-1}\right)\,.
\end{align}
The result is clear if $J(\lambda)=+\infty$. Let us now assume that $J(\lambda)<\infty$.
Let $(E,\nu)\in\fT_{\cM}$ such that $\widetilde{I}(E,\nu)=J(\lambda)$ and $\nu(\sR^d)=\lambda$. 
 Let $g\in\sC_c(\sR^d,\sR)$ such that $g=1$ on $\cV_2(\Omega,1)$ and $g\geq 0$. Hence, we have $\nu(g)=\lambda$. Let $\delta>0$.
Let $U_0$ be a neighborhood of $E$ and $U_1=\{\mu\in\cM(\sR^d):|\mu(g)-\nu(g)|\leq \delta\}$.
We have
\begin{align*}
\forall \ep>0 \qquad \Prb(\exists \cE_n\in\sC_n(\ep):\,(\bR(\cE_n),\mu_n(\cE_n))\in U_0\times U_1)
&\leq \Prb \left(\exists \cE_n \in \sC_n(\Gamma^1,\Gamma^2,\Omega):\, V(\cE_n)\leq (\lambda+\delta)n^{d-1}\right)\\
&= \Prb (\phi_n(\Gamma^1,\Gamma^2,\Omega)\leq(\lambda+\delta)n ^{d-1})\,.
\end{align*}
Using proposition \ref{prop:lowerbound}, it follows that
\begin{align*}
-J(\lambda)=-\widetilde{I}(E,\nu)\leq \liminf_{n\rightarrow\infty}\frac{1}{n^{d-1}}\log \Prb\left(\phi_n(\Gamma^1,\Gamma^2,\Omega)\leq(\lambda+\delta)n ^{d-1}\right)\,.
\end{align*}

{\noindent \bf Step 3. We prove an upper bound.}
Let $\lambda\in[0,\phi_\Omega]$. We aim to prove that
\begin{align}\label{eq:pgdfluxlimsup}
\forall \delta>0\qquad\limsup_{n\rightarrow\infty}\frac{1}{n^{d-1}}\log\Prb(\phi_n(\Gamma^1,\Gamma^2,\Omega)\leq \lambda n ^{d-1})\leq -J\left(\min(\lambda+\delta,\phi_\Omega)\right).
\end{align}
Let $K>0$, $\delta>0$ and $\ep_0>0$.
Thanks to proposition \ref{prop:upperbound}, to each $(E,\nu)\in\fT_{\cM}$ such that $\widetilde{I}(E,\nu)<\infty$, we can associate a neighborhood $U_{(E,\nu)}$ of $(E,\nu)$ such that
$$\limsup_{n\rightarrow \infty }\frac{1}{n^{d-1}}\log \Prb(\exists\cE_n\in\sC_n(\Gamma^1,\Gamma^2,\Omega): \,(\bR(\cE_n),\mu_n(\cE_n))\in U_{(E,\nu)})\leq -(1-\ep_0)\widetilde{I}(E,\nu)\,.$$
Up to taking $U_{(E,\nu)}\cap(\cB(\sR^d)\times \{\mu\in\cM(\sR^d):|\mu(g)-\nu(g)|< \delta\}$, we can assume that $$U_{(E,\nu)}\subset(\cB(\sR^d)\times \{\mu\in\cM(\sR^d):|\mu(g)-\nu(g)|< \delta\}\,.$$ We recall that $g$ was defined in the previous step.
Since $\widetilde{I}$ is a good rate function, the set $\widetilde{I}^{-1}([0,K])$ is compact
we can extract from $(U_{(E,\nu)}, (E,\nu)\in \widetilde{I}^{-1}([0,K]))$ a finite covering $(U_{(E_i,\nu_i)},i=1,\dots,N)$ of $\widetilde{I}^{-1}([0,K])$.
From lemma 6.6 in \cite{Cerf:StFlour}, there exists a neighborhood $U$ of $0$ such that 
$$\widetilde{I}^{-1}([0,K])+ U \subset \bigcup_{i=1}^NU_{(E_i,\nu_i)}\,.$$
Let $\ep>0$. It is easy to check that the $(\Gamma_n^1,\Gamma_n^2)$-cutset $\cE_n^{min}$ that achieves the minimal capacity (if there are several cutsets that achieve the minimal capacity, we choose one according to a deterministic rule) is in $\sC_n(\ep)$.
Using the previous inclusion, we have
\begin{align*}
\Prb(\phi_n(\Gamma^1,\Gamma^2,\Omega)\leq\lambda n ^{d-1}) &\leq \Prb(\exists \cE_n\in \sC_n(\ep): V(\cE_n)\leq \lambda n ^{d-1})\\
&\leq\sum_{i=1}^N \Prb(\exists \cE_n\in \sC_n(\Gamma_1,\Gamma_2,\Omega): V(\cE_n)\leq \lambda n ^{d-1},(\bR(\cE_n),\mu_n(\cE_n))\in U_{(E_i,\nu_i)} )\\
&+\Prb( \exists \cE_n\in \sC_n(\ep); (\bR(\cE_n),\mu_n(\cE_n))\notin \widetilde{I}^{-1}([0,K])+ U)\,.
\end{align*}
Note that on the event $\{\exists \cE_n\in \sC_n(\Gamma_1,\Gamma_2,\Omega): V(\cE_n)\leq \lambda n ^{d-1}\}$, if for $i\in \{1,\dots,N\}$, $(\bR(\cE_n),\mu_n(\cE_n))\in U_{(E_i,\nu_i)}$ then we have 
$$\nu_i(\sR^d)=\nu_i(g)\leq \mu_n(\cE_n)(g)+\delta \leq \lambda +\delta\,.$$
Consequently, for $i$ such that $\nu_i(\sR^d)>\lambda +\delta$, we have
$$\Prb(\exists \cE_n\in \sC_n(\Gamma_1,\Gamma_2,\Omega): V(\cE_n)\leq \lambda n ^{d-1},(\bR(\cE_n),\mu_n(\cE_n))\in U_{(E_i,\nu_i)} )=0\,.$$
Using proposition \ref{prop:goodettightness}, we have
\begin{align*}
\lim_{\ep\rightarrow 0}\limsup_{n\rightarrow\infty}\frac{1}{n^{d-1}}\log\Prb\left(\exists \cE_n\in\sC_n(\ep):\,(\bR(\cE_n),\mu_n(\cE_n))\notin \widetilde{I}^{-1} ([0,K])+U\right)\leq -c K\,.
\end{align*}
Using lemma \ref{lem:estimeanalyse}, it follows that 
\begin{align*}
\limsup_{n\rightarrow\infty}\frac{1}{n^{d-1}}&\log\Prb(\phi_n(\Gamma^1,\Gamma^2,\Omega)\leq\lambda n ^{d-1})\\&\leq -\min\left((1-\ep_0)\min\{\widetilde{I}(E_i,\nu_i):\nu_i(\sR^d)\leq\lambda+\delta, i=1,\dots,N\} ,cK\right)\\&\leq -\min((1-\ep_0)J\left(\min(\lambda+\delta,\phi_\Omega)\right),cK)\,
\end{align*}
where we use that for any $i\in\{1,\dots,N\}$, since $(E_i,\nu_i)\in\fT_{\cM}$, by the minimality condition, we have $\nu_i(\sR^d)\leq \phi_\Omega$.
If $J\left(\min(\lambda+\delta,\phi_\Omega)\right)=+\infty$, then
by letting $K$ go to infinity, we obtain
\begin{align*}
\limsup_{n\rightarrow\infty}\frac{1}{n^{d-1}}\log\Prb(\phi_n(\Gamma^1,\Gamma^2,\Omega)\leq \lambda n ^{d-1})\leq -\infty = -J\left(\min(\lambda+\delta,\phi_\Omega)\right).
\end{align*}
If $J\left(\min(\lambda+\delta,\phi_\Omega)\right)<\infty$, then
by letting $K$ go to infinity and then $\ep_0$ go to $0$, we obtain
\begin{align*}
\limsup_{n\rightarrow\infty}\frac{1}{n^{d-1}}\log\Prb(\phi_n(\Gamma^1,\Gamma^2,\Omega)\leq \lambda n ^{d-1})\leq -J\left(\min(\lambda+\delta,\phi_\Omega)\right)\,.
\end{align*}
The result follows.

{\noindent \bf Step 4. We prove that the function $J$ is lower semi-continuous.} Let $\lambda>0$.
Let $(\lambda_n)_{n\geq 1}$ be a sequence of non negative real number that converges towards $\lambda$. If $\liminf_{n\rightarrow \infty }J(\lambda_n)=+\infty$ there is nothing to prove. Let us assume that 
$\liminf_{n\rightarrow \infty }J(\lambda_n)<\infty$. Let $\psi$ be an extraction such that
\begin{align}\label{eq:lscJ}
\lim_{n\rightarrow \infty}J(\lambda_{\psi(n)})= \liminf_{n\rightarrow \infty }J(\lambda_n)\,
\end{align}
and for all $n\geq 1$, $J(\lambda_{\psi(n)})<\infty$.
There exists $(E_n,\nu_n)\in\fT_{\cM}$ such that 
$$\nu_n(\sR^d)=\lambda_{\psi(n)}\qquad \text{and}\qquad J(\lambda_{\psi(n)})=\widetilde{I}(E_n,\nu_n)\,.$$
Since the sequence $(J(\lambda_{\psi(n)}))_{n\geq 1}$ converges, there exists $m>0$ such that 
$$\forall n \geq 1\qquad J(\lambda_{\psi(n)})\leq m\,.$$
Since $\widetilde{I}$ is a good rate function, the set $\widetilde{I}^{-1}([0,m])$ is compact. We can extract from the sequence $(E_n,\nu_n)_{n\geq 1}$ a sequence $(E_{\phi(n)},\nu_{\phi(n)})_{n\geq 1}$ that converges towards $(E,\nu)\in \widetilde{I}^{-1}([0,m])$ ($E_n$ converges for the distance $\dis$ and $\nu_n$ converges weakly towards $\nu$). Since $\widetilde{I}$ is lower semi-continuous, we have
$$\lim_{n\rightarrow\infty}J(\lambda_{\psi(n)})=\liminf_{n\rightarrow\infty}J(\lambda_{\psi(\phi(n))})=\liminf_{n\rightarrow \infty}\widetilde{I}(E_{\phi(n)},\nu_{\phi(n)})\geq \widetilde{I}(E,\nu)\,.$$
Since $\nu_{\phi(n)}$ weakly converges towards $\nu$,  we have 
$$\nu(\sR^d)=\lim_{n\rightarrow\infty}\nu_{\phi(n)}(\sR^d)=\lim_{n\rightarrow\infty}\lambda_{\psi(\phi(n))}=\lambda\,.$$
Combining the two previous inequalities we obtain
$$\lim_{n\rightarrow\infty}J(\lambda_{\psi(n)})\geq \widetilde{I}(E,\nu)\geq J(\lambda)\,.$$
By \eqref{eq:lscJ}, we have
$$J(\lambda)\leq \liminf_{n\rightarrow\infty}J(\lambda_n)\,.$$
It follows that $J$ is lower semi-continuous on $\sR^+$.

{\noindent \bf Step 5. Conclusion.}
 Using inequality \eqref{eq:pgdflowliminf}, we have 
\[\forall \lambda> 0\qquad \liminf_{n\rightarrow\infty}\frac{1}{n^{d-1}}\log\Prb(\phi_n(\Gamma^1,\Gamma^2,\Omega)\leq\lambda n ^{d-1})\geq -\lim_{\delta \rightarrow 0}J(\lambda-\delta):=-J(\lambda^-)\]
and using inequality \eqref{eq:pgdfluxlimsup} and the fact that $J$ is lower semi-continuous \[\forall \lambda\in[0,\phi_{\Omega}]\qquad\limsup_{n\rightarrow\infty}\frac{1}{n^{d-1}}\log\Prb(\phi_n(\Gamma^1,\Gamma^2,\Omega)\leq\lambda n ^{d-1})\leq -\lim_{\delta \rightarrow 0}J\left(\min(\lambda+\delta,\phi_\Omega)\right)\leq -J(\lambda)\,.\] 
\end{proof}

\begin{proof}[Proof of theorem \ref{thm:lldmf}]
Note that for any $\lambda\geq \phi_{\Omega}$, we have $J(\lambda)=+\infty$ since by definition there does not exist any $(E,\nu)\in\fT_{\cM}$ such that $\nu(\sR^d)=\lambda>\phi_\Omega$ because of the condition of minimality.

{\bf$\bullet$ Lower bound.}
We prove the local lower bound:
$$\forall \lambda\geq 0 \quad\forall\ep>0\qquad \liminf_{n\rightarrow\infty}\frac{1}{n^{d-1}}\log \Prb\left(\frac{\phi_n(\Gamma^1,\Gamma^2,\Omega)}{n^{d-1}}\in]\lambda-\ep,\lambda+\ep[\right)\geq-J(\lambda)\,.$$
Let $\lambda>0$ and $\ep>0$. If $J(\lambda)=+\infty$, there is nothing to prove. If $\lambda=\phi_\Omega$, since $J(\phi_\Omega)=0$ and by the law of large numbers for $\phi_n(\Gamma^1,\Gamma^2,\Omega)$ (see theorem \ref{thm:CerfTheret}), the result follows.
 Otherwise, $\lambda<\phi_\Omega$ and we have
\begin{align*}
\Prb\left(\frac{\phi_n(\Gamma^1,\Gamma^2,\Omega)}{n^{d-1}}\in]\lambda-\ep,\lambda+\ep[\right)\geq \Prb\left(\frac{\phi_n(\Gamma^1,\Gamma^2,\Omega)}{n^{d-1}}\leq \lambda+\delta\right)-\Prb\left(\frac{\phi_n(\Gamma^1,\Gamma^2,\Omega)}{n^{d-1}}\leq \lambda-\ep \right)
\end{align*}
where $\delta>0$ is chosen small enough such that $\lambda+\delta<\phi_\Omega$ and $\delta\leq \ep$.
Since $J(\lambda+\delta)<J(\lambda-\ep)$, by proposition \ref{prop:propJ} and lemma \ref{lem:estimeanalyse}, it leads to
$$\liminf_{n\rightarrow\infty}\frac{1}{n^{d-1}}\log \Prb\left(\frac{\phi_n(\Gamma^1,\Gamma^2,\Omega)}{n^{d-1}}\in]\lambda-\ep,\lambda+\ep[\right)\geq -J((\lambda+\delta) ^-)\geq -J(\lambda)$$
where we use that $J$ is decreasing on $]\lambda_{min},\phi_{\Omega}[$.

{\bf$\bullet$ Upper bound.} We have to prove that for all closed subset $\cF$ of $\sR^+$
\[\limsup_{n\rightarrow\infty}\frac{1}{n^{d-1}}\log \Prb\left(\frac{\phi_n(\Gamma^1,\Gamma^2,\Omega)}{n^{d-1}}\in\cF\right)\leq-\inf_{\cF}J\,.\]
Let $\cF$ be a closed subset of $\sR^+$. If $\phi_\Omega\in\cF$, then $\inf_{\cF}J=0$ and  the result is obvious.
We suppose now that $\phi_{\Omega}\notin\cF$. We consider $\cF_1=\cF\cap[0,\phi_\Omega]$ and $\cF_2=\cF\cap[\phi_{\Omega},+\infty[$. 
We claim that 
\begin{align}\label{eq:f2inutile}
\limsup_{n\rightarrow\infty}\frac{1}{n^{d-1}}\log\Prb\left(\frac{\phi_n(\Gamma^1,\Gamma^2,\Omega)}{n^{d-1}}\in\cF\right)= \limsup_{n\rightarrow\infty}\frac{1}{n^{d-1}}\log\Prb\left(\frac{\phi_n(\Gamma^1,\Gamma^2,\Omega)}{n^{d-1}}\in\cF_1\right)\,.
\end{align}
Cerf and Théret proved in theorem 1 in \cite{CT2} that the upper large deviations of $\phi_n(\Gamma^1,\Gamma^2,\Omega)$ are of volume order. Hence, if $\cF_2$ is not empty, we have $f_2=\inf\cF_2>\phi_\Omega$ and
$$\limsup_{n\rightarrow\infty}\frac{1}{n^{d-1}}\log\Prb\left(\frac{\phi_n(\Gamma^1,\Gamma^2,\Omega)}{n^{d-1}}\in \cF_2\right)=-\infty\,.$$
This equality trivially holds when $\cF_2$ is empty. By lemma \ref{lem:estimeanalyse}, we deduce equality \eqref{eq:f2inutile}.
Let us assume that $\cF_1$ is not empty.
Let $f_1=\sup \cF_1$. Since $\cF$ is closed, we have $f_1\in\cF_1$ and $f_1<\phi_\Omega$, and using equality \eqref{eq:f2inutile}, we have
\[\limsup_{n\rightarrow\infty}\frac{1}{n^{d-1}}\log\Prb\left(\frac{\phi_n(\Gamma^1,\Gamma^2,\Omega)}{n^{d-1}}\in\cF\right)\leq \limsup_{n\rightarrow\infty}\frac{1}{n^{d-1}}\log\Prb\left(\frac{\phi_n(\Gamma^1,\Gamma^2,\Omega)}{n^{d-1}}\leq f_1\right)\,.\]
Using proposition \ref{prop:propJ}, it yields that
$$\limsup_{n\rightarrow\infty}\frac{1}{n^{d-1}}\log\Prb\left(\frac{\phi_n(\Gamma^1,\Gamma^2,\Omega)}{n^{d-1}}\in\cF\right)\leq \limsup_{n\rightarrow\infty}\frac{1}{n^{d-1}}\log\Prb\left(\frac{\phi_n(\Gamma^1,\Gamma^2,\Omega)}{n^{d-1}}\leq f_1\right)=-J(f_1)=-\inf_{\cF}J$$
since $J$ is decreasing on $]\lambda_{min},\phi_{\Omega}]$.
Let us assume that $\cF_1=\emptyset$.
Using equality \eqref{eq:f2inutile}, we have
$$\limsup_{n\rightarrow\infty}\frac{1}{n^{d-1}}\log\Prb\left(\frac{\phi_n(\Gamma^1,\Gamma^2,\Omega)}{n^{d-1}}\in\cF\right)=-\infty=-\inf_{\cF_1}J=-\inf_{\cF}J\,.$$

{\bf$\bullet$ Property of $\lambda_{min}$.}
We claim that for all $\lambda<\lambda_{min}$, there exists $N\geq 1$ such that
\[\forall n\geq N\qquad\Prb\left(\frac{\phi_n(\Gamma^1,\Gamma^2,\Omega)}{n^{d-1}}\leq \lambda\right)=0\,.\]
If $\lambda_{min}=0$, we have nothing to prove. Let us assume $\lambda_{min}>0$.
We recall that \[\delta_G=\inf\left\{t: \Prb(t(e)\leq t)>0\right\}\,.\]
Note that the function 
$$\cI_0: F\mapsto \int_{\fF}\|n_\bullet(x)\|_1d\cH^{d-1}(x)$$
is lower semi-continuous. This can be deduced from the lower semi-continuity of the surface energy (see section 14.2 in \cite{Cerf:StFlour}).
Therefore, the infimum of $\cI_0$ is achieved on the following compact set
$$\{F\in\cB(\sR^d): \,F\subset \Omega, \,\cP(F,\Omega)\leq 2d\cH^{d-1}(\Gamma^1)\}\,.$$
We denote by $F_0$ a set that achieves the infimum.
Let $\lambda\in [\lambda_{min},\phi_\Omega]$. Let $(E,\nu)\in\fT_\cM$ such that $\nu(\sR^d)=\lambda$ and $\widetilde{I}(E,\nu)=J(\lambda)$. Write $\nu=f\cH^{d-1}|_{\fE}$. We have $f\geq \delta_G \|n_\bullet\|_1$ $\cH^{d-1}$-almost everywhere on $\fE$, if not, by theorem \ref{thm:lldtau}, it contradicts the fact that $\int_{\fE}\cJ_{n_\bullet(x)}(f(x))d\cH^{d-1}(x)<\infty$.
Hence, we have $$\lambda\geq \int_{\fE}\delta_G\|n_\bullet\|_1d\cH^{d-1}=\delta_G\, \cI_0(E)\geq \delta_G\, \cI_0(F_0)\,.$$
It yields that
$$\lambda_{min}\geq \delta_G\, \cI_0(F_0)\,.$$
For any $\ep>0$, we define $f_\ep$ as follows
\begin{align}\label{eq:deffep}
\forall x\in\fF_0 \qquad f_\ep(x)= (1+\ep)\delta_G\|n_{\bullet}(x)\|_1\,.
\end{align}
As long as $G$ is not a dirac mass (the study of large deviations is trivial in that case), we have $$\forall v\in\sS^{d-1}\qquad\nu_G(v)> \delta_G\|v\|_1\,.$$ Thus, we can choose $\ep$ small enough such that 
$$\forall v\in\sS^{d-1} \qquad (1+\ep)\delta_G\|v\|_1\leq \nu_G(v)\,.$$
We have \[\cI(F_0,f_\ep)=\int_{\fF_0}\cJ_{n_\bullet(x)}(f_\ep(x))d\cH^{d-1}(x)\leq \cH^{d-1}(\fF_0)\sup_{v\in\sS^{d-1}}\cJ_v((1+\ep)\delta_G\|v\|_1)<\infty\,\] We claim that $(F_0,f_\ep)$ is minimal. Let $E\subset \Omega$ such that $\cP(E,\Omega)<\infty$. We set
\[\forall x\in \fE\qquad g(x) = \left\{
    \begin{array}{ll}
        f_\ep(x)  & \mbox{if } x\in\fF_0\cap\fE\\
        \nu_G(n_\bullet(x))& \mbox{if }x\in\fE\setminus \fF_0\,.
    \end{array}
\right.
\]
Let us prove that $\capa(E,g)\geq \capa(F_0,f_\ep)$. 
We distinguish two cases. We assume first that $\cP(E,\Omega)\leq 2d \cH^{d-1}(\Gamma^1)$. 
Since $F_0$ achieves the infimum on $F_0$ on the set $\{F\in\cB(\sR^d): \,F\subset \Omega, \,\cP(F,\Omega)\leq 2d\cH^{d-1}(\Gamma^1)\}$, then $\cI_0(F_0)\leq \cI_0(E)$ and
\begin{align}\label{eq:capmin2}
\int_{\fE\setminus \fF_0}\|n_\bullet(x)\|_1d\cH^{d-1}(x)\geq \int_{\fF_0\setminus \fE}\|n_\bullet(x)\|_1d\cH^{d-1}(x)\,.
\end{align}
Thanks to the choice of $\ep$, we have
\begin{align}\label{eq:capmin1}
\int_{\fE\setminus \fF_0}\nu_G(n_\bullet(x))d\cH^{d-1}(x)\geq (1+\ep)\delta_G \int_{\fE\setminus \fF_0}\|n_\bullet(x)\|_1d\cH^{d-1}(x)\,.
\end{align}
Combining inequalities \eqref{eq:capmin2} and \eqref{eq:capmin1}, it yields that
\begin{align*}
\capa(E,g)&\geq \int_{\fE\cap\fF_0}f_\ep(x)d\cH^{d-1}(x)+ \int_{\fE\setminus \fF_0}\nu_G(n_\bullet(x))d\cH^{d-1}(x)\\
&\geq\int_{\fE\cap\fF_0}f_\ep(x)d\cH^{d-1}(x)+ \int_{\fF_0\setminus \fE}(1+\ep)\delta_G\|n_\bullet(x)\|_1d\cH^{d-1}(x)=\capa(F_0,f_\ep)\,.
\end{align*}
Let us assume that $\cP(E,\Omega)> 2d\cH^{d-1}(\Gamma^1)$.
It follows that
\begin{align*}
\capa(E,g)&\geq \int_{\fE}\delta_G\|n_\bullet(x)\|_1d\cH^{d-1}(x)\geq\delta_G\cH^{d-1}(\fE)\geq \delta_G \cP(E,\Omega)\geq 2d\delta_G\cH^{d-1}(\Gamma^1)\,.
\end{align*}
Besides, we have
\begin{align*}
\cI_0(F_0)\leq \cI_0(\emptyset)= \int_{\Gamma^1}\|n_\Omega(x)\|_1d\cH^{d-1}(x)\leq d\cH^{d-1}(\Gamma^1)\,.
\end{align*}
Hence, we have
\[\capa(F_0,f_\ep)=(1+\ep)\delta_G \cI_0(F_0)\leq (1+\ep)\delta_G d\cH^{d-1}(\Gamma^1)\leq \capa(E,g)\,.\]
Finally, $(F_0,f_\ep)$ is minimal and $(F_0,f_\ep)\in\fT$. It follows that
$$\lambda_{min}\leq (1+\ep)\delta_G\, \cI_0(F_0)\,.$$
By letting $\ep$ go to $0$, we obtain that
$$\lambda_{min}=\delta_G\, \cI_0(F_0)\,.$$
Let $\lambda<\lambda_{min}$. Let us assume there exists an increasing sequence $(a_n)_{n\geq 1}$ such that
\[\Prb\left(\frac{\phi_{a_n}(\Gamma^1,\Gamma^2,\Omega)}{a_n^{d-1}}\leq \lambda\right)>0\,.\]
On the event $\{\phi_{a_n}(\Gamma^1,\Gamma^2,\Omega)\leq \lambda a_n^{d-1}\}$ we pick $\cE_{a_n}\in\sC_{a_n}(\Gamma^1,\Gamma ^2,\Omega)$ such that $V(\cE_{a_n})\leq \lambda a_n^{d-1}$. It follows that $\card(\cE_{a_n})\leq \lambda  a_n^{d-1}/\delta_G$. 
Besides, we have 
$$\cI_0(\bR(\cE_{a_n}))a_n^{d-1}\leq \card(\cE_{a_n})\,.$$
It follows that $\bR(\cE_{a_n})$ belongs to the compact set $\{F\subset \Omega:\cI_0(F)\leq \lambda/\delta_G\}$, up to extracting a subsequence again, we can assume that $\lim_{n\rightarrow\infty}\dis(\bR(\cE_{a_n}),E)=0$ for some $E\subset \Omega$.
Since the function $\cI_0$ is lower semi continuous,
it follows that $$\cI_0(E)\leq \liminf_{n\rightarrow \infty}\cI_0(\bR(\cE_{a_n}))\leq   \lambda/\delta_G <\cI_0(F_0)\,.$$
This contradicts the minimality of $F_0$. The result follows.

%
%

\end{proof}

\bibliographystyle{plain}

\end{document}